\documentclass[11pt,reqno]{amsart}

\usepackage{amsmath, amsthm, amssymb}
\usepackage{amsfonts}
\usepackage{mathtools}

\usepackage[ansinew]{inputenc}
\usepackage[dvips]{epsfig}
\usepackage{graphicx}
\usepackage[english]{babel}

\usepackage{stackengine}
\stackMath

\usepackage{aliascnt}

\theoremstyle{plain}
\newtheorem{theorem}{Theorem}[section]

\newaliascnt{lemma}{theorem}
\newtheorem{lemma}[lemma]{Lemma}
\aliascntresetthe{lemma}

\newaliascnt{proposition}{theorem}
\newtheorem{proposition}[proposition]{Proposition}
\aliascntresetthe{proposition}

\newaliascnt{corollary}{theorem}

\aliascntresetthe{corollary}

\theoremstyle{definition}
\newaliascnt{definition}{theorem}
\newtheorem{definition}[definition]{Definition}
\aliascntresetthe{definition}

\newaliascnt{remark}{theorem}
\newtheorem{remark}[remark]{Remark}
\aliascntresetthe{remark}

\newaliascnt{example}{theorem}
\newtheorem{example}[example]{Example}
\aliascntresetthe{example}

\theoremstyle{definition}
\newtheorem*{remark*}{Remark}
\newtheorem*{assumption*}{Assumption}

\numberwithin{equation}{section}

\usepackage[backgroundcolor=white, bordercolor=blue,
linecolor=blue]{todonotes}

\usepackage{hyperref}
\usepackage{cleveref}

\usepackage{dsfont}
\usepackage{bbm}

\newcommand{\N}{\mathbb{N}}
\newcommand{\R}{\mathbb{R}}

\newcommand{\cH}{\mathcal{H}}

\newcommand{\cP}{\mathcal{P}}
\newcommand{\cQ}{\mathcal{Q}}

\newcommand{\cN}{\mathcal{N}}

\newcommand{\cK}{\mathcal{K}}
\newcommand{\cM}{\mathcal{M}}

\newcommand{\eps}{\epsilon}

\newcommand{\ddiv}{\mathrm{div}_v}

\newcommand{\bbR}{\mathbb{R}}

\newcommand{\sqarrow}[1]{%
\mathrel{\makebox[4em][c]{$\mathop{\rightsquigarrow}\limits_{(#1)}$}}%
}

\DeclareMathOperator{\supp}{supp}

\DeclareMathOperator*{\osc}{osc}

\renewcommand{\d}{\textnormal{\,d}}

\newcommand{\average}{{\mathchoice {\kern1ex\vcenter{\hrule height.4pt
width 6pt depth0pt} \kern-9.7pt} {\kern1ex\vcenter{\hrule
height.4pt width 4.3pt depth0pt} \kern-7pt} {} {} }}
\newcommand{\dashint}{\average\int}

\begin{document}
\allowdisplaybreaks
\title{Kinetic Fokker-Planck equations with Maxwell boundary conditions}

\author{Kyeongbae Kim}
\address{Institute for Applied Mathematics, University of Bonn, Endenicher Allee 60, 53115, Bonn, Germany}
\email{kkim@uni-bonn.de}
\urladdr{https://sites.google.com/view/kyeongbaekim/}

\author{Marvin Weidner}
\address{Institute for Applied Mathematics, University of Bonn, Endenicher Allee 60, 53115, Bonn, Germany}
\email{mweidner@uni-bonn.de}
\urladdr{https://sites.google.com/view/marvinweidner/}



\allowdisplaybreaks

\begin{abstract}
We develop the boundary regularity theory for solutions to linear kinetic Fokker-Planck equations with Maxwell boundary conditions. These conditions interpolate between diffuse and specular reflection via an accommodation coefficient $\alpha \in [0,1]$. While existing literature is restricted to the extreme cases $\alpha = 0$ and $\alpha = 1$, we resolve the entire intermediate regime $\alpha \in (0,1)$. Specifically, we show that solutions are H\"older continuous if the coefficients are merely uniformly elliptic. Furthermore, for sufficiently smooth coefficients, we establish  boundary regularity of order $\frac{3}{\pi} \arccos(\frac{\alpha}{2}) - 1$ up to the grazing set. This exponent is optimal. Beyond Maxwell conditions, we develop a unified approach that extends to a broad class of reflection boundary conditions, including super-elastic collisions.
\end{abstract}

\maketitle

\section{Introduction}  

In this article, we establish an optimal regularity theory for solutions $f$ to linear kinetic equations
\begin{align}
\label{eq:PDE}
\partial_tf+v\cdot\nabla_xf-\ddiv(A\nabla_vf)=B\cdot\nabla_vf+F\quad \mbox{in  $(0,T)\times \Omega\times \bbR^n$}
\end{align}
in a bounded domain $\Omega \subset \R^n$, subject to Maxwell-type  reflection boundary conditions. \eqref{eq:PDE} is a fundamental model in kinetic theory, arising naturally as the linearization of the Landau equation from plasma physics. The study of such linear kinetic models is of central importance for advancing the theory of nonlinear equations, as sharp regularity estimates are the central mechanism required to establish conditional regularity and continuation criteria for the Landau and non-cutoff Boltzmann equations (see \cite{HeSn20,ImSi22,HST25}). To date, these results are restricted to $x$-periodic solutions and extending them to equations in bounded domains remains an interesting open problem \cite{Sil23}.

In real-world applications, the accuracy of these kinetic models depends heavily on how the reflection of particles at the spatial boundary $\partial \Omega$ is encoded via appropriate boundary conditions \cite{Cer00,Vil02}. Recent advances have successfully characterized the boundary regularity of linear kinetic equations (see \cite{Sil22,Zhu24}, \cite{DGY22,RoWe25}, \cite{HJJ15,HLW24,Zhu25,KiWe26}) for physically idealized scenarios, namely diffuse or specular reflection, which model boundaries that are perfectly thermalizing or perfectly smooth, respectively. In practice, however, such absolute surface properties are rarely observed. Instead, physical boundaries typically exhibit \emph{a combination of both}. This motivates the study of \emph{Maxwell boundary conditions}, which linearly interpolate between these extreme cases \cite{Max79,Mis10}.

We formalize this by splitting the boundary $\gamma = (0,T) \times \partial \Omega \times \R^n$ into the following three parts
\begin{align*}
    \gamma_{\pm} = (0,T) \times \{(x,v) \in \partial \Omega \times \R^n : \pm n_x \cdot v > 0 \}, \quad \gamma_0 = (0,T) \times \{(x,v) \in \partial \Omega \times \R^n : n_x \cdot v = 0\},
\end{align*}
where $n_x \in \mathbb{S}^{n-1}$ denotes the outer unit normal vector at $\partial \Omega$, and $\gamma_+$ and $\gamma_-$ model the outgoing and incoming particles, respectively, and $\gamma_0$ denotes the so called grazing set. 

To capture how particles are reflected back into the domain, and therefore balance out $f\big\rvert_{\gamma_-}$ and $f\big\rvert_{\gamma_+}$, Maxwell proposed in 1879 a phenomenological law of splitting the reflection operator into a local and a nonlocal part \cite{Max79}. These are mixed via an accommodation coefficient $\alpha \in [0,1]$:
\begin{align}
\label{eq:bdry}
    f=\alpha \mathcal{R}f+(1-\alpha)\mathcal{N}f \quad  \mbox{in $\gamma_-$}.
\end{align}
This so called ``Maxwell'' boundary condition was the only reflection law for gas-surface interactions that appeared in the literature before the 1960s \cite{Mis10}.

Here, the local reflection operator $\mathcal{R}$ is defined as the specular reflection of particles
\begin{align*}
    \mathcal{R} f(t,x,v) := f(t,x,R_x(v)) \quad \text{in } \gamma_-, \qquad \text{where } R_x(v) = v - 2 (n_x \cdot v) n_x,
\end{align*}
which describes the natural behavior of particles bouncing elastically off a perfectly smooth wall. 

The nonlocal reflection operator $\cN$ is given by
\begin{align*}
    \cN f(t,x,v):= \cM(t,x,v) \int_{\R^n} f(t,x,w) (w \cdot n_x)_- \d w,
\end{align*}
where $\cM$ is a smooth function with  $\int_{\bbR^n}\mathcal{M}(t,x,w)(w_n)_+\,dw=1 $. A natural choice for $\cM$ is given by a boundary Maxwellian, such as $\cM(v) = c_0 e^{-|v|^2}$ for a suitable $c_0 > 0$. The operator $\cN$ accounts for the microscopic roughness of the boundary. Particles interact with the boundary and are re-emitted according to a velocity distribution determined by the boundary temperature. 

There are several important results on existence, stability, long-time behavior, and diffusive limits of solutions to linear kinetic equations \eqref{eq:PDE} subject to the Maxwell boundary condition \eqref{eq:bdry} and also for nonlinear models such as the Landau and non-cutoff Boltzmann equation (see \cite{Mis10,GJZ18,BCMT23,Den23,OuSi24,Zhu24b,ImLo25}). However, the boundary behavior of solutions is far from understood, and even the continuity of solutions has remained open for $\alpha \in (0,1)$. Already in the simplest case of Kolmogorov's equation
\begin{align}
\label{eq:Kolm}
    \partial_t f + v \cdot \nabla_x f - \Delta_v f = F \quad \text{in } (0,T) \times \Omega \times \R^n,
\end{align}
the existing boundary regularity theory is strictly limited to the two idealized endpoints \footnote{All results in \cite{RoWe25} and \cite{KiWe26} also hold for solutions to \eqref{eq:PDE} with sufficiently regular coefficients}:
\begin{itemize}
    \item $\alpha = 0$ (diffuse reflection): solutions are globally $C^{1/2}$ and $C^{\infty}$ away from $\gamma_0$ (see \cite{KiWe26}).
    \item $\alpha = 1$ (specular reflection):  solutions are globally $C^{4,1}$ and $C^{\infty}$ away from $\gamma_0$ (see \cite{RoWe25}).
\end{itemize}

The goal of this article is to establish the regularity theory for weak solutions to \eqref{eq:PDE} subject to the Maxwell boundary condition \eqref{eq:bdry} for $\alpha \in (0,1)$. In particular, we will precisely answer the following natural question:
\begin{align*}
    \text{How does the optimal regularity of solutions interpolate between $C^{1/2}$ and $C^{4,1}$ for $\alpha \in (0,1)~$?}
\end{align*}
As we explain below, the case $\alpha \in (0,1)$ behaves fundamentally differently from $\alpha = 0$ and $\alpha = 1$, and its treatment requires substantial new ideas. In this article, we develop a unified approach that extends to a broad class of reflection-type boundary conditions, encompassing, in particular, super-elastic boundary conditions for which we also obtain optimal boundary regularity estimates.

\subsection{Maxwell boundary condition}

Our first main result establishes the optimal regularity of solutions to \eqref{eq:PDE} subject to the Maxwell boundary condition \eqref{eq:bdry} with $\alpha \in (0,1)$. We first state a simplified version of our result.

\pagebreak

\begin{theorem}
\label{thm:main-Maxwell}
    Let $\Omega \subset \R^n$ be a bounded smooth domain and let $A,B,F \in C^{\infty}((0,T) \times \overline{\Omega} \times \R^n)$ and $A$ be uniformly elliptic. Let $\cM \in C^{\infty}(\gamma_-)$. Let $f$ be a weak solution to \eqref{eq:PDE} with Maxwell boundary condition \eqref{eq:bdry} for some $\alpha \in (0,1)$ and assume that $f,F$, and $\cM$ have fast decay as $|v| \to \infty$. Then, for any $[t_0,t_1] \subset (0,T)$, it holds
\begin{itemize}
    \item[(i)] $f$ is smooth away from the grazing set, i.e. $f \in C^{\infty}(([t_0,t_1] \times \overline{\Omega} \times \R^n) \setminus \gamma_0)$.\\
    
    \item[(ii)] $f \in C^{\lambda_{\alpha}}([t_0,t_1] \times \overline{\Omega} \times \R^n)$, where $\lambda_{\alpha} = \frac{3}{\pi}\arccos(\frac{\alpha}{2})-1$, and we have
    \begin{align*}
        \|f\|_{C^{\lambda_{\alpha}}([t_0,t_1]\times\Omega\times \bbR^n)}\leq c \left( \Vert (1 + |v|)^q f \Vert_{L^{\infty}((0,T) \times \Omega ; L^1(\R^n))} +\|(1 + |v|)^qF\|_{L^\infty((0,T)\times\Omega\times \bbR^n)}\right)
    \end{align*}
    where $c,q > 0$ depend only on $n,A,B,\cM,t_0,t_1,\Omega$, and $\alpha$.  
\end{itemize}
\end{theorem}

To the best of our knowledge, this is the first regularity result for solutions to kinetic equations with Maxwell boundary conditions. The $C^{\lambda_{\alpha}}$ regularity is sharp in the sense that we find an explicit solution $f$ with $F,\cM \in C^{\infty}$ such that $f \in C^{\lambda_{\alpha}}$ but $f \not\in C^{\lambda_{\alpha} + \eps}$ for any $\eps > 0$ and $\alpha \in (0,1)$ at the grazing set $\gamma_0$ (see \autoref{ex:opt}). Moreover, it seems that the explicit regularity exponent $\lambda_{\alpha}$ has not previously appeared in the literature.

Note that $\lambda_{\alpha} \nearrow \frac{1}{2}$ as $\alpha \searrow 0$, and therefore \autoref{thm:main-Maxwell} recovers the optimal $C^{\frac{1}{2}}$ regularity of solutions with diffuse reflection boundary conditions established in \cite{KiWe26}. In striking contrast, as $\alpha \nearrow 1$, we have $\lambda_{\alpha} \searrow 0$. Consequently, we do not recover the $C^{4,1}$ regularity of solutions with specular reflection boundary conditions from \cite{RoWe25}. This reveals that the higher regularity in the specular case is a highly rigid feature, relying entirely on symmetry properties of solutions that immediately disappear for $\alpha < 1$. Hence, our proof of \autoref{thm:main-Maxwell} for the intermediate regime $\alpha \in (0,1)$ requires  distinct mathematical techniques (see the discussion in Subsection \ref{subsec:strategy}).

Remarkably, our main result is not restricted to equations with smooth coefficients. 
\begin{itemize}
    \item We show that the $C^{\lambda_{\alpha}}$ regularity of solutions in \autoref{thm:main-Maxwell}(ii) remains true when $\partial \Omega \in C^{2,\eps}$, $A \in C^{\eps}$, $\cM \in C^{\lambda_{\alpha}+\eps}$ for some $\eps > 0$, and $B,F \in L^{\infty}$ (see \autoref{thm:Maxwell-body}).
    \item  Under suitable regularity and decay assumptions on $A,B,F,\cM$, we prove Schauder-type estimates of order $k+\eps$ away from $\gamma_0$ in localized cylinders  (see \autoref{thm.maxhalf})
\end{itemize}

Most importantly, we also develop a De Giorgi-Nash-Moser-type regularity theory at the boundary for equations with merely measurable coefficients by establishing H\"older regularity of weak solutions in this setting. To the best of our knowledge, also this result is entirely new in the context of Maxwell boundary conditions and was only known in the endpoint cases $\alpha = 0$ (see \cite{Zhu24}) and $\alpha = 1$ (see \cite{Sil22,Zhu24}).

\begin{theorem}
\label{thm:main-Holder}
    Let $\Omega \subset \R^n$ be a $C^{2,\eps}$ domain for some $\eps > 0$ and let $B,F \in L^{\infty}((0,T) \times \Omega \times \R^n)$ and $\Lambda^{-1} I \le A \le \Lambda I$.  Let $\cM \in C^{\eps}(\gamma_-)$. Let $f$ be a weak solution to \eqref{eq:PDE} with Maxwell boundary condition \eqref{eq:bdry} for some $\alpha \in (0,1)$ and assume that $f,F$, and $\cM$ have fast decay as $|v| \to \infty$. Then, for any $[t_0,t_1] \subset (0,T)$, it holds
    \begin{align*}
        \|f\|_{C^{\beta}([t_0,t_1]\times\Omega\times \bbR^n)}\leq c \left( \Vert (1 + |v|)^q f \Vert_{L^{\infty}((0,T) \times \Omega ; L^1(\R^n))} +\|(1 + |v|)^qF\|_{L^\infty((0,T)\times\Omega\times \bbR^n)}\right)
    \end{align*}
    where $c,q > 0$ depend only on $n,\Lambda,B,\cM,t_0,t_1,\Omega, \alpha,\eps$, and $\beta \in (0,1)$ depends only on $n,\Lambda,B,\cM,\Omega, \alpha,\eps$.  
\end{theorem}

Note that all constants in \autoref{thm:main-Holder} are uniform as $\alpha \searrow 0$ and explode as $\alpha \nearrow 1$. This is in align with the sharpness of the $C^{\lambda_{\alpha}}$ regularity from \autoref{thm:main-Maxwell}.

It is an interesting task to establish analogous regularity results for generalized models of the Maxwell condition \cite{Cer00,Mis10}, where the accommodation coefficient is allowed to be flux-dependent, or the coefficient $\cM$ in the reflection operator is replaced by a kernel $\cK(v,w)$. In Subsection \ref{subsec:ext} we discuss such corresponding extensions.

\subsection{General reflection-type boundary conditions}

Beyond the Maxwell condition \eqref{eq:bdry}, our proofs of \autoref{thm:main-Maxwell} and \autoref{thm:main-Holder} yield a unified framework that naturally extends to the following much broader class of reflection-type boundary conditions:
\begin{align}
\label{eq:bdry-gen}
    f = \alpha \mathcal{R}_b f + (1-\alpha)g \quad\text{in } \gamma_-,
\end{align}
where $\alpha \in (0,1)$, $g : \gamma_- \to \R$, and we define for some $b \in (0,1]$ 
\begin{align*}
    \mathcal{R}_b f(t,x,v) = f\big(t,x,v - (1+b) (n_x \cdot v)n_x \big).
\end{align*}
Condition \eqref{eq:bdry-gen} generalizes \eqref{eq:bdry} by replacing the nonlocal reflection operator with a prescribed in-flow datum $g$, and modifying the local reflection by introducing $\mathcal{R}_b$, which coincides with $\mathcal{R}$ when $b=1$. 

In  case $b = 1$, \eqref{eq:bdry-gen} is also known as \emph{$\alpha$-partial specular reflection condition} and it models systems where only a portion of particles undergoes specular reflection, while a fraction $(1-\alpha)$ of them are absorbed by the wall. Well-posedness of linear kinetic equations with such boundary conditions was shown in \cite{Zhu24}, but to date, no regularity results are available.

For $b \in (0,1)$, the operator $\mathcal{R}_b$ models \emph{super-elastic collisions}, where incoming particles (i.e.~after hitting the boundary) move faster by a factor of $b^{-1}$ and have a smaller angle with $-n_x$.\footnote{Indeed, the normal speed of the pre-collisional (outgoing) velocity is given by $[v - (1+b) (v \cdot n_x) n_x ] \cdot n_x = -b (n_x \cdot v)$.} This models active (heating) boundaries which inject energy into the system. Such reflection laws also occur in the study of granular gases  \cite{BTF01,BrPo04}.

A particularly relevant scenario occurs when $\alpha = b^2$ and $g = 0$, in which case the system conserves mass and \eqref{eq:PDE}--\eqref{eq:bdry-gen} becomes exactly the adjoint of the inelastic problem studied in \cite{HJV19,HJV19b}. In the inelastic regime $b > 1$, where particles slow down after hitting the boundary, regularity completely breaks down at $\gamma_0$. For large enough $\beta$, this even leads to non-uniqueness of solutions due to inelastic collapse, as demonstrated in \cite{HJV19}. By contrast, in this article, we operate entirely in the super-elastic regime $b \in (0,1]$, where \eqref{eq:bdry-gen} interpolates between absorbing ($b=0$) and specular reflection ($b = 1$) boundary conditions when $\alpha = b^2$. 

Our next main result establishes optimal regularity results for solutions to \eqref{eq:PDE} subject to the boundary condition \eqref{eq:bdry-gen} in the regime $\alpha/b^2\in(0,1]$. As we discuss in Subsection \ref{subsec:strategy}, this is the precise range in which the boundary remains dissipative in $L^1$. The optimal regularity exponent is given by
\begin{align*}
    \lambda_{\alpha,b}\coloneqq \inf\left\{ \lambda>0\,:\,2\cos\left(\pi\left(\tfrac{\lambda}{3}+\tfrac13\right)\right)=\alpha b^{\lambda}\right\} \in [0,\tfrac{1}{2}].
\end{align*}

To the best of our knowledge, this is the first systematic regularity result for kinetic equations with super-elastic boundary conditions and our results are new for all values of $\alpha \in (0,1)$, $b \in (0,1]$.

\begin{theorem}
\label{thm:main-general}
    Let $\Omega \subset \R^n$ be a bounded domain for some $\eps > 0$ and $f$ be a weak solution to \eqref{eq:PDE} with boundary condition \eqref{eq:bdry-gen} for some $\alpha \in (0,1)$, $b \in (0,1]$, and $\alpha/b^2 \in (0,1]$. Moreover, assume that $A$ is uniformly elliptic and that $f,g$, and $F$ have fast decay as $|v| \to \infty$. \\
    Then, for any $[t_0,t_1] \subset (0,T)$, the following hold true:
\begin{itemize}
    \item[(i)] If $A,B,\Omega,F,g$ are smooth, then $f$ is smooth away from the grazing set, i.e. $f \in C^{\infty}(([t_0,t_1] \times \overline{\Omega} \times \R^n) \setminus \gamma_0)$.\\
    
    \item[(ii)] If $A \in C^{\eps}$, $B,F \in L^{\infty}$, $g \in C^{\lambda_{\alpha,b} + \eps}$, and $\Omega$ is a $C^{2,\eps}$ domain for some $\eps > 0$, then $f \in C^{\lambda_{\alpha,b}}([t_0,t_1] \times \overline{\Omega} \times \R^n)$, and we have
    \begin{align*}
        \|f\|_{C^{\lambda_{\alpha,b}}([t_0,t_1]\times\Omega\times \bbR^n)}\leq c \big( \Vert (1 + |v|)^q f \Vert_{L^{1}((0,T) \times \Omega \times \R^n)} + [g]_{C^{\lambda_{\alpha,b} + \eps}_q(\gamma_-)} +\|(1 + |v|)^qF\|_{L^\infty((0,T)\times\Omega\times \bbR^n)}\big),
    \end{align*}
    where $c,q > 0$ depend only on $n,A,B,t_0,t_1,\Omega,\eps,b$, and $\alpha$. \footnote{We write $[g]_{C^{\beta}_q(\gamma_-)} = \sup_{z_0 \in \gamma_-} (1 + |v_0|)^q [g]_{C^{\beta}(\cQ_1(z_0) \cap \gamma_-)}$} \\

    \item[(iii)] If $A,B,F \in L^{\infty}$, $g \in C^{\eps}$, and $\Omega$ is a $C^{2,\eps}$ domain for some $\eps > 0$, then $f \in C^{\beta}([t_0,t_1] \times \overline{\Omega} \times \R^n)$, and we have
    \begin{align*}
        \|f\|_{C^{\beta}([t_0,t_1]\times\Omega\times \bbR^n)}\leq c \big( \Vert (1 + |v|)^q f \Vert_{L^{1}((0,T) \times \Omega \times \R^n)} + [g]_{C^{\beta}_q(\gamma_-)} +\|(1 + |v|)^qF\|_{L^\infty((0,T)\times\Omega\times \bbR^n)}\big),
    \end{align*}
    where $c,q > 0$ depend only on $n,A,B,t_0,t_1,\Omega,\eps,b,\alpha$, and $\beta \in (0,1)$ depends only on $n,A,B,\Omega,\eps,b, \alpha$.
\end{itemize}
\end{theorem}

The $C^{\lambda_{\alpha,b}}$ regularity of solutions in (ii) is sharp as we show by constructing explicit solutions in the half-line with $F \equiv g \equiv 0$ in \autoref{lem.phialpha}. We also refer to \autoref{thm.inel} for a version of this result in the half-space, where we make the dependencies of the constants $c,q,\beta$ more precise.

Moreover, we have $\lambda_{\alpha,1} = \lambda_{\alpha}$, i.e. \autoref{thm:main-general} yields regularity for solutions with $\alpha$-partial specular reflection condition and demonstrates that the specific choice of $g = \cN f$ for the Maxwell condition does not influence the optimal boundary regularity of solutions at the grazing set (see \autoref{thm:main-Maxwell}(ii)). In fact, we prove \autoref{thm:main-Maxwell} and \autoref{thm:main-Holder} by setting $g = \cN f$ in \autoref{thm:main-general} after showing that $\cN f$ is sufficiently regular in $\gamma_-$. The latter is a non-trivial task, since $\cN f$ depends on $f\big\vert_{\gamma_+}$ in a nonlocal way.

Note that although the regularity exponent $\lambda_{\alpha,b}$ is implicit, one can easily show that in the mass conserving case $\alpha = b^2$, it holds $\lambda_{b^2,b} \nearrow \tfrac{1}{2}$ as $b \searrow 0$, and $\lambda_{b^2,b} \searrow 0$ as $b \nearrow 1$. This reveals that also under super-elastic boundary conditions, regularity degenerates close to the specular regime.

\subsection{Strategy of the proof}
\label{subsec:strategy}

Our main results establish a De Giorgi-Nash-Moser-type theory up to the boundary and also yield optimal higher regularity for equations with H\"older continuous coefficients. Because the boundary conditions \eqref{eq:bdry-gen} and \eqref{eq:bdry} simultaneously encode specular (or super-elastic) reflection and prescribed in-flow (or diffuse reflection), existing techniques developed for the endpoint cases $\alpha = 0$ and $\alpha = 1$ break down. Instead, the coupling of the two boundary conditions requires several novel techniques. 

In the sequel, we explain the main ideas of our proof of \autoref{thm:main-general} in the following simplified setting
    \begin{equation}
    \label{eq:PDE-half-intro}
\left\{
\begin{alignedat}{3}
\partial_tf+v\cdot\nabla_xf-\Delta_v f&=F&&\qquad \mbox{in  $\{ x_n > 0\} \cap \cQ_1 =: \mathcal{H}_1$}, \\
f&=\alpha \mathcal{R}_bf&&\qquad  \mbox{on $(\{ x_n = 0\} \times \{ v_n < 0 \}) \cap \cQ_1 = \gamma_-\cap \mathcal{Q}_1$}.
\end{alignedat} \right.
\end{equation}
Incorporating a nonlocal reflection term $(1-\alpha) \cN f$ or a suitably regular in-flow datum $(1-\alpha) g$ into the boundary condition introduces additional technical difficulties, which are essentially perturbative.

\subsubsection{Local boundedness and H\" older regularity}

For prescribed in-flow or pure diffuse reflection, the strategy to obtain H\" older regularity is to extend the solution by a constant (matching the maximum of the boundary datum) across the boundary to apply interior growth lemmas \cite{Sil22,Zhu24}. Conversely, for pure specular reflection, one can extend the solution via mirror reflection, which entirely reduces the problem to the interior setting \cite{Sil22,Zhu24}.

Both of the aforementioned strategies immediately fail for \eqref{eq:bdry-gen} and \eqref{eq:bdry}. Mirror reflection leads to discontinuities at the grazing set when $\alpha \not= 1$. Meanwhile, extending the solution by its maximum boundary value is impossible without first establishing $L^{\infty}$ bounds for the boundary trace $f\big\rvert_{\gamma_-}$, which is non-trivial since the boundary data depends on $f\big\rvert_{\gamma_+}$ itself through the $\alpha$-reflection condition.

To overcome these obstacles, we introduce a general mechanism to prove a gain of integrability in all variables for linear kinetic equations up to the boundary in $\cH_1$. This allows us to carry out a De Giorgi iteration scheme for $f$ to deduce local boundedness (see \autoref{lem.bdd}). Once this is established, we can recast \eqref{eq:PDE-half-intro} as an in-flow problem and use the results of \cite{Sil22} to prove oscillation decay at boundary points. The oscillation decay of the boundary data itself follows from the fact that $\alpha < 1$ and $b \in (0,1]$, which ultimately yields H\"older regularity for solutions to \eqref{eq:PDE-half-intro} (see \autoref{lem.holsmall}). 

As is common in kinetic equations, the main difficulty is transferring information from the diffusion variable $v$ to the $x$-variable. We find that for any solution $h$ to 
\begin{align*}
    \partial_th+v\cdot\nabla_xh-\Delta_v h&=F \quad \text{in } \R \times \{ x_n > 0 \} \times \R^n
\end{align*}
it holds that $\widetilde{h} = h \chi_{\{x_n > 0 \}}$ solves,
\begin{align*}
    \partial_t\widetilde{h}+v\cdot\nabla_x\widetilde{h}-\Delta_v \widetilde{h}&= F\chi_{\{x_n > 0 \}} + \mu \quad \text{in } \R^{1+2n},
\end{align*}
where $\mu$ is a measure supported on $\{ x_n = 0 \}$ such that for some $s > \frac{1}{2}$ it holds
\begin{align*}
    \Vert \mu \Vert_{L^2_{t,x',v}(\R^{2n} ; H^{-s}_{x_n}(\R))}^2 \lesssim \int_{\{ x_n = 0 \}} |h(t,x,v) v_n|^2 \d \Gamma.
\end{align*}
By adapting classical transfer of regularity results in the full space $\R^{1+2n}$  (see \cite{Bou02,PaPo04,ImSi20}) to source terms in $H^{-s}_x$ with $s > \frac{1}{2}$ and employing a suitable localization argument, we can deduce that  weak subsolutions $f \ge 0$ to \eqref{eq:PDE-half-intro}$_{(i)}$ satisfy for some $q > 2$
\begin{align}
\label{eq:gain-of-int-intro}
    \Vert f \Vert_{L^q(\cH_{1/2})}^2 \lesssim \Vert F \Vert_{L^2(\cH_1)}^2 + \Vert f \Vert_{L^2_{t,x}H^1_v(\cH_1)}^2 + \int_{\{ x_n = 0 \} \cap \cQ_1} |f(t,x,v) v_n|^2 \d \Gamma .
\end{align}
Note that an analogous result holds for equations with divergence-type source terms in $L^2$ (see \autoref{lem.extension}) and that \eqref{eq:gain-of-int-intro} is finite for all $f \in H^1_{\text{kin}}(\cQ_1)$, independent of any boundary condition.

The boundary condition \eqref{eq:PDE-half-intro}$_{(ii)}$ is crucial for estimating the right-hand side of \eqref{eq:gain-of-int-intro}, which is achieved by a kinetic Caccioppoli-type inequality at the boundary (see \autoref{lem.extension1}). The main observation is that by the boundary condition we have
\begin{align}
    \label{eq:Cacc-intro}
    \int_{\{x_n = 0\} \cap \cQ_1} f(t,x,v) (-v_n) \d \Gamma \ge \left(1 - \frac{\alpha}{b^2} \right) \int_{\{x_n = 0\} \cap \cQ_1} f(t,x,v) |v_n| \d \Gamma \ge 0,
\end{align}
i.e., the boundary flux is dissipative in $L^1$ when $\frac{\alpha}{b^2} \le 1$. This allows to estimate $[f]_{H^1_v}^2$ from above, independently of any boundary terms, showing that
\begin{align}
\label{eq:Cacc-intro-2}
    \Vert \nabla_v f \Vert_{L^2(\mathcal{H}_{1/2})}^2 + \int_{\{ x_n = 0 \} \cap \cQ_{1/2}} |f v_n|^2 \d \Gamma \lesssim \Vert F \Vert_{L^2(\cH_1)}^2 + \Vert f \Vert_{L^2(\cH_1)}^2.
\end{align}
Special care is required in the proof of \eqref{eq:Cacc-intro-2} since it is still an open problem to determine whether $(f^2 |v_n|)\big\rvert_{\gamma} \in L^1_{\text{loc}}$ (see \cite{AAMN24,Sil22}). Therefore it is not justified to test \eqref{eq:PDE-half-intro} with $f$ itself.\footnote{If $(f^2 |v_n|)\big\rvert_{\gamma} \in L^1_{\text{loc}}$ is true, then dissipativity of the boundary flux in $L^2$ would be sufficient to establish \autoref{lem.extension1} and thus also for \autoref{thm:main-general}. $L^2$ dissipativity holds true already under the weaker condition $\frac{\alpha}{b} \le 1$.}

By combining \eqref{eq:gain-of-int-intro} with \eqref{eq:Cacc-intro-2}, we can set up a De Giorgi iteration scheme and deduce local boundedness estimates for weak solutions to \eqref{eq:PDE-half-intro} (see \autoref{lem.bdd}). 

\begin{remark}
    Our technique has very broad applicability and also yields local boundedness for kinetic equations with bounce back conditions
\begin{align}
\label{eq:bounce-back}
    f(t,x,v) = f(t,x,-v) \quad\text{on } \gamma_- \cap \cQ_1,
\end{align}
as we demonstrate in Subsection \ref{subsec:bounce-back} (see \autoref{thm.bounce}). It is an interesting problem to establish an optimal regularity theory for solutions to \eqref{eq:PDE}--\eqref{eq:bdry} or \eqref{eq:PDE}--\eqref{eq:bdry-gen} when replacing the specular reflection operator $\mathcal{R}$ by \eqref{eq:bounce-back} for any $\alpha \in [0,1]$ and $b = 1$. We believe that the techniques and ideas in Subsection \ref{subsec:bounce-back} might be helpful in proving such results.
\end{remark}

\subsubsection{Liouville theorem and higher regularity}

To prove the optimal regularity in \autoref{thm:main-Maxwell} and \autoref{thm:main-general}(i),(ii), we follow the blow-up technique developed in \cite{RoWe25} (see also \cite{KiWe26}). Specifically, we show that any solution to \eqref{eq:PDE-half-intro} satisfies the following expansion at $0 \in \gamma_0$
\begin{align*}
    |f(t,x,v) - m \phi_{\alpha,b}(x_n,v_n)| \lesssim r^{\lambda_{\alpha,b}+\eps} \quad \forall (t,x,v) \in \mathcal{H}_r
\end{align*}
for some $\eps > 0$ and $m \in \R$, where $\phi_{\alpha,b}$ is the explicit $\lambda_{\alpha,b}$-homogeneous solution given in \eqref{phialpha:defn.phialpha}.

The main difficulty for general reflection boundary conditions lies in the classification of blow-up limits. The previously established H\"older regularity of solutions is crucial to show that all blow-up limits are 1D stationary solutions to the following problem
    \begin{equation}
    \label{eq:PDE-1D-intro}
\left\{
\begin{alignedat}{3}
v\partial_xf_0-\partial_{vv} f_0&=0&&\qquad \mbox{in  $\{ x > 0\} \times \R$}, \\
f_0&=\alpha \mathcal{R}_bf_0&&\qquad  \mbox{on $\{0\} \times \{ v < 0 \}$},
\end{alignedat} \right.
\end{equation}
and satisfy suitable polynomial growth at infinity. However, new arguments are needed in order to show that all solutions to \eqref{eq:PDE-1D-intro} are of the form $f_0 = m \phi_{\alpha,b}$ (see \autoref{prop.liou1}). 

We achieve this classification by establishing a boundary Harnack-type estimate for solutions to \eqref{eq:PDE-1D-intro} at $0 \in \gamma_0$, similar to \cite{KiWe26,KiWe26b}. However, due to the $\alpha$-reflection condition, the explicit 1D solution $\phi_{\alpha,b}$ is incompatible with the barrier arguments from \cite{KiWe26}, which are used to transfer comparability of two solutions from the interior to the boundary $\{x = 0\}$. Instead, we establish a novel maximum principle (see \autoref{lem.maximum}) for solutions to \eqref{eq:PDE-1D-intro} in $\cH_1$, subject to Dirichlet conditions on the remaining kinetic part of the boundary
\begin{align*}
    \big((0,1) \times \{ -1\}\big) \cup \big((0,1) \times \{1\}\big) \cup \big(\{1\} \times (-1,0)\big).
\end{align*}
Furthermore, we introduce a modified mirror-type reflection adapted to the operator $\mathcal{R}_b$ (see \autoref{lem.mirroext}). This allow us to transfer information from a neighborhood of $(0,-1) \in \gamma_+$ to $(0,1) \in \gamma_-$ while staying away from the grazing set. Together, these ingredients yield the 1D Liouville theorem (see \autoref{prop.liou1}), classifying all solutions to \eqref{eq:PDE-1D-intro} and thereby all possible blow-up limits.

\subsection{Outline}

This article is structured as follows. In Section \ref{sec:prelim} we introduce the necessary function spaces and weak solutions concepts required for this work. Section \ref{sec:transfer} is dedicated to proving transfer of regularity results for kinetic equations in the half-space. In Section \ref{sec:4} we show H\" older regularity estimates for equations with bounded measurable coefficients. In Section \ref{sec:Liouville}, we establish Liouville theorems in the half-space and in Section \ref{sec:opt} we prove optimal regularity results for solutions subject to the boundary condition \eqref{eq:bdry-gen}. Finally, in Section \ref{sec:main}, we prove our main results \autoref{thm:main-Maxwell}, \autoref{thm:main-Holder}, and \autoref{thm:main-general}.

\subsection{Acknowledgments}

Kyeongbae Kim and Marvin Weidner were supported by the Deutsche Forschungsgemeinschaft (DFG,
German Research Foundation) under Germany's Excellence Strategy - EXC-2047/1 - 390685813 and through the CRC 1720 ``Analysis of criticality: from complex phenomena to models and estimates'', 53930965.

\section{Preliminaries}
\label{sec:prelim}

In this section, we collect some important notation that will be used throughout the article. Moreover, we introduce kinetic H\"older spaces, and recall several regularity results for linear kinetic Fokker-Planck equations.

First, we collect some notation.
\begin{itemize}
\item For any $z=(t,x,v)\in \bbR^{1+2n}$ and $R>0$, we write $S_Rz\coloneqq (R^2t,R^3x,Rv)$. For $z_1=(t_1,x_1,v_1)\in \bbR^{1+2n}$, we define
\begin{align*}
    z_1\circ z=(t_1+t,x_1+x+tv_1,v+v_1).
\end{align*}
In particular, $z^{-1}\coloneqq (-t,-x+tv,-v)$.
\item Given $f$ and $g$, define the convolution of $f$ and $g$ by
    \begin{align*}
        f*g(z)\coloneqq \int_{\bbR^{1+2n}}f(w)g(w^{-1}\circ z)\,dw=\int_{\bbR^n}f(z\circ w^{-1})g(w)\,dw.
    \end{align*}
    \item $x_-\coloneqq -\min\{x,0\}$ and $x_+\coloneqq \max\{x,0\}$.
    \item We write $\,d\Gamma\coloneqq \,dt\,dx'\,dv$.
    \item Let us write $X=(x,v)\in \bbR^{2n}$. Given $X\in \bbR^{2n}$, we write $X'=(x',v')\in\bbR^{n-1}\times \bbR^{n-1}$, where $X=(x',x_n,v',v_n)$.
\item ${Q}_r\coloneqq B_{r^3}\times B_r\quad\text{and}\quad {H}_r\coloneqq B_{r^3}\times B_r\cap \{x_n>0\}$.
\item $ \mathcal{Q}_r\coloneqq  I_{r}\times B_{r^3}\times B_r\quad\text{and}\quad \mathcal{H}_r\coloneqq \mathcal{Q}_r\cap\big(\R \times \{x_n>0\} \times \R^n)$ with $I_{r}\coloneqq (-r^2,r^2)$. In addition, for any $z_0\in \bbR^{1+2n}$, we write $\mathcal{Q}_r(z_0)\coloneqq z_0\circ \mathcal{Q}_r$ and $\mathcal{H}_r(z_0)\coloneqq z_0\circ \mathcal{H}_r$.
\item For any $v\in\bbR^n$, we write $\langle v\rangle\coloneqq (1+|v|)$.
\item Let $(a,b)\times \Omega\times V\subset \bbR\times \bbR^n\times \bbR^n$, where $\Omega$ and $V$ are bounded subsets of $\bbR^n$. We write the kinetic boundary of $(a,b)\times \Omega\times V$ as 
\begin{align*}
    \partial_{\mathrm{kin}}((a,b)\times \Omega\times V)&\coloneqq (\{t=a\}\times \Omega\times V)\cup ((a,b)\times\Omega\times \partial V)\\
    &\quad\quad\cup \{(t,x,v)\in (a,b)\times\partial\Omega\times V\,:\, n_x\cdot v<0\},
\end{align*}
where $n_x$ is the outer unit normal vector to $\partial\Omega$ at $x$. In addition, we write
\begin{align*}
    &\gamma_-\coloneqq \{(t,x,v)\in (a,b)\times\partial\Omega\times V\,:\, n_x\cdot v<0\},\\
    &\gamma_+\coloneqq \{(t,x,v)\in (a,b)\times\partial\Omega\times V\,:\, n_x\cdot v>0\},\\
    &\gamma_0\coloneqq \{(t,x,v)\in (a,b)\times\partial\Omega\times V\,:\, n_x\cdot v=0\}.
\end{align*}
\end{itemize}

Throughout the paper, we assume for some $\Lambda\geq1$
\begin{align*}
    \Lambda^{-1}I\leq A\leq \Lambda I\quad\text{and}\quad |B|\leq \Lambda.
\end{align*}

We recall the kinetic Sobolev space $H^1_{\mathrm{kin}}(\mathcal{Q})$, which is given by
\begin{align*}
  H^1_{\mathrm{kin}}(\mathcal{Q})\coloneqq \left\{f\in L^2(Q)\,:\, \nabla_vf\in L^2(\mathcal{Q})\quad\text{and}\quad(\partial_t+v\cdot\nabla_x)f\in L_{t,x}^2H_v^{-1}(\mathcal{Q})\right\},
\end{align*}
where we say $f\in L_{t,x}^2H_v^{-1}(\mathcal{Q})$, if for any $\psi\in C^\infty_c(\mathcal{Q})$
\begin{align*}
    \left|\int_{\mathcal{Q}}f(\partial_t+v\cdot\nabla_x)\psi\,dz\right|\leq c\|\nabla_v\psi\|_{L^2(\mathcal{Q})}
\end{align*}
for some constant $c$. In particular, we write
\begin{align*}
    \|f\|_{H^1_{\mathrm{kin}}(\mathcal{Q})}\coloneqq \|f\|_{L^2(\mathcal{Q})}+\|\nabla_vf\|_{L^2(\mathcal{Q})}+\|(\partial_t+v\cdot\nabla_x)f\|_{L_{t,x}^2H_v^{-1}(\mathcal{Q})}.
\end{align*}
Next, we recall the kinetic distance, kinetic polynomials and kinetic H\"older spaces as in \cite{ImSi21}. 

For any $z_1,z_2\in \bbR^{1+2n}$, we write
\begin{align*}
    \mathrm{dist}_{\mathrm{kin}}(z_1,z_2)\coloneqq \min_{w\in\bbR^n}\left\{\max\{|t_1-t_2|^{\frac12},|x_1-(x_2+(t_1-t_2)w)|^{\frac13},|v_1-w|,|v_2-w|\right\}.
\end{align*}
Let $\cP_k$ be the space of kinetic polynomials of kinetic degree $k$ defined by 
\begin{align*}
    \cP_k\coloneqq \mathrm{span}\left\{\sum_{2|\beta_t|+3|\beta_x|+|\beta_v|\leq k}t^{\beta_t}x_1^{\beta_{x_1}}\cdots x_n^{\beta_{x_n}}v_1^{\beta_{v_1}}\cdots v_n^{\beta_{v_n}}\,:\, |\beta_x|=\sum_{i=1}^n\beta_{x_i},\, |\beta_v|=\sum_{i=1}^n\beta_{v_i}\right\}.
\end{align*}
For any domain $D\subset \bbR^{1+2n}$, $k\in\N\cup\{0\}$ and $\alpha\in[0,1)$, we say $f\in C^{k,\alpha}(D)$ if $\|f\|_{C^{k,\alpha}(D)}\coloneqq[f]_{C^{k,\alpha}(D)}+\|f\|_{L^\infty(D)}<\infty$, where
\begin{align*}
    [f]_{C^{k,\alpha}(D)}\coloneqq \sup_{z_1\in D}\inf_{p\in \cP_k}\sup_{z_2\in D\setminus\{z_1\}}\frac{|(f-p)(z_2)|}{\mathrm{dist}(z_1,z_2)^{k+\alpha}}.
\end{align*}

 We now introduce the reflection-type boundary condition that will be analyzed throughout the remainder of this article and define the corresponding notion of weak solutions.

\begin{definition}\label{defn.weak.sol}
  Let $z_0\in\bbR^{1+2n}$ with $x_{0,n}=v_{0,n}=0$. We say that $f$ is a weak solution to 
    \begin{equation*}
\left\{
\begin{alignedat}{3}
\partial_tf+v\cdot\nabla_xf-\ddiv(A\nabla_vf)&=B\cdot\nabla_vf+F&&\qquad \mbox{in  $\mathcal{H}_R(z_0)$}, \\
f(t,x',0,v)&=\alpha f(t,x',0,v',-bv_n)+(1-\alpha)g&&\qquad  \mbox{in $\gamma_-\cap \mathcal{Q}_R(z_0)$},
\end{alignedat} \right.
\end{equation*}
where $\alpha\in[0,1]$, $b\in(0,1]$ with $\alpha/b^2\in[0,1]$, $g\in L^\infty(\gamma_-)$ and $F\in L^2(\mathcal{Q}_R(z_0))$, if $f\in H^1_{\mathrm{kin}}(\mathcal{Q}_R(z_0))$ and for any $\psi\in C^\infty_c(\bbR^{1+2n})$ with $\mathrm{\supp}\,\psi\cap \partial \mathcal{H}_R(z_0)\subset \gamma$,
\begin{equation}\label{defn:eq.weak}
\begin{aligned}
    &-\int_{\mathcal{Q}_R(z_0)}f(\partial_t+v\cdot\nabla_x)\psi\,dz+\int_{\gamma}f\psi(n_x\cdot v)\,d\Gamma+\int_{\mathcal{Q}_R(z_0)}A\nabla_vf\cdot\nabla\psi\,dz\\
    &=\int_{\mathcal{Q}_R(z_0)}B\cdot\nabla f\psi\,dz+\int_{\mathcal{Q}_R(z_0)}F\psi\,dz
\end{aligned}
\end{equation}
with 
\begin{align}\label{defn2:eq.weak}
    \int_{\gamma_-}f\psi(n_x\cdot v)\,d\Gamma=\int_{\gamma_-}(\alpha\mathcal{R}_bf+(1-\alpha)g)\psi(n_x\cdot v)_-\,d\Gamma.
\end{align}
Here, we set $\mathcal{R}_bf(t,x',0,v',v_n)\coloneqq f(t,x',0,v',-bv_n)$.
\end{definition}
\begin{remark}\label{rmk.defn.weak}
\begin{itemize}
    \item First, note that if $f$ is a classical solution, then the weak formulation given in \autoref{defn.weak.sol} follows directly.
    \item Next, since $f\in H^1_{\mathrm{kin}}(Q)$, note from \cite[Proposition 4.2]{Sil22} that we have for any $\widetilde{\gamma}\Subset \gamma$,
\begin{align*}
    \int_{\widetilde{\gamma}}f^2\min\{|n_x\cdot v|,|n_x\cdot v|^2\}\,d\Gamma\leq c\|f\|_{H^1_{\mathrm{kin}}(Q)}^2.
\end{align*}
This together with the fact that $|n_x\cdot v|\leq c$ implies for any test function $\psi$,
\begin{align*}
    \int_{\gamma}|f\psi||n_x\cdot v|\leq \left(\int_{\gamma\cap\supp\,\psi}|f|^2|n_x\cdot v|^2\right)^{\frac12}\leq c\|f\|_{H^1_{\mathrm{kin}}(Q)},
\end{align*}
and shows that \eqref{defn:eq.weak} and \eqref{defn2:eq.weak} are well-defined.
\end{itemize}

\end{remark}

For later use, we recall two known regularity results for kinetic Fokker-Planck equations with prescribed in-flow boundary conditions. The first result is a $C^\alpha$ regularity estimate for equations with bounded measurable coefficients.

\begin{lemma}\label{lem.calpha}
    Let $z_0 \in \R^{1+2n}$ and $R \leq 1$. Let $f$ be a weak solution to 
    \begin{equation*}
\left\{
\begin{alignedat}{3}
\partial_tf+v\cdot\nabla_xf-\ddiv(A\nabla_vf)&=B\cdot\nabla_vf+F&&\qquad \mbox{in  $\mathcal{H}_R(z_0)$}, \\
f&=g&&\qquad  \mbox{in $\gamma_-\cap \mathcal{Q}_R(z_0)$}.
\end{alignedat} \right.
\end{equation*}
Then there is a small constant $\alpha=\alpha(n,\Lambda)$ such that 
\begin{align*}
    R^{\alpha}[f]_{C^\alpha(\mathcal{H}_{R/2}(z_0))}\leq c\left(R^{-(4n+2)}\|f\|_{L^1(\mathcal{H}_R(z_0))}+R^2\|F\|_{L^\infty(\mathcal{H}_R(z_0))}+R^{\alpha}[g]_{C^\alpha(\gamma_-\cap\mathcal{Q}_{R}(z_0))}\right)
\end{align*}
for some constant $c=c(n,\Lambda,R\|B\|_{L^\infty(\mathcal{H}_R(z_0))})$. In addition, we have 
\begin{align*}
    \|f\|_{L^\infty(\mathcal{H}_{R/2}(z_0))}\leq c\left(R^{-(4n+2)}\|f\|_{L^1(\mathcal{H}_R(z_0))}+R^2\|F\|_{L^\infty(\mathcal{H}_R(z_0))}+\|g\|_{L^\infty(\gamma_-\cap\mathcal{Q}_{R}(z_0))}\right)
\end{align*}
for some constant $c=c(n,\Lambda,R\|B\|_{L^\infty(\mathcal{H}_R(z_0))})$.
\end{lemma}
The proof follows directly from \cite[Lemma 2.25]{RoWe25} together with the scaling argument in \cite[Proposition 5.1]{KiWe26}, in order to get the dependence $R\|B\|_{L^\infty(\mathcal{H}_R(z_0))}$, instead of $\|B\|_{L^\infty(\mathcal{H}_R(z_0))}$.

Next, we provide Schauder-type estimates away from the grazing set $\gamma_0 =\{x_n=0\}\times \{v_n=0\}$.

\begin{lemma}\label{lem.schhol.away-}
    Let $z_0 \in \R^{1+2n}$ and $R \leq 1$. Let $\mathcal{Q}_R(z_0)\cap \gamma_0=\emptyset$. Let $f$ be a weak solution to 
    \begin{equation*}
\left\{
\begin{alignedat}{3}
\partial_tf+v\cdot\nabla_xf-\ddiv(A\nabla_vf)&=B\cdot\nabla_vf+F&&\qquad \mbox{in  $\mathcal{H}_R(z_0)$}, \\
f&=g&&\qquad  \mbox{in $\gamma_-\cap \mathcal{Q}_R(z_0)$},
\end{alignedat} \right.
\end{equation*}
    where $A\in C^{k-1,\eps}(\mathcal{H}_R(z_0))$, $B\in C^{k-2,\eps}(\mathcal{H}_R(z_0))$, $F\in C^{k-2,\eps}(\mathcal{H}_R(z_0))$, and $g\in C^{k,\eps}(\gamma_-\cap\mathcal{Q}_R(z_0))$ for some $k\in\N\cup\{0\}$ and $\eps\in(0,1)$. Then we have
    \begin{align*}
        R^{k+\eps}[f]_{C^{k,\eps}(\mathcal{H}_{R/2}(z_0))}\leq c(\|f\|_{L^\infty(\mathcal{H}_R(z_0))}+R^{k+\eps}[F]_{C^{k-2+\eps}(\mathcal{H}_R(z_0))}+R^{k+\eps}[g]_{C^{k+\eps}(\gamma_-\cap \mathcal{Q}_R(z_0))}),
    \end{align*}
    where $c=c(n,\Lambda,k,\eps,\|A\|_{C^{k-1+\eps}(\mathcal{H}_R(z_0))},\|B\|_{C^{k-2+\eps}(\mathcal{H}_R(z_0))})$. In particular, when $k=0$, we get
    \begin{align}\label{lem.schhol.away-.depB}
        R^{\eps}[f]_{C^{\eps}(\mathcal{H}_{R/2}(z_0))}\leq c\mathbf{B}(z_0,R)^{\eps}(\|f\|_{L^\infty(\mathcal{H}_R(z_0))}+R^{2}\|F\|_{L^{\infty}(\mathcal{H}_R(z_0))}+R^{\eps}[g]_{C^{\eps}(\gamma_-\cap \mathcal{Q}_R(z_0))}),
    \end{align}
    where $c=c(n,\Lambda,\eps,\|A\|_{C^{\eps}(\mathcal{H}_R(z_0))})$ and $\mathbf{B}(z_0,R)\coloneqq (1+\|B\|_{L^\infty(\mathcal{H}_R(z_0))})$.
\end{lemma}
\begin{proof}
This result was established in \cite[Proposition 4.5]{RoWe25} for equations in non-divergence form. The claim can be shown by combining the expansion estimates at the non-grazing boundary $\gamma_+\cup\gamma_-$ from \cite[Lemma 6.1]{KiWe26} with the interior regularity estimates from \cite[Lemma 2.5]{KiWe26} and \cite[Lemma 2.23]{RoWe25} in the exact same way as in \cite[Lemma 4.4]{RoWe25}. Lastly, \eqref{lem.schhol.away-.depB} follows by applying the scaling argument in the proof of \cite[(5.3) in Proposition 5.1]{KiWe26}.
\end{proof}

Lastly, we prove an elementary property about the reflection operator $R_b(z)=(t,x,v',-bv_n)$ for $b>0$. In particular, the property for the case $b=1$ is given in \cite[Lemma 2.20]{RoWe25}.
\begin{lemma}\label{lem.geo}
   Let $\gamma_\mp=\{x_n=0\}\times \{\pm v_n>0\}$. For any $R>0$ and $z_0\in \bbR^{1+2n}\cap \{x_n=0\}$, it holds 
    \begin{align*}
        R_b(\gamma_-\cap \mathcal{Q}_R(z_0))=\gamma_+\cap  \mathcal{Q}_{R,b}(R_b(z_0)),
    \end{align*}
    where 
    \begin{equation}\label{defn.bflat.cylinder}
    \begin{aligned}
        &\mathcal{Q}_{R,b}(R_b(z_0))\coloneqq \{(t,x,v)\,:\, t\in I_{R}(t_0),\, x\in B_{R^3}(x_0+(t-t_0)(v_0',-bv_{0,n})),\, v\in B_{R,b}(v_0',-bv_{0,n})\},\\
        &B_{R,b}\coloneqq \{(v',bv_n)\,:\, v\in B_R\}.
    \end{aligned}
    \end{equation} In addition, there is $c=c(n)$ such that we have 
    \begin{align*}
         [\mathcal{R}_bf]_{C^{k,\eps}(\gamma_-\cap \mathcal{Q}_R(z_0))}\leq c[f]_{C^{k,\eps}(\gamma_+ \cap  \mathcal{Q}_{R,b}(R_b(z_0)))}.
    \end{align*}
\end{lemma}

\begin{proof}
    The first claim follows directly from the definition. Next, we observe from \cite[(2.6)]{RoWe25} that for any $z,z_1\in \gamma_-\cap \mathcal{Q}_R(z_0)$,
    \begin{align*}
        \mathrm{dist}_{\mathrm{kin}}(z,z_1)&\eqsim |t-t_1|^{\frac12}+|x-(x_1+(t-t_1)v_1|^{\frac13}+|v-v_1|\\
        &\eqsim_n |t-t_1|^{\frac12}+|x'-(x_1'+(t-t_1)v_1'|^{\frac13}+|(t-t_1)v_{1,n}|^{\frac13}+|v-v_1|\\
        &\eqsim_{n,b}|t-t_1|^{\frac12}+|x'-(x_1'+(t-t_1)v_1'|^{\frac13}+|(t-t_1)bv_{1,n}|^{\frac13}+|R_{b}(v)-R_{b}(v_1)|\\
        &\eqsim_n\mathrm{dist}_{\mathrm{kin}}(R_{b}(z),R_{b}(z_1)),
    \end{align*}
    where $R_{b}(v)\coloneqq(v',-bv_n)$ and we have used the fact that $x_n=x_{1,n}=0$. In particular, we have 
    \begin{align*}
      \frac1c\mathrm{dist}_{\mathrm{kin}}(R_{b}(z),R_{b}(z_1))  \leq \mathrm{dist}_{\mathrm{kin}}(z,z_1)\leq \frac{c}{b}\mathrm{dist}_{\mathrm{kin}}(R_{b}(z),R_{b}(z_1))
    \end{align*} for some $c=c(n)$.
    Thus, the second claim follows immediately upon using the previous two observations and the fact that $z \mapsto p(R_{b}(z))\in \cP_k$, if $p\in\cP_k$.
\end{proof}

\section{Transfer of regularity}
\label{sec:transfer}

In this section, we establish a transfer of regularity result up to the boundary for kinetic equations posed in the half-space $\{ x_n > 0\}$. Our main results are \autoref{lem.extension}, which yields higher integrability of subsolutions in all variables and will be used to show local boundedness in the following section, as well as \autoref{prop.transbdd}, which contains fractional Sobolev regularity in $x$ and is of independent interest. 

Since such results are classical for global solutions the main idea in the proof of our results is to extend the equation from the half-space $\{ x_n > 0 \}$ to the whole space and to apply existing transfer of regularity results in that setting. The main technical difficulty comes from the fact that extending the solution yields an additional right-hand side, which belongs to the dual space of a fractional Sobolev space. Our argument works in a very general framework (see \autoref{prop.transbdd}) and holds for any boundary condition. 

In order to prove \autoref{lem.extension}, we first need to establish a suitable gain of integrability for the Kolmogorov equation in the full space with source terms of the form $(-\Delta_x)^{\frac{s}{2}}$. To do so, we follow the argument used in \cite[Theorem 2.1]{PaPo04} and \cite[Proposition 2.2]{ImSi20}. Note that in those papers, the authors consider right-hand sides given by $F_1+D_v^{s}F_2$, where $D_v$ is a differential operator with respect to the $v$-variable.

\begin{lemma}\label{lem.highint}
Given functions $F_1,F_2,F_3\in L^2((0,T)\times\bbR^{2n})$ and $s\in[0,\frac23)$, there is a distributional solution $f$ to
    \begin{align}\label{highint:eq.distri}
        \partial_tf+v\cdot\nabla_xf-\Delta_vf=F_1-\ddiv F_2+(-\Delta_x)^{\frac{s}{2}}F_3\quad\text{in } (0,T)\times \bbR^{2n}.
    \end{align}
     Moreover, we have
    \begin{align}\label{highint:est.lp}
        \|f\|_{L^2((0,T)\times \bbR^{2n})}+\|f\|_{L^q((0,T)\times \bbR^{2n})}\leq c\sum_{j=1}^3\|F_j\|_{L^2((0,T)\times \bbR^{2n})}
    \end{align}
    for some $q=q(n,s)>2$, where $c=c(n,s,T)$.
\end{lemma}
\begin{proof}
    Let $J=J(t,x,v)$ be the fundamental solution to the Kolmogorov equation in $(0,\infty) \times \R^{2n}$ and note that by \cite[Proposition 2.1]{ImSi20}, $J$ is of the form 
    \begin{align*}
        J(t,x,v)=\frac{c_n}{t^{2n}}\mathcal{J}\left(\frac{x}{t^{\frac32}},\frac{v}{t^{\frac12}}\right)
    \end{align*}
    for some constant $c_n$ and a smooth function $\mathcal{J}$ satisfying $\|D^k\mathcal{J}\|_{L^p(\bbR^{2n})}\leq c(k,n,p)$ for any $k \in \N \cup \{0\}$. In addition, we have for any $p\geq1$,
    \begin{equation}\label{highint:ineq1.normJ}
    \begin{aligned}
        \|\partial_vJ(t,\cdot,\cdot)\|_{L^p(\bbR^{2n})}=t^{-2n(1-\frac1{p})-\frac12}\|\partial_v\mathcal{J}\|_{L^p(\bbR^{2n})},\\
        \|(-\Delta_x)^{\frac{s}2}J(t,\cdot,\cdot)\|_{L^p(\bbR^{2n})}=t^{-2n(1-\frac1{p})-\frac{3}2s}\|(-\Delta_x)^{\frac{s}2}\mathcal{J}\|_{L^p(\bbR^{2n})},\\
         \|\partial_xJ(t,\cdot,\cdot)\|_{L^p(\bbR^{2n})}=t^{-2n(1-\frac1{p})-\frac{3}2}\|\partial_x\mathcal{J}\|_{L^p(\bbR^{2n})}.
    \end{aligned}
    \end{equation}
    We define a function $f$ by 
    \begin{align*}
        f(t,x,v)&\coloneqq \int_{0}^{t}\int_{\bbR^n}\int_{\bbR^n}J((\tau,y,w)^{-1}\circ (t,x,v)) F_1(\tau,y,w)\,dw\,dy\,d\tau\\
        &\quad-\int_{0}^{t}\int_{\bbR^n}\int_{\bbR^n}(-\nabla_w)J((\tau,y,w)^{-1}\circ (t,x,v)) \cdot F_2(\tau,y,w)\,dw\,dy\,d\tau\\
        &\quad-\int_{0}^{t}\int_{\bbR^n}\int_{\bbR^n}(-\Delta_y)^{\frac{s}{2}}J((\tau,y,w)^{-1}\circ (t,x,v))  F_3(\tau,y,w)\,dw\,dy\,d\tau
        \eqqcolon \sum_{j=1}^3\int_{0}^{t}f_j(\tau)\,d\tau.
    \end{align*}
    Let $p_1=1$ and $p_2=\frac{2n+1}{2n+\frac12(\frac{3}2s+1)}$ and note that
    \begin{align}\label{highint:ineq2.p1p2}
        p_2>1 \quad\text{and}\quad\left(-2n\left(1-\frac{1}{p_i}\right)-\frac3{2}s\right)p_i>-1,
    \end{align} as $s<\frac23$ and $\frac12(\frac32s+1)\in(\frac32s,1)$.  Next, choose $q_i$ for $i \in \{1,2\}$ by
    \begin{align*}
        1+\frac{1}{q_i}=\frac12+\frac{1}{p_i}
    \end{align*}
    to see that ${q}_1=2$ and $q_2>2$. As in \cite[Proposition 2.2]{ImSi20}, we use Young's convolution inequality together with \eqref{highint:ineq1.normJ} applied with $p=p_i$ to obtain
    \begin{align*}
        &\|f_1(\tau,\cdot,\cdot)\|_{L^{q_i}( \bbR^{2n})}\leq c\|F_1(\tau,\cdot,\cdot)\|_{L^2(\bbR^{2n})}(t-\tau)^{-2n(1-\frac1{p_i})}\|\mathcal{J}\|_{L^{p_i}(\bbR^{2n})},\\
        &\|f_2(\tau,\cdot,\cdot)\|_{L^{q_i}( \bbR^{2n})}\leq c\|F_2(\tau,\cdot,\cdot)\|_{L^2(\bbR^{2n})}(t-\tau)^{-2n(1-\frac1{p_i})-\frac12}(\|\partial_x\mathcal{J}\|_{L^{p_i}(\bbR^{2n})}+\|\partial_v\mathcal{J}\|_{L^{p_i}(\bbR^{2n})}),\\
         &\|f_3(\tau,\cdot,\cdot)\|_{L^{q_i}( \bbR^{2n})}\leq c\|F_3(\tau,\cdot,\cdot)\|_{L^2(\bbR^{2n})}(t-\tau)^{-2n(1-\frac1{p_i})-\frac32s}\|(-\Delta_x)^{\frac{s}2}\mathcal{J}\|_{L^{p_i}(\bbR^{2n})}.
    \end{align*}
    Using Young's convolution inequality again together with the second inequality in \eqref{highint:ineq2.p1p2}, we derive 
    \begin{align*}
        \|f_j\|_{L^{q_i}((0,T)\times \bbR^{2n})}\leq c\|F_j\|_{L^2((0,T)\times \bbR^{2n})},
    \end{align*}
    for any $j \in \{1,2,3\}$ and $i \in \{1,2\}$, where $c=c(n,s,T)$. In particular, we have used the condition $s<\frac23$ to obtain the estimate of $f_3$ in $L^{q_i}$ for each $i\in\{1,2\}$. This proves \eqref{highint:est.lp}. Moreover, since $f$ converges to zero at $\infty$ as $f\in L^2((0,T)\times \bbR^{2n})$, using the definition of $f$ and integration by parts yields that $f$ is a distributional solution to \eqref{highint:eq.distri}. This completes the proof.
\end{proof}

We are now in a position to prove a gain of integrability for nonnegative subsolutions $f$ to Kolmogorov's equation in $\mathcal{H}_1$. As we mentioned before, we will extend the equation to the whole space and then find a suitable global solution, which allows to use the comparison principle to obtain higher integrability. 

\begin{lemma}\label{lem.extension}
    Let $f\geq0$ be a subsolution to 
    \begin{align}\label{extension:eq.given}
        \partial_tf+v\cdot\nabla_xf-\Delta_vf= F_1+\nabla_v\cdot F_2\quad\text{in }\mathcal{H}_1,
    \end{align}
    where $F_1,F_2\in L^2(\mathcal{H}_1)$. Let $\rho\in[\frac12,1)$. Then we have
    \begin{equation}\label{extension:ineq.goal}
    \begin{aligned}
        \|f\|_{L^q(\mathcal{H}_{\rho})}&\leq \frac{c}{1-\rho}\left(\|\nabla_vf\|_{L^2(\mathcal{H}_1)}+\|f\|_{L^2(\mathcal{H}_1)}+\|F_2\|_{L^2(\mathcal{H}_1)}\right)\\
        &\quad+c\left(\|F_1\|_{L^2(\mathcal{H}_1)}+\left(\int_{\mathcal{Q}_1\cap \{x_n=0\}}|f v_n|^2\,d\Gamma\right)^{\frac12}\right)
    \end{aligned}
    \end{equation}
    for some constants $c=c(n)$ and $q=q(n)>2$.
\end{lemma}
Let us point out that the last term in \eqref{extension:ineq.goal} is finite whenever $f\in H^1_{\mathrm{kin}}$ by \cite[Proposition 4.2]{Sil22}.

\begin{proof}
    Let $\phi\in C_c^\infty(\mathcal{Q}_1)$ be a cutoff function such that $\phi\equiv 1$ in $\mathcal{Q}_\rho$. As a first step, we prove that $\widetilde{f}\coloneqq f\phi\mbox{\large$\chi$}_{x_n>0}$ is a subsolution of a suitable equation in the whole space. First, we choose a smooth function $\xi_\eps=\xi_\eps(x_n)$ such that $\xi_\eps(0)=0$, $\xi_\eps(\eps)=1$, and $0\leq\xi_\eps\leq 1$ with 
    \begin{align}\label{extension:moll.dirac}
        \xi_\eps\rightharpoonup \mbox{\large$\chi$}_{x_n>0}\quad\text{and}\quad\partial_{x_n}\xi_\eps\rightharpoonup^* \delta_{\{x_n=0\}},
    \end{align}as $\eps\to0$. Then we observe for any nonnegative smooth function $\psi\in C_c^\infty(\bbR^{1+2n})$, 
    \begin{align*}
        &\sum_{i=1}^2J_{i,\eps}\coloneqq\int_{\mathcal{H}_1}-f(\partial_t+v\cdot\nabla_x)(\phi\xi_\eps\psi)\,dz+\int_{\mathcal{H}_1}\nabla_vf\cdot\nabla_v(\phi\xi_\eps\psi)\,dz\\
        &\leq \int_{\mathcal{H}_1}F_1(\phi\xi_\eps\psi)-F_2\cdot \nabla_v(\phi\xi_\eps\psi)\,dz\eqqcolon J_{3,\eps},
    \end{align*}
    by testing \eqref{extension:eq.given} with $\phi\xi_\eps\psi \in C_c^\infty(\mathcal{H}_1)$. We now investigate the limit of $J_{i,\eps}$ as $\eps\to0$ for each $i\in\{1,2,3\}$.
    \begin{itemize}
        \item We rewrite $J_{1,\eps}$ as 
        \begin{align*}
            J_{1,\eps}=-\int_{\mathcal{H}_1}f\xi_\eps\phi(\partial_t+v\cdot\nabla_x)\psi\,dz-\int_{\mathcal{H}_1}f\xi_\eps\psi(\partial_t+v\cdot\nabla_x)\phi\,dz-\int_{\mathcal{H}_1}f\phi\psi (v_n\partial_{x_n}\xi_\eps)\,dz.
        \end{align*}
        Using \eqref{extension:moll.dirac} and the fact that $\mathrm{supp} (\phi\xi_\eps\psi)\subset \mathcal{H}_1$, we have 
        \begin{align*}
            \lim_{\eps\to0}J_{1,\eps}&=-\int_{\bbR}\int_{ \bbR^{2n}}f\mbox{\large$\chi$}_{x_n>0}\phi(\partial_t+v\cdot\nabla_x)\psi\,dz-\int_{\mathcal{H}_1}f\mbox{\large$\chi$}_{x_n>0}\psi(\partial_t+v\cdot\nabla_x)\phi\,dz\\
            &\quad-\int_{\mathcal{Q}_1\cap \{x_n=0\}}(f\phi v_n) \psi \,d\Gamma.
        \end{align*}
        \item Next, using the first property in \eqref{extension:moll.dirac}, we derive
        \begin{align*}
            \lim_{\eps\to0}J_{2,\eps}=\int_{\bbR}\int_{ \bbR^{2n}}\phi\nabla_v(f\mbox{\large$\chi$}_{\{x_n>0\}})\cdot\nabla_v\psi+\psi\nabla_v(f\mbox{\large$\chi$}_{\{x_n>0\}})\cdot\nabla_v\phi\,dz.
        \end{align*}
        Using the Leibniz rule, we further rewrite this limit as
        \begin{align*}
            \lim_{\eps\to0}J_{2,\eps}=\int_{\bbR}\int_{ \bbR^{2n}}\nabla_v(f\mbox{\large$\chi$}_{\{x_n>0\}}\phi)\cdot\nabla_v\psi-\int_{\mathcal{H}_1}(f\mbox{\large$\chi$}_{\{x_n>0\}})\nabla_v\phi\cdot\nabla_v\psi+\int_{\mathcal{H}_1}\nabla_v(f\mbox{\large$\chi$}_{\{x_n>0\}})\cdot\nabla_v\phi\,\psi.
        \end{align*}
        \item Similarly, we have 
        \begin{align*}
            \lim_{\eps\to0}J_{3,\eps}=\int_{\R}\int_{\bbR^{2n}}(F_1\phi-F_2\cdot\nabla_v\phi)\mbox{\large$\chi$}_{\{x_n>0\}} \psi-\phi\mbox{\large$\chi$}_{\{x_n>0\}}F_2\cdot\nabla_v\psi\,dz.
        \end{align*}
    \end{itemize}
    Combining all the previously established identities, we deduce that $\widetilde{f}=f\phi\mbox{\large$\chi$}_{\{x_n>0\}}$ is a subsolution to
    \begin{align*}
        (\partial_t+v\cdot\nabla_x)\widetilde{f}-\Delta_v\widetilde{f}=\widetilde{F}_1+\nabla_v\cdot\widetilde{F}_2+\mu\quad\text{in } \bbR^{1+2n},
    \end{align*}
    where 
    \begin{align*}
        &\widetilde{F}_1\coloneqq \left[f(\partial_t+v\cdot\nabla_x)\phi-\nabla_vf\cdot\nabla_v\phi+F_1\phi-F_2\cdot\nabla_v\phi\right]\mbox{\large$\chi$}_{\{x_n>0\}},\\
        &\widetilde{F}_2\coloneqq (-f\nabla_v\phi+\phi F_2)\mbox{\large$\chi$}_{\{x_n>0\}},\\
        &\mu\coloneqq fv_n\phi\delta_{\{x_n=0\}}.
    \end{align*}
    
    Next, we prove that $\mu\in L^2((-\infty,0]\times \bbR^n;H^{-s}_x(\bbR^n))$ for any $s>\frac12$. To do so, we observe that for any $\psi\in C_c^\infty(\R^{1+2n})$ and $s>\frac12$, 
    \begin{equation}\label{extension:est.meas}
    \begin{aligned}
        \left|\int_{\bbR^{1+2n}}\psi\,d\mu\right|=\left|\int_{\mathcal{Q}_1\cap\{x_n=0\}}v_nf\phi\psi\,d\Gamma\right|&\leq \left(\int_{\mathcal{Q}_1\cap\{x_n=0\}}|v_n f|^2\,d\Gamma)\right)^{\frac12}\left(\int_{\mathcal{Q}_1\cap\{x_n=0\}}|\psi|^2\,d\Gamma)\right)^{\frac12}\\
        &\leq c\left(\int_{\mathcal{Q}_1\cap\{x_n=0\}}|v_n f|^2\,d\Gamma)\right)^{\frac12}\|\psi\|_{L^2( \bbR^{1+n};H^{s}_x(\bbR^n))},
    \end{aligned}
    \end{equation}
    for some constant $c=c(s)$, where we have used H\"older's inequality and the embedding lemma given in \cite[Theorem 8.2]{DinPalVal12} with $n=1$ and $s>\frac12$. In particular, we have used in the last step
    \begin{align*}
        \left(\int_{\mathcal{Q}_1\cap\{x_n=0\}}|\psi(t,x',0,v)|^2\,d\Gamma)\right)^{\frac12}\leq \left(\int_{\mathcal{Q}_1\cap\{x_n=0\}}\|\psi\|_{H_{x_n}^s([0,1])}^2\,d\Gamma)\right)^{\frac12}&\leq c\|\psi\|_{L^2(\bbR^{1+2n-1};H^s_{x_n}(\bbR))}\\
        &\leq c\|\psi\|_{L^2(\bbR^{1+n};H^s_x(\bbR^n))}.
    \end{align*}
    Now using \eqref{extension:est.meas} and the representation for the Hilbert space $H^s$, we can write 
    \begin{align*}
        \int_{\bbR^{1+2n}}\psi\,d\mu=\int_{\bbR^{1+2n}}((-\Delta_x)^{\frac{s}2}\widetilde{F}_3+\widetilde{F}_4)\psi\,dz,
    \end{align*}
    where 
    \begin{align*}
        \|\widetilde{F}_3\|_{L^2(\bbR^{1+2n})}+\|\widetilde{F}_4\|_{L^2( \bbR^{1+2n})}\leq c\left(\int_{\mathcal{Q}_1\cap\{x_n=0\}}|v_n f|^2\,d\Gamma)\right)^{\frac12}.
    \end{align*}
    In addition, we observe that $\widetilde{F}_3(t,\cdot,\cdot)$ and $\widetilde{F}_4(t,\cdot,\cdot)$ are zero when $t\leq-1$, as $\mu=fv_n\phi\delta_{\{x_n=0\}}$, where $\phi\in C_c^\infty(\mathcal{Q}_1)$. Thus, $\widetilde{f}$ is a subsolution to 
    \begin{align}
        \label{eq:tilde-f}(\partial_t+v\cdot\nabla_x)\widetilde{f}-\Delta_v\widetilde{f}=\widetilde{F}_1+\nabla_v\cdot\widetilde{F}_2+(-\Delta_x)^{\frac{s}{2}}\widetilde{F}_3+\widetilde{F}_4\quad\text{in } \bbR^{1+2n},
    \end{align}
    where $\widetilde{F}_i(t,\cdot,\cdot)\equiv 0$ when $t\leq-1$, and it holds    \begin{equation}\label{extension:est.tildeFi}
    \begin{aligned}
        \sum_{i=1}^4\|\widetilde{F}_i\|_{L^2(\bbR^{1+2n})} &\leq \frac{c}{1-\rho}\left(\|\nabla_vf\|_{L^2(\mathcal{Q}_1)}+\|f\|_{L^2(\mathcal{Q}_1)}+\|F_2\|_{L^2(\mathcal{Q}_1)}\right)\\
        &\quad+c\left(\|F_1\|_{L^2(\mathcal{Q}_1)}+\left(\int_{\mathcal{Q}_1\cap \{x_n=0\}}|v_nf|^2\,d\Gamma\right)^{\frac12}\right)
    \end{aligned}
    \end{equation}
    for some constant $c=c(s)$, where we have used the fact that we can choose $\phi$ in such a way that $|\nabla_v\phi| + |(\partial_t+v\cdot\nabla_x)\phi| \leq \frac{c}{1-\rho}$, where $\rho\in[\frac12,1)$.

    We now find a distributional solution $f_0$ to the equation in \eqref{eq:tilde-f} and use the comparison principle to prove the desired result. To this end, first we fix $s=(\frac12+\frac23)\frac12 \in(\frac12,\frac23)$. Hence, by \autoref{lem.highint}, there is a distributional solution $f_0$ to 
    \begin{align*}
        (\partial_t+v\cdot\nabla_x){f}_0-\Delta_v{f}_0=\widetilde{F}_1+\nabla_v\cdot\widetilde{F}_2+(-\Delta_x)^{\frac{s}{2}}\widetilde{F}_3+\widetilde{F}_4\quad\text{in }(-2,2)\times \bbR^{2n}
    \end{align*}
    with 
    \begin{align}\label{extension:est.f0}
        \|f_0\|_{L^2((-2,2)\times \bbR^{2n})}+\|f_0\|_{L^q((-2,2)\times \bbR^{2n})}\leq c(n,s) \sum_{i=1}^4\|\widetilde{F}_i\|_{L^2( \bbR^{1+2n})}
    \end{align}
    for some $q=q(n)>2$. Since $f_0$ is not a weak solution, we need to mollify the function $f_0$, in order to apply the comparison principle. To do so, we 
    choose a smooth function $\eta\in C_c^\infty(\cQ_1)$ and define for any $\delta>0$, $(\widetilde{f}-f_0)_\delta\coloneqq (\widetilde{f}-f_0)* \eta_\delta$ with $\eta_\delta(z)\coloneqq \frac{1}{\delta^{4n+2}}\eta(S_{\frac1\delta} z)$. The function $(\widetilde{f}-f_0)_\delta$ mollifies the function $\widetilde{f}-f_0$, and converges to $\widetilde{f}-f_0$ in $L^2( \bbR^{1+2n})$ as $\delta\to0$. Moreover, since $\widetilde{f}$ is also a distributional solution to \eqref{eq:tilde-f}, we observe for small $\delta>0$,
    \begin{align*}
        (\partial_t+v\cdot\nabla_x)(\widetilde{f}-f_0)_\delta-\Delta_v(\widetilde{f}-f_0)_\delta\leq 0\quad\text{in }(-3/2,2)\times \bbR^{2n}
    \end{align*}
    with $(\widetilde{f}-f_0)_\delta(-\frac32,\cdot,\cdot)\equiv0$ as $(\widetilde{f}-f_0)(t,\cdot,\cdot)=0$ for any $t\leq-1$. By the comparison principle, we have $(\widetilde{f})_\delta\leq (f_{0})_{\delta}$ in $(-\frac32,2)\times \bbR^{2n}$. Combined with the convergence that $(\widetilde{f}-f_0)_\delta\to \widetilde{f}-f_0$ in $L^2$, this implies $\widetilde{f}\leq f_0$ in $(-\frac32,2)\times \bbR^{2n}$. Therefore, the desired estimate follows from \eqref{extension:est.tildeFi} and \eqref{extension:est.f0}.
\end{proof}

Our extension argument is sufficiently general to yield a transfer of regularity in the $x$-direction for kinetic equations posed in the half-space $\{x_n>0\}$. Since higher regularity up to boundary $\{x_n=0\}$ is of independent interest for kinetic equations, we conclude this section by establishing fractional regularity in the $x$-direction in the half-space.
\begin{proposition}\label{prop.transbdd}
    Let $f\in L^2(\{x_n>0\})$ be a solution to 
    \begin{align*}
        \partial_tf+v\cdot\nabla_xf=F_1+\nabla_v\cdot F_2\quad\text{in }\{x_n>0\},
    \end{align*}
    where $F_1,F_2\in L^2(\{x_n>0\})$. Then for any $\alpha<\frac14$,
    \begin{align*}
        \|(-\Delta_x)^{\frac\alpha2}(f\mbox{\large$\chi$}_{x_n>0})\|_{L^2(\bbR^{1+2n})}&\leq c\left(\|f\|_{L^2(\{x_n>0\}}+\|\nabla_vf\|_{L^2(\{x_n>0\})}+\sum_{i=1}^2\|F_i\|_{L^2(\{x_n>0\})}\right)\\
        &\quad+c\left(\int_{\bbR^{1+2n}\cap \{x_n=0\}}|v_nf|^2\,d\Gamma\right)^{\frac12}
    \end{align*}
    for some constant $c=c(n,\alpha)$.
\end{proposition}
\begin{proof}
    As in the proof of \autoref{lem.extension}, we observe that $\widetilde{f}\coloneqq f\mbox{\large$\chi$}_{x_n>0}$ is a solution to 
    \begin{align*}
         \partial_t\widetilde{f}+v\cdot\nabla_x\widetilde{f}=\widetilde{F}_1+\nabla_v\cdot \widetilde{F}_2+\mu\quad\text{in }\bbR^{1+2n},
    \end{align*}
    where $\mu\coloneqq fv_n\delta_{\{x_n=0\}}$ and $\widetilde{F}_i$ is extended zero when $x_n\leq0$. In particular, we have the representation of $\mu$ by
    \begin{align*}
        \mu=(-\Delta_x)^{\frac{r}{2}}\widetilde{F}_3+\widetilde{F}_4,
    \end{align*}
    where $r>\frac12$ and
    \begin{align}\label{transbdd:ineq1}
        \|\widetilde{F}_3\|_{L^2(\bbR^{1+2n})}+\|\widetilde{F}_4\|_{L^2(\bbR^{1+2n})}\leq c\left(\int_{\bbR^{1+2n}\cap \{x_n=0\}}|v_nf|^2\,d\Gamma\right)^{\frac12}
    \end{align}
    for some constant $c=c(s)$. 
    Now we repeat the the exact same proof in \cite[Theorem 2.1 and Cor 2.2]{Bou02} with $s=\frac{1-r}{2}$ for the equation
    \begin{align*}
        \partial_t\widetilde{f}+v\cdot\nabla_x\widetilde{f}=\widetilde{F}_1+\nabla_v\cdot\widetilde{F}_2+(-\Delta_x)^{\frac{r}{2}}\widetilde{F}_3+\widetilde{F}_4.
    \end{align*}
    Then we deduce
    \begin{align}\label{transbdd:ineq2}
        \|(-\Delta_x)^{\frac{1-r}4}\widetilde{f}\|_{L^2(\bbR^{1+2n})}\leq c\left(\|\widetilde{f}\|_{L^2(\mathbb{R}^{1+2n})}+\|\nabla_v\widetilde{f}\|_{L^2(\bbR^{1+2n})}+\sum_{i=1}^4\|\widetilde{F}_i\|_{L^2(\bbR^{1+2n})}\right)
    \end{align}
    for some constant $c=c(n,r)$. Since $r$ can be very close to $\frac12$, combining \eqref{transbdd:ineq1} and \eqref{transbdd:ineq2} leads to the desired estimate. 
\end{proof}

\section{H\"older regularity up to the boundary }\label{sec:4}

In this section, we establish H\"older regularity up to the boundary for kinetic equations with bounded measurable coefficients subject to a general reflection-type boundary condition of the form
\begin{align*}
    f =\alpha \mathcal{R}_bf +(1-\alpha)g \quad  \text{in } \gamma_-, \quad \text{where} \quad \mathcal{R}_bf(t,x',0,v) = f(t,x',0,v',-bv_n),
\end{align*}
 and $\alpha\in[0,1]$, $b\in(0,1]$, and $\alpha/b^2\in[0,1]$. Note that this condition encompasses Maxwell conditions and also super-elastic boundary conditions. 

Specifically, we first obtain $L^\infty$-estimates and subsequently establish $C^\alpha$-estimates. Finally, we explain how to prove local boundedness for solutions with bounce-back condition. This section indicates the general scope of applicability of our technique.

\subsection{Local boundedness}

The following is the main result of this section.

\begin{lemma}\label{lem.bdd}
   Let $z_0\in\gamma_0$ and $R\leq1$.  Let $f$ be a weak solution to 
    \begin{equation}\label{bdd:eq.bdry}
\left\{
\begin{alignedat}{3}
\partial_tf+v\cdot\nabla_xf-\ddiv(A\nabla_vf)&=B\cdot\nabla_vf+F&&\qquad \mbox{in  $\mathcal{H}_R(z_0)$}, \\
f&=\alpha \mathcal{R}_bf+(1-\alpha)g&&\qquad  \mbox{in $\mathcal{Q}_R(z_0)\cap\gamma_-$},
\end{alignedat} \right.
\end{equation}
for some $\alpha\in[0,1]$, $b\in(0,1]$ with $\alpha/b^2\in[0,1]$, where $g\equiv0$ if $\alpha=1$. Then we have 
\begin{align}\label{bdd:est.bdd}
    \|f\|_{L^\infty(H_{R/2}(z_0))}\leq c\left(\dashint_{\mathcal{H}_R(z_0)}|f|\,dz+R^2\|F\|_{L^\infty(\mathcal{H}_R(z_0))}+(1-\alpha)\|g\|_{L^\infty(\mathcal{Q}_R(z_0)\cap\gamma_-)}\right)
\end{align}
for some constant $c=c(n,\Lambda,R\|B\|_{L^\infty(\mathcal{H}_R(z_0))},b)$ independent of $\alpha$, but $c\to\infty$ as $b\to0$. 
\end{lemma}
To obtain this result, we will show that for any $k\geq {\|g\|_{L^\infty(\mathcal{Q}_R(z_0))}}$, it holds that $(f-k)_+$ is a subsolution to \eqref{bdd:eq.bdry} and, in particular, the $L^2$-norm of $(f-k)_+(n_x\cdot v)$ at the boundary $\gamma$ and $\nabla_v(f-k)_+$ is bounded (see \autoref{lem.extension1}). Here, the boundary condition plays a crucial role in controlling the norm of $(f-k)_+(n_x\cdot v)$. Then we apply \autoref{lem.extension} to the function $(f-k)_+$ to get a higher integrability result, which gives $L^\infty$ estimates by the standard De Giorgi iteration.

As a first step to get \eqref{bdd:est.bdd}, we prove the following.

\begin{lemma}\label{lem.changeofvar}
    Let $z_0\in\gamma_0$ and let $f$ be a weak solution to \eqref{bdd:eq.bdry} with $R=1$. Let $\frac12\leq r<\rho\leq1$ and $\varphi\in C^{1,1}(\bbR)$ with $\varphi(0)=0$ and $\|\varphi'\|_{L^\infty(\bbR)}+\|\varphi''\|_{L^\infty(\bbR)}\leq c$ for some constant $c$. For any $\psi\in C_c^\infty(\mathcal{Q}_{\frac{r+3\rho}4}(z_0))$, we have
    \begin{equation}\label{changeofvar:eq1}
    \begin{aligned}
        &-\int_{\mathcal{H}_{\frac{r+3\rho}{4}}(z_0)}\varphi(f)(\partial_t+v\cdot\nabla_x)\psi\,dz+\int_{\gamma\cap \mathcal{Q}_{\frac{r+3\rho}4}(z_0)}\varphi(f)\psi(n_x\cdot v)\,d\Gamma\\
        &\quad+\int_{\mathcal{H}_{\frac{r+3\rho}4}(z_0)}A\nabla_vf\cdot\nabla_v\psi\varphi'(f)\,dz+\int_{\mathcal{H}_{\frac{r+3\rho}4}(z_0)}A\nabla_vf\cdot\nabla_vf\varphi''(f)\psi\,dz\\
        &=\int_{\mathcal{H}_{\frac{r+3\rho}4}(z_0)}(B\cdot\nabla_vf+F)\varphi'(f)\psi\,dz.
    \end{aligned}
    \end{equation}
    Moreover, if $\psi=0$ in $\gamma_-$, then we have \eqref{changeofvar:eq1} without assuming the global boundedness of $\varphi'$.
\end{lemma}

 Note that since the equation \eqref{bdd:eq.bdry} is scale-invariant, we always assume $R=1$. For example, $f(z)\coloneqq f(z_0\circ S_Rz)$ is a solution to \eqref{bdd:eq.bdry} with $z_0$, $R$ $A$, $B$, $F$ and $g$ replaced by $0$, $1$, $A(z_0\circ S_Rz)$, $RB(z_0\circ S_Rz)$, $R^2F(z_0\circ S_Rz)$ and $g(z_0\circ S_Rz)$, respectively.

\begin{proof}
   As $z_0\in\gamma_0$, we may assume $t_0=-1$ to ensure that $\mathcal{H}_{1}(z_0)\subset \{t<0\}$. We are now going to mollify the equation to get the desired results. Choose a function ${\eta}\in C^{\infty}_c(\R^{1+2n})$ such that 
   \begin{align*}
        \mathrm{supp}\,{\eta}\subset \{t>0\}\cap \{x_n<0\}\cap \{v_n>0\}\quad\text{and}\quad \int_{\bbR^{1+2n}}{\eta}=1.
    \end{align*}Then for small $\delta>0$, the function $f_\delta\coloneqq \eta_\delta*f$ solves
    \begin{align}\label{extension1:eq.mollif}
    (\partial_t+v\cdot\nabla_x)f_\delta-\ddiv(\eta_\delta*(A\nabla_vf)   )=\eta_\delta *(B\cdot\nabla_vf)+\eta_\delta*F\quad\text{in }\mathcal{H}_{\frac{r+3\rho}{4}}(z_0),
    \end{align}
    where $\eta_\delta(z)\coloneqq \frac{1}{\delta^{4n+2}}\eta(S_{\frac1\delta}z)$.

    Now testing \eqref{extension1:eq.mollif} with $\varphi'(f_\delta)\psi$ for some $\psi\in C_c^\infty(\mathcal{Q}_{\frac{r+3\rho}{4}}(z_0))$, we have
 \begin{equation}\label{changeofvar:weak.for}
    \begin{aligned}
        &-\int_{\mathcal{H}_{\frac{r+3\rho}4}(z_0)}{\varphi_{}(f_\delta)}(\partial_t+v\cdot\nabla_x)\psi+\int_{\gamma\cap \mathcal{Q}_{\frac{r+3\rho}4}(z_0)}{\varphi_{}(f_\delta)}\psi(n_x\cdot v)\\
        &\quad+\int_{\mathcal{H}_{\frac{r+3\rho}4}(z_0)}\eta_\delta*(A\nabla_vf)\cdot\nabla_v\psi\varphi'_{}(f_\delta)+\int_{\mathcal{H}_{\frac{r+3\rho}4}(z_0)}\eta_\delta*(A\nabla_vf)\cdot\nabla_vf_\delta\varphi''_{}(f_\delta)\psi\\
        &=\int_{\mathcal{H}_{\frac{r+3\rho}4}(z_0)}\eta_\delta*(B\cdot\nabla_vf+F)\varphi'_{}(f_\delta)\psi.
    \end{aligned}
    \end{equation}
    Except for the second term on the left-hand side of \eqref{changeofvar:weak.for}, each limit as $\delta \to 0$ of the remaining terms in \eqref{changeofvar:weak.for} is already discussed in \cite[Lemma 5.6]{ Sil22}. Hence, we focus on the limit of the integral on the boundary $\gamma$.
    \begin{itemize}
        \item By the fact that $|\varphi'|\leq c$, we have 
\begin{equation*}
\begin{aligned}
    J\coloneqq\int_{\gamma\cap \mathcal{Q}_{\frac{r+3\rho}4}(z_0)}|({\varphi_{}(f_\delta)}-\varphi(f))\psi(n_x\cdot v)|&\leq c\int_{\gamma\cap \mathcal{Q}_{\frac{r+3\rho}4}(z_0)}|(f_\delta-f)\psi(n_x\cdot v)|\\
    &\leq c\left(\int_{\gamma\cap\supp\,\psi}|f-f_\delta|^2|n_x\cdot v|^2\right)^{\frac12},
\end{aligned}
\end{equation*}
which tends to $0$, as $\delta\to0$. In particular, we have used \cite[Proposition 4.3]{Sil22} together with the fact that $f_\delta \to f$ in $L^2_{t,x}H^1_v$ and $(\partial_t+v\cdot\nabla_x)f_\delta\rightharpoonup (\partial_t+v\cdot\nabla_x)f$ in $L^2_{t,x}H^{-1}_v$. 
\item Next, we consider the case that $\varphi'$ is not bounded, but $\psi\equiv0$ in $\gamma_-$. Then we observe 
\begin{align*}
&    |\psi(t,x',0,v,v_n)|=|\psi(t,x',0,v,v_n)-\psi(t,x',0,v,0)|\leq c|v_n|,\\
&|\varphi(f_\delta)-\varphi(f)|\leq\left| \int_{0}^1\varphi'(\xi f_\delta+(1-\xi)f)\,d\xi\right||f_\delta-f|\leq c(1+|f_\delta|+|f| )|f_\delta-f|,
\end{align*}
where we have used $|\varphi''|\leq c$ for the last inequality.
Using this, we have
\begin{equation*}
\begin{aligned}
   J&\leq c\int_{\gamma\cap \supp \psi}|(1+|f_\delta|+|f|))(f_\delta-f)(n_x\cdot v)^2|\\
    &\leq c\left(\int_{\gamma\cap\supp\,\psi}|f-f_\delta|^2|n_x\cdot v|^2\right)^{\frac12}\left(\int_{\gamma\cap\supp\,\psi}(1+|f|^2+|f_\delta|^2)|n_x\cdot v|^2\right)^{\frac12},
\end{aligned}
\end{equation*}
which goes to 0, as $\delta\to0$.
    \end{itemize}
As we mentioned before, by applying the same arguments given in the proof of \cite[Lemma 5.6]{Sil22} to the remaining terms in \eqref{changeofvar:weak.for}, we get \eqref{changeofvar:eq1} by taking $\delta\to0$. This completes the proof.
\end{proof}
Now using the formula established in \eqref{changeofvar:eq1}, we are ready to prove that $(f-k)_+$ is a subsolution. Moreover, we prove that its weighted boundary $L^2$ norm on $\gamma$ with the weight $|v\cdot n_x|^2$, as well as the $L^2$ of its gradient, are bounded.
\begin{lemma}\label{lem.extension1}
    Let $z_0\in\gamma_0$ and let $f$ be a weak solution to \eqref{bdd:eq.bdry} with $R=1$. Let $\frac12\leq r<\rho\leq1$ and $k\geq{\|g\|_{L^\infty(\mathcal{Q}_1(z_0)\cap \gamma_-)}}$.  Then $(f-k)_+$ is a subsolution to 
     \begin{align}\label{extension1:eq.subsol}
        \partial_t(f-k)_++v\cdot\nabla_x(f-k)_+-\ddiv(A\nabla_v(f-k)_+)=F_1\quad\text{in }\mathcal{H}_{\frac{r+\rho}2}(z_0),
    \end{align}
    where 
    \begin{align*}
        \int_{\mathcal{H}_{\frac{r+\rho}{2}}(z_0)}|F_1|^2\,dz\leq c\left(\int_{\mathcal{H}_{\rho}(z_0)}|\nabla_v(f-k)_+|^2+|F\mbox{\large$\chi$}_{f>k}|^2\,dz\right)
    \end{align*}
    for some constant $c=c(n,\Lambda)$. Furthermore, we have
    \begin{align}\label{extension1:ineq.bdry.weightv}
        \int_{\mathcal{Q}_{r}(z_0)\cap \gamma}(f-k)_+^2{|v\cdot n_x|^2}\,d\Gamma+\int_{\mathcal{H}_r(z_0)}|\nabla_v(f-k)_+|^2\,dz\leq c\left(\int_{\mathcal{H}_{\rho}(z_0)}\frac{|(f-k)_+|^2}{(\rho-r)^2}+|F\mbox{\large$\chi$}_{f>k}|^2\,dz\right)
    \end{align}
    for some constant $c=c(n,\Lambda,b)$, where $c\to\infty$ as $b\to0$. 
\end{lemma}
\begin{proof}
    As $z_0\in\gamma_0$, we may assume without loss of generality that $z_0=0$. Since $(\xi-k)_+$ is not in $C^{1,1}$, we consider a regularized function $\varphi_{\eps}$ by 
\begin{equation*}
\begin{aligned}
    \varphi_{\eps}(\xi)\coloneqq \begin{cases}
        \xi-k-\eps\quad&\text{if }\xi>k+2\eps,\\
        h_\eps(\xi)\quad&\text{if }k<\xi\leq k+2\eps,\\
        0\quad&\text{if }\xi\leq k
    \end{cases}
\end{aligned}
\end{equation*}
for some convex and smooth function $h_\eps$ such that $|\varphi_{\eps}'|\leq c$. In particular, we have $\varphi_{\eps}(\xi)\to (\xi-k)_+$, as $\eps\to0$. 

By \autoref{lem.changeofvar}, we have for any $\psi\in C_c^\infty(\mathcal{H}_{\frac{r+\rho}2})$,
 \begin{equation*}
    \begin{aligned}
        &-\int_{\mathcal{H}_{\frac{r+3\rho}4}}{\varphi_{\eps}(f)}(\partial_t+v\cdot\nabla_x)\psi+\int_{\mathcal{H}_{\frac{r+3\rho}4}}(A\nabla_vf)\cdot\nabla_v\psi\varphi'_{\eps}(f)
        \leq \int_{\mathcal{H}_{\frac{r+3\rho}4}}(B\cdot\nabla_vf+F)\varphi'_{\eps}(f)\psi,
    \end{aligned}
    \end{equation*}
    where we have used the fact that $\psi\equiv0$ on $\gamma$ and $\varphi_\eps''\geq0$. Taking $\eps\to0$, we deduce that $(f-k)_+$ is a subsolution to 
\begin{align*}
    \partial_t(f-k)_++v\cdot\nabla_x(f-k)_+-\ddiv(A\nabla_v(f-k)_+)= F_1\quad\text{in }H_{\frac{r+\rho}{2}}
\end{align*}
with
\begin{align*}
    F_1\coloneqq B\cdot\nabla_v(f-k)_++F\mbox{\large$\chi$}_{f>k}.
\end{align*}
Actually, note that \eqref{extension1:eq.subsol} holds for any $k\in\bbR$.

Now, we are going to prove \eqref{extension1:ineq.bdry.weightv}. To do so we will use two different test functions to get $L^2$ bounds of $|\nabla_v(f-k)_+|$ and $(f-k)_+|v\cdot n_x|$, respectively.

 Let $\psi \in C_c^\infty(\mathcal{Q}_{\frac{r+\rho}2})$ with $\psi\equiv 1$ on $\mathcal{Q}_r\cap \{x_n\geq0\}$ and $\psi(t,x,v',v_n)=\psi(t,x,v',-bv_n)$ for $v_n\geq0$, which gives $|\nabla_v\psi|+|(\partial_t+v\cdot\nabla_x)\psi|\leq\frac{c}{b(\rho-r)}$. For $M\gg1$, we define
    \begin{align*}
    \varphi_{M}(\xi)\coloneqq\begin{cases}  2M(\xi-\frac{M}2-k)&\quad\text{if }\xi>M+k,\\
    (\xi-k)_+^2&\quad\text{if }\xi\leq M+k,
    \end{cases}
    \end{align*}
    to see that $\varphi_M\in C^{1,1}$ with $\varphi_M(0)=0$, $\|\varphi'\|_{L^\infty(\bbR)}+\|\varphi''\|_{L^\infty(\bbR)}\leq c(M)$ and that $\varphi_M$ is convex. Note that we cannot apply \autoref{lem.changeofvar} directly with $\varphi(\xi)=(\xi-k)_+^2$, as $\psi\equiv 1$ in $\mathcal{Q}_r$ and since we do not know that the second term on the left-hand side of \eqref{changeofvar:eq1} is well-defined. Therefore, we introduce the  auxiliary function $\varphi_M$ and take $M\to\infty$ later.
    
    Now using \autoref{lem.changeofvar} with $\varphi$ and $\psi$ replaced by $\varphi_M$ and $\psi^2$, we have
    \begin{equation}\label{extension1:ene.ineq.grad} 
    \begin{aligned}
       \sum_{i=1}^2J_i &\coloneqq\int_{\gamma\cap \mathcal{Q}_{\frac{r+3\rho}4}}\varphi_M(f)\psi^2(n_x\cdot v)\,d\Gamma+\int_{\mathcal{H}_{\frac{r+3\rho}4}}A\nabla_vf\cdot\nabla_vf\varphi_M''(f)\psi^2\,dz\\
        &\leq-\int_{\mathcal{H}_{\frac{r+3\rho}4}}A\nabla_vf\cdot\nabla_v\psi^2\varphi_M'(f)\,dz+ \int_{\mathcal{H}_{\frac{r+3\rho}{4}}}\varphi_M(f)(\partial_t+v\cdot\nabla_x)\psi^2\,dz\\
        &\qquad+\int_{\mathcal{H}_{\frac{r+3\rho}4}}(B\cdot\nabla_vf+F)\varphi_M'(f)\psi^2\,dz\eqqcolon\sum_{i=3}^5J_i.
    \end{aligned}
    \end{equation}
    We now estimate the terms $J_i$. Using the boundary condition, we have
        \begin{align*}
            J_{1,1}&\coloneqq\int_{\gamma_-\cap \mathcal{Q}_{\frac{r+3\rho}4}}\varphi_M(f)\psi^2(n_x\cdot v)\,d\Gamma\\
            &=\int_{\gamma_-\cap \mathcal{Q}_{\frac{r+3\rho}4}}\varphi_M(\alpha \mathcal{R}_bf+(1-\alpha)g)\psi^2(n_x\cdot v)\,d\Gamma\\
            &\leq\int_{\gamma_-\cap \mathcal{Q}_{\frac{r+3\rho}4}}(\alpha \mathcal{R}_bf+(1-\alpha)g-k)_+^2{\mbox{\large$\chi$}}_{\{\alpha \mathcal{R}_bf+(1-\alpha)g<{M+k}\}}\psi^2(n_x\cdot v)\,d\Gamma\\
            &\quad+\int_{\gamma_-\cap \mathcal{Q}_{\frac{r+3\rho}4}}2M(\alpha \mathcal{R}_bf-{M}/2-\alpha k){\mbox{\large$\chi$}}_{\{\alpha \mathcal{R}_bf+(1-\alpha)g>{M+k}\}}\psi^2(n_x\cdot v)\,d\Gamma\eqqcolon J_{1,1,1}+J_{1,1,2},
        \end{align*}
        where we have also used the fact that $(1-\alpha)g-(1-\alpha)k\leq0$ for the last term. For the term $J_{1,1,1}$, we first observe that for $M\gg1$,
        \begin{align*}
          (\alpha f+(1-\alpha)\mathcal{R}_bg-k)^2   \leq 2M\alpha(f-M/2-k)\quad\text{if } M+k<f<\frac{M+k-(1-\alpha)\mathcal{R}_bg}{\alpha}.
        \end{align*}
        Thus, using the change of variables $v_n \mapsto - b v_n$, the above observation, the fact that $g\leq k$, and that $\psi^2(t,x,v',v_n)=\psi^2(t,x,v',-bv_n)$, we get
        \begin{align*}
            J_{1,1,1}&\geq-\frac{1}{b^2}\int_{\Gamma_b}\alpha(f-k)_+^2{\mbox{\large$\chi$}}_{\{f<{M+k}\}}\psi^2(n_x\cdot v)_+\,d\Gamma\\
            &\quad-\frac{1}{b^2}\int_{\Gamma_b}(\alpha f+(1-\alpha)\mathcal{R}_bg-k)_+^2{\mbox{\large$\chi$}}_{\{M+k<f<\frac{M+k-(1-\alpha)\mathcal{R}_bg}{\alpha}\}}\psi^2(n_x\cdot v)_+\,d\Gamma\\
            &\geq -\frac{\alpha^2}{b^2}\int_{\Gamma_b}(f-k)_+^2{\mbox{\large$\chi$}}_{\{f<{M+k}\}}\psi^2(n_x\cdot v)\,d\Gamma\\
            &\quad-\frac{\alpha}{b^2}\int_{\Gamma_b}2M(f-M/2-k)_+{\mbox{\large$\chi$}}_{\{M+k<f<\frac{M+k-(1-\alpha)\mathcal{R}_bg}{\alpha}\}}\,d\Gamma,
        \end{align*}
        where we write $\Gamma_b\equiv\gamma_+\cap \mathcal{Q}_{\frac{r+3\rho}4,b}\subset \gamma_+\cap \mathcal{Q}_{\frac{r+3\rho}4}$ and the domain $\mathcal{Q}_{\frac{r+3\rho}4,b}$ is determined in \autoref{lem.geo} (applied with $z_0 = 0$).
        Similarly,  we get
        \begin{align*}
            J_{1,1,2}\geq -\frac{\alpha}{b^2}\int_{\Gamma_b}2M(f-M/2-k){\mbox{\large$\chi$}}_{\{f>\frac{M+k-(1-\alpha)\mathcal{R}_bg}{\alpha}\}}\,d\Gamma,
        \end{align*}
        where we have also used the fact that $(\alpha f-M/2-\alpha k)_+\leq\alpha( f-M/2- k)_+$. Combining all the estimates $J_{1,1,1}$ and $J_{1,1,2}$ together with $\Gamma_b\subset \gamma_+\cap \mathcal{Q}_{\frac{r+3\rho}4}$ and
        the assumption $\alpha/b^2\leq1$, we have 
        \begin{align*}
            J_{1,1}+\int_{\gamma_+\cap \mathcal{Q}_{\frac{r+3\rho}4}}\varphi_M(f)\psi^2(n_x\cdot v)\,d\Gamma \ge \left( 1 - \frac{\alpha}{b^2} \right) \int_{\gamma_+\cap \mathcal{Q}_{\frac{r+3\rho}4}}\varphi_M(f)\psi^2(n_x\cdot v)\,d\Gamma \geq0,
        \end{align*}
         which implies $J_1\geq0$.

         After a few simple computations and using the definition of $\varphi_M$, we deduce
         \begin{align*}
             &J_2\geq2\int_{\mathcal{H}_{\frac{r+3\rho}4}}A\nabla_v(f-k)_+\cdot\nabla_v(f-k)_+{\mbox{\large$\chi$}}_{\{f<{M+k}\}}\psi^2,\\
             &J_3\leq c\int_{\mathcal{H}_{\frac{r+3\rho}4}}\psi|\nabla_v(f-k)_+||\nabla_v\psi|(f-k)_+,\\
             &J_4\leq \frac{c}{(\rho-r)^2}\int_{\mathcal{H}_{\frac{r+3\rho}4}}(f-k)_+^2,\\
             &J_5\leq c\int_{\mathcal{H}_{\frac{r+3\rho}4}}(|B\nabla_v(f-k)_+|+|F|)(f-k)_+ \psi^2
         \end{align*}
         for some constant $c=c(n,\Lambda,b)$,
         where we have used the fact that $M\leq (f-k)_+$ when $f>M+k$. Altogether, we obtain 
         \begin{align*}
             \int_{\mathcal{H}_{\frac{r+3\rho}4}}A\nabla_v(f-k)_+\cdot\nabla_v(f-k)_+{\mbox{\large$\chi$}}_{\{f<{M+k}\}}\psi^2&\leq \frac{c}{(\rho-r)}\int_{\mathcal{H}_{\frac{r+3\rho}4}}|\nabla_v(f-k)_+|\psi(f-k)_+\\
             &+c\int_{\mathcal{H}_{\frac{r+3\rho}4}}\frac{(f-k)^2_{+}}{(\rho-r)^2}+|F{\mbox{\large$\chi$}}_{\{f>k\}}|^2.
         \end{align*}
         Now taking $M\to\infty$ and using Young's inequality, we get
\begin{equation}\label{extension1:ineq2}
    \begin{aligned}
         \int_{\mathcal{H}_{r,b}}|\nabla_v(f-k)_+|^2
        &\leq c\int_{\mathcal{H}_{\rho}}\frac{(f-k)_+^2}{(\rho-r)^2}+|F\mbox{\large$\chi$}_{f>k}|^2
    \end{aligned}
    \end{equation}
for some constant $c=c(n,\Lambda)$, where $\mathcal{H}_{r,b}=\mathcal{Q}_{r,b}\cap \{x_n>0\}$, as the test function $\psi\equiv1$ on $\mathcal{Q}_r\cap \{x_n\geq0\}$ and $\psi(t,x,v',v_n)=\psi(t,x,v',-bv_n)$ for $v_n\geq0$. 

To obtain the desired estimate, we need to extend the domain of the integral on the left-hand side of \eqref{extension1:ineq2}, from $\mathcal{H}_{r,b}$ to $\mathcal{H}_r$. Note that for any $z_0=(t_0,x_0,v_0',v_{0,n})\in \mathcal{H}_r\setminus \mathcal{H}_{r,b}$, we have $v_{0,n}>-br$. Thus, for any $\mathcal{H}_{m({\rho-r})}(z_0)$ with a small constant $m=m(b)$ such that $\mathcal{H}_{m(\rho-r)}(z_0)\cap\gamma_0=\emptyset$ and $\mathcal{H}_{m({\rho-r})}(z_0)\subset \mathcal{H}_\rho$, we have the usual energy estimate
\begin{align*}
    \int_{\mathcal{H}_{\frac{m(\rho-r)}{2}}(z_0)}|\nabla_v(f-k)_+|^2\leq  c\int_{\mathcal{H}_{m({\rho-r})}(z_0)}\frac{(f-k)_+^2}{(\rho-r)^2}+|F\mbox{\large$\chi$}_{f>k}|^2,
\end{align*}
as the cylinder $\mathcal{H}_{{m(\rho-r)}}(z_0)$ is away from $\gamma_-\cup\gamma_0$. Combining this and \eqref{extension1:ineq2} together with the covering argument from \cite[Lemma 2.3]{KLN25} to ensure a uniformly bounded overlap of the cylinders, we have 
\begin{align}\label{extension1:ineq3}
    \int_{\mathcal{H}_r(z_0)}|\nabla_v(f-k)_+|^2\,dz\leq c\left(\int_{\mathcal{H}_{\rho}(z_0)}\frac{|(f-k)_+|^2}{(\rho-r)^2}+|F\mbox{\large$\chi$}_{f>k}|^2\,dz\right)
\end{align}
for some constant $c=c(n,\Lambda,b)$, where $c\to\infty $ as $b\to0$.

Next, we are going to obtain weighted $L^2$ estimates at $\gamma$. To this end, we choose a function
    $\phi$ that satisfies $\phi(t,x',v',v_n)\eqsim |v_n|$ in $\mathcal{Q}_r\cap\{v_n<0\}$, $\phi(t,x,v',v_n)=0$ in $v_n\geq0$, and $\supp \phi\subset \mathcal{Q}_{\frac{r+\rho}2}$. Since $\phi=0$ in $\gamma_-$, by \autoref{lem.changeofvar} with $\psi=\phi$, and $\varphi(\xi)=(\xi-k)_+^2$, we have \eqref{changeofvar:eq1} with $\psi=\phi$ and $\varphi(\xi)=(\xi-k)_+^2$. Using the fact that $\varphi''\geq0$ and the definition of $\varphi$, we directly deduce that 
    \begin{equation}\label{extension1:ineq.energy1}
    \begin{aligned}
         &\int_{\gamma\cap\mathcal{Q}_{\frac{r+\rho}2}}(f-k)^2_{+}(n_x\cdot v)\phi+\int_{\mathcal{H}_{\frac{r+\rho}2}}A\nabla_v(f-k)_+\cdot\nabla_v(f-k)_+\phi\\
        &\leq 2\left(\int_{\mathcal{H}_{\frac{r+\rho}2}}(f-k)_+^2|(\partial_t+v\cdot\nabla_x)\phi|+\int_{\mathcal{H}_{\frac{r+\rho}2}}|A\nabla_v(f-k)_+||\nabla_v\phi |(f-k)_+\right)\\
        &\quad+2\int_{\mathcal{H}_{\frac{r+\rho}2}}(B\cdot\nabla_v(f-k)_++F{\mbox{\large$\chi$}}_{f>k})(f-k)_+\phi.
    \end{aligned}
    \end{equation}
    Now, using Cauchy's inequality, the fact that $|(\partial_t+v\cdot\nabla_x)\phi|+|\nabla_v\phi|\leq c(\rho-r)^{-1}$ and \eqref{extension1:ineq3} with $r$ and $\rho$ replaced by $\frac{r+\rho}{2}$ and $\frac{r+3\rho}4$, we deduce
    \begin{align}\label{extension1:ineq0}
        \int_{\gamma\cap\mathcal{Q}_{\frac{r+\rho}2}}(f-k)^2_{+}\phi|n_x\cdot v|\leq c\int_{\mathcal{H}_{\frac{r+3\rho}4}}\frac{(f-k)_+^2}{(\rho-r)^2}+|F\mbox{\large$\chi$}_{f>k}|^2.
    \end{align}
    We now use the condition $\phi(t,x,v',v_n)\eqsim |v_n|=(n_x\cdot v)_+$ in $\mathcal{Q}_r\cap \{v_n<0\}$ and $\phi(t,x,v',v_n)=0$ in $v_n\geq0$ to see that
    \begin{align}\label{extension1:ineq01}
        \int_{\mathcal{Q}_r\cap\gamma_+}(f-k)_+^2(n_x\cdot v)_+^2\,d\Gamma\leq c\int_{\mathcal{H}_{\frac{r+3\rho}4}}\frac{(f-k)_+^2}{(\rho-r)^2}+|F\mbox{\large$\chi$}_{f>k}|^2.
    \end{align}
    We now use the boundary condition to obtain a suitable estimate also on $\gamma_-$. From the fact that $f=\alpha\mathcal{R}_bf+(1-\alpha)g$ on $\gamma_-$, the choice of $k$ and the change of variables $v_n \mapsto - b v_n$, we observe 
\begin{align*}
    \int_{\mathcal{Q}_r\cap\gamma_-}(f-k)_+^2(n_x\cdot v)_-^2\,d\Gamma&=  \int_{\mathcal{Q}_r\cap\gamma_-}\Big(\alpha(\mathcal{R}_bf-k)+ (1-\alpha)(g-k)  \Big)_+^2(n_x\cdot v)^2_-\,d\Gamma\\
    &=\int_{\gamma_+\cap \mathcal{Q}_{r,b}} \Big(\alpha(f-k)+(1-\alpha)(\mathcal{R}_bg-k)\Big)_+^2 (n_x\cdot v)^2_+\frac{\,d\Gamma}{b^3}\\
    &\leq \int_{\mathcal{Q}_r\cap\gamma_+}\frac{\alpha^2}{b^3}(f-k)_+^2 (n_x\cdot v)^2_+\,d\Gamma,
\end{align*}
where we have used $\mathcal{Q}_{r,b}\subset \mathcal{Q}_r$ and $\mathcal{R}_b g \le k$. 
When $\alpha=0$, by the boundary condition, we already have $(f-k)_+\equiv0$ on $\mathcal{Q}_r\cap\gamma_-$.
Thus, by combining the previous two estimates, we have 
\begin{align*}
        \int_{\gamma\cap \mathcal{Q}_r}(f-k)^2_{+}(n_x\cdot v)^2\leq c\int_{\mathcal{H}_{\rho}}\frac{(f-k)_+^2}{(\rho-r)^2}+|F\mbox{\large$\chi$}_{f>k}|^2
    \end{align*}
    for some constant $c=c(n,\Lambda,b)$ with $c\to\infty$ as $b\to0$.
Combining this with \eqref{extension1:ineq3}, we finish the proof.

\end{proof}

We now combine \autoref{lem.extension} with \autoref{lem.extension1} to derive a reverse H\"older-type inequality of $(f-k)_+$.
\begin{lemma}\label{lem.q.int}
   Let $z_0\in\gamma_0$ and let $f$ be a weak solution to \eqref{bdd:eq.bdry} with $R=1$. Let $\frac12\leq r<\rho\leq1$ and $k\geq{\|g\|_{L^\infty(\mathcal{Q}_1(z_0)\cap \gamma_-)}}$. Then we have
    \begin{align*}
        \left(\int_{\mathcal{H}_{r}(z_0)}({f}-k)_+^q\,dz\right)^{\frac{2}q}\leq c\left(\int_{\mathcal{H}_\rho(z_0)}\frac{({f}-k)_+^2}{(\rho-r)^2}+|F\mbox{\large$\chi$}_{f>k}|^2\,dz\right)
    \end{align*}
    for some $q>2$, where $c=c(n,\Lambda,b)$ with $c\to\infty$ as $b\to0$. 
\end{lemma}
\begin{proof}
    Since $(f-k)_+$ is a nonnegative subsolution to \eqref{extension1:eq.subsol}, we can rewrite this as
    \begin{align*}
        \partial_t(f-k)_++v\cdot\nabla_x(f-k)_+-\Delta_v(f-k)_+=F_1+\nabla_v\cdot F_2\quad\text{in }\mathcal{H}_{\frac{r+\rho}2(z_0)},
    \end{align*}
    where $F_2\coloneqq (I-A)\nabla_v(f-k)_+$.
    Thus, by applying \autoref{lem.extension} and using \eqref{extension1:ineq.bdry.weightv}, we deduce the desired estimate. 
\end{proof}
Now we are ready to prove \autoref{lem.bdd}.
\begin{proof}[Proof of \autoref{lem.bdd}.]
By a scaling argument, we may assume $R=1$. By the standard De Giorgi iteration argument as in the proof of \cite[Theorem 3.1]{GIMV19}, we obtain the desired estimate \eqref{bdd:est.bdd}. This completes the proof.
\end{proof}

\subsection{H\"older regularity}

We now present the second main result of this section. Since we have already proved in \autoref{lem.bdd} that solutions are bounded up to the boundary, we can assume that $f\big\rvert_{\gamma}$ is a bounded function. Away from $\gamma_0$, the regularity at $\gamma_+$ ensures the H\"older continuity of $\mathcal{R}_bf$. Consequently, for $g\in C^\alpha(\gamma_-)$, the boundary term $\alpha\mathcal{R}_bf+(1-\alpha)g$ is H\"older continuous away from $\gamma_0$. Hence, the regularity theory for in-flow (see \cite{Sil22}) yields that $f$ is H\"older-continuous up to $\gamma_-\cup\gamma_+$. We may therefore assume that $f$ attains its boundary values in a continuous manner when $g\in C^\alpha(\gamma_-)$. 

To prove the boundary regularity up to $\gamma_0$, we treat the problem as a kinetic equation with in-flow condition and use the improvement of oscillation scheme given in \cite[Lemma 8.4]{Sil22}. To deduce that the oscillation of the boundary data decays fast enough for solutions to be H\"older continuous, we crucially use that $\alpha<1$ (see \eqref{holsmall:ineq1.osck}).

\begin{lemma}\label{lem.oscillation}
 Let $z_0\in\gamma_0$ and $R\leq1$.  Let $f$ be a weak solution to 
    \begin{equation}\label{oscillation:eq.bdry}
\left\{
\begin{alignedat}{3}
\partial_tf+v\cdot\nabla_xf-\ddiv(A\nabla_vf)&=B\cdot\nabla_vf+F&&\qquad \mbox{in  $\mathcal{H}_R(z_0)$}, \\
f&=\alpha \mathcal{R}_bf+(1-\alpha)g&&\qquad  \mbox{in $\mathcal{Q}_R(z_0)\cap\gamma_-$},
\end{alignedat} \right.
\end{equation}
for some $\alpha\in[0,1)$, $b\in(0,1]$ with $\alpha/b^2\in[0,1]$. Then there is $\beta=\beta(n,\Lambda,\alpha)$ such that
\begin{align}\label{holsmall:est.hol}
    \osc_{\mathcal{H}_\rho(z_0)}f\leq c\left(\frac{\rho}{R}\right)^\beta\left(\|f\|_{L^\infty(\mathcal{H}_R(z_0))}+R^2\|F\|_{L^\infty(\mathcal{H}_R(z_0))}+R^{\beta}[g]_{C^\beta(\mathcal{Q}_R(z_0)\cap\gamma_-)}\right)
\end{align}
for some constant $c=c(n,\Lambda,R\|B\|_{L^\infty(\mathcal{H}_R(z_0))},\alpha)$. Moreover, we have the following limit behavior of the constants $\beta$ and $c$: $\lim\limits_{\alpha\to0}\beta>0$, $\lim\limits_{\alpha\to1}\beta=0$, $\lim\limits_{\alpha\to0}c<\infty$, $\lim\limits_{\alpha\to1}c=\infty$.
\end{lemma}
Note that the fact that $\beta\to0$ and $c\to\infty $ as $\alpha\to1$ is natural due to the 1D computations in \autoref{lem.phialpha}.
\begin{proof}
    Since $z_0\in\gamma_0$, we may assume $z_0=0$ and $R=1$. Thus, we will prove that there is a small constant $\beta=\beta(n,\Lambda,\alpha)$ such that for any $\rho\leq \frac{1}{2}$,
    \begin{align}\label{holsmall:ineq0.osc}
        \osc_{\mathcal{H}_\rho}f\leq c\rho^{\beta}\left(\|f\|_{L^\infty(\mathcal{H}_1)}+\|F\|_{L^\infty(\mathcal{H}_1)}+[g]_{C^\beta(\mathcal{Q}_1\cap\gamma_-)}\right).
    \end{align} 
     Since $f$ is a bounded weak solution with boundary data $\alpha \mathcal{R}_bf+(1-\alpha)g$ at $\gamma_-$ by \autoref{lem.bdd}, we can apply \cite[Lemma 8.4]{Sil22}. Note that even though the lemma is given in one-sided cylinders, straightforward modifications of the proof make the lemma true in two-sided cylinders (see the discussion in \cite[Section 2.8]{RoWe25}). Before applying the lemma, we observe that $\widetilde{f}\coloneqq\frac{f}{\sup_{\mathcal{H}_1}f-\sup_{\gamma_-\cap \mathcal{Q}_1}f+\eps_0^{-1}\|F\|_{L^\infty(\mathcal{H}_1)}}$ is a solution to \eqref{oscillation:eq.bdry} with $z_0$, $r$ and $F$ replaced by $0$, $1$ and $\widetilde{F}\coloneqq\frac{F}{\sup_{\mathcal{H}_1}f-\sup_{\gamma_-\cap \mathcal{Q}_1}f+\eps_0^{-1}\|F\|_{L^\infty(\mathcal{H}_1)}}$, where the small constant $\eps_0=\eps_0(n,\Lambda)\in(0,1)$ is determined in \cite[Lemma 8.4]{Sil22}. Then we have $\widetilde{f}-\sup_{\mathcal{Q}_1\cap\gamma_-}\widetilde{f}\leq 1$ in $\mathcal{Q}_1$ and $\|\widetilde{F}\|_{L^\infty(\mathcal{H}_1)}\leq \eps_0$. Thus, there is a small constant $\theta=\theta(n,\Lambda)$ such that for any $z\in \mathcal{H}_{\frac12}$, it holds
    \begin{align*}
        \frac{f(z)-\sup_{\gamma_-\cap \mathcal{Q}_1}f}{\sup_{\mathcal{H}_1}f-\sup_{\gamma_-\cap \mathcal{Q}_1}f+\eps_0^{-1}\|F\|_{L^\infty(\mathcal{H}_1)}}=\widetilde{f}-\sup_{\mathcal{Q}_1\cap\gamma_-}\widetilde{f}\leq 1-\theta.
    \end{align*}
    Similarly, considering $-f$ instead of $f$, we have for any $z\in \mathcal{H}_{\frac12}$,
    \begin{align*}
        \frac{-f(z)+\inf_{\gamma_-\cap \mathcal{Q}_1}f}{-\inf_{\mathcal{H}_1}f+\inf_{\gamma_-\cap \mathcal{Q}_1}f+\eps_0^{-1}\|F\|_{L^\infty(\mathcal{H}_1)}}=\frac{-f(z)-\sup_{\gamma_-\cap \mathcal{Q}_1}(-f)}{\sup_{\mathcal{H}_1}(-f)-\sup_{\gamma_-\cap \mathcal{Q}_1}(-f)+\eps_0^{-1}\|F\|_{L^\infty(\mathcal{H}_1)}}\leq 1-\theta.
    \end{align*}
    Combining the previous two inequalities leads to
    \begin{align*}
        \osc_{\mathcal{H}_{\frac12}}f&\leq (1-\theta)(\osc_{\mathcal{H}_1}f-\osc_{\mathcal{Q}_{1}\cap\gamma_-}f)+\osc_{\mathcal{Q}_1\cap\gamma_-}f+2(1-\theta)\eps_0^{-1}\|F\|_{L^\infty(\mathcal{H}_1)}\\
        &\leq (1-\theta)\osc_{\mathcal{H}_1}f+\theta\osc_{\mathcal{Q}_1\cap\gamma_-}f+2(1-\theta)\eps_0^{-1}\|F\|_{L^\infty(\mathcal{H}_1)}.
    \end{align*}
    Now, we use the boundary condition to observe that
    \begin{align*}
        \osc_{\mathcal{Q}_1\cap\gamma_-}f\leq \osc_{\mathcal{Q}_1\cap\gamma_-}\alpha\mathcal{R}_bf+\osc_{\mathcal{Q}_1\cap\gamma_-}(1-\alpha)g\leq \alpha\osc_{\mathcal{Q}_{1}\cap\gamma_+} f+(1-\alpha)\osc_{\mathcal{Q}_1\cap\gamma_-}g,
    \end{align*}
    where we have used the fact that $\osc\limits_{\mathcal{Q}_1\cap\gamma_-}f(t,x',0,v',-bv_n)\leq \osc\limits_{\mathcal{Q}_1\cap\gamma_+}f(t,x',0,v',v_n)$, as $b\in[0,1]$.
    This implies
    \begin{align*}
         \osc_{\mathcal{H}_{\frac12}}f\leq (1-\theta+\theta\alpha)\osc_{\mathcal{H}_1}f+\theta(1-\alpha)\osc_{\mathcal{Q}_1\cap\gamma_-}g+2(1-\theta)\eps_0^{-1}\|F\|_{L^\infty(\mathcal{H}_1)}.
    \end{align*}
    By the scaling argument together with the fact that $1-\alpha\leq1$, we deduce for any $k\geq1$, 
    \begin{align}\label{holsmall:ineq1.osck}
        \osc_{\mathcal{H}_{{2^{-k}}}}f\leq (1-\theta+\theta\alpha)\osc_{\cH_{{2^{-k+1}}}}f+\theta\osc_{\mathcal{Q}_{{2^{-k+1}}}\cap\gamma_-}g+2(1-\theta)\eps_0^{-1}2^{-2(k-1)}\|F\|_{L^\infty(\cH_{2^{-(k-1)}})}.
    \end{align}
    Then using this together with the inductive argument, we derive
    \begin{align*}
        \osc_{\mathcal{H}_{{2^{-k}}}}f&\leq (1-\theta+\alpha\theta)^{k}\osc_{\mathcal{H}_{1}}f+\sum_{i=1}^{k}(1-\theta+\theta\alpha)^{i-1}\left[\theta\osc_{\mathcal{Q}_{2^{-(k-i)}} \cap\gamma_-}g+c2^{-2(k-i)}\|F\|_{L^\infty(\mathcal{H}_{2^{-(k-i)}})}\right]\\
        &\leq (1-\theta+\alpha\theta)^k\osc_{\mathcal{H}_1}f+c\sum_{i=1}^{k}(1-\theta+\theta\alpha)^{i-1}\theta2^{-(k-i)\beta}([g]_{C^\beta(\mathcal{Q}_1\cap\gamma_-)}+\|F\|_{L^\infty(\mathcal{H}_1)})\\
        &\leq c2^{-k\beta}(\|f\|_{L^\infty(\mathcal{H}_1)}+[g]_{C^\beta(\mathcal{Q}_1\cap\gamma_-)}+\|F\|_{L^\infty(\mathcal{H}_1)})
    \end{align*}
    by choosing a small $\beta=\beta(n,\Lambda)$ such that $2^{-\beta}=\frac{2-\theta+\alpha\theta}2\in (1-\theta+\alpha\theta,1)$, where $c=c(n,\Lambda,\alpha)$. This verifies \eqref{holsmall:ineq0.osc}. In particular, we observe that the constant $\beta\to0$ and $c\to\infty$ as $\alpha\to1$, and $\beta$ is positive and $c$ is bounded as $\alpha\to0$.
\end{proof}

Now we combine \autoref{lem.oscillation} and \autoref{lem.calpha} together with a covering argument to prove the boundary H\"older estimate. Since the reflection operator
\begin{lemma}\label{lem.holsmall}
   Let $z_0\in\gamma_0$ and $R\leq1$.  Let $f$ be a weak solution to 
    \begin{equation*}
\left\{
\begin{alignedat}{3}
\partial_tf+v\cdot\nabla_xf-\ddiv(A\nabla_vf)&=B\cdot\nabla_vf+F&&\qquad \mbox{in  $\mathcal{H}_R(z_0)$}, \\
f&=\alpha \mathcal{R}_bf+(1-\alpha)g&&\qquad  \mbox{in $\mathcal{Q}_R(z_0)\cap\gamma_-$},
\end{alignedat} \right.
\end{equation*}
for some $\alpha\in[0,1)$, $b\in(0,1]$ with $\alpha/b^2\in[0,1]$. Then there is $\beta=\beta(n,\Lambda,\alpha)$ such that
\begin{align}\label{holsmall:est.hol}
    R^{\beta}[f]_{C^\beta(H_{R/2}(z_0))}\leq c\left(R^{-4n-2}\|f\|_{L^1(\mathcal{H}_R(z_0))}+R^2\|F\|_{L^\infty(\mathcal{H}_R(z_0))}+R^{\beta}[g]_{C^\beta(\mathcal{Q}_R(z_0)\cap\gamma_-)}\right)
\end{align}
for some constant $c=c(n,\Lambda,\alpha,b)$. Moreover, we have the following limit behavior of the constants $\beta$ and $c$: $\lim\limits_{\alpha\to0}\beta>0$, $\lim\limits_{\alpha\to1}\beta=0$, $\lim\limits_{\alpha\to0}c<\infty$, $\lim\limits_{b\to0}c=\lim\limits_{\alpha\to1}c=\infty$.
\end{lemma}

\begin{proof}
    First, we may assume $z_0=0$ and $R=1$, as $z_0\in\gamma_0$. At the grazing set that $f$ is now well-defined by \eqref{holsmall:ineq0.osc}. Next, note that $\alpha\mathcal{R}_bf+(1-\alpha)g$ is H\"older continuous away from $\gamma_0$ by the fact that $f$ is H\"older continuous in $\gamma_+$. Thus, the pointwise value of $f(z)$ is well defined in $\overline{\mathcal{H}_{1/2}}$. Now, using a covering argument together with \eqref{holsmall:ineq0.osc} and H\"older estimates away from $\gamma_0$, we are going to prove for any $z_1\in \mathcal{H}_{\frac12}$ and $\rho\leq\frac1{2^{10}}$, 
    \begin{align}\label{holsmall:des2}
        \|f-f(z_1)\|_{L^\infty(\mathcal{H}_{\rho}(z_1))}\leq c\rho^{\beta}\left(\|f\|_{L^\infty(\mathcal{H}_1)}+\|F\|_{L^\infty(\mathcal{H}_1)}+[g]_{C^\beta(\mathcal{Q}_1\cap\gamma_-)}\right),
    \end{align}
    where $c=c(n,\Lambda,\alpha,b)$. In particular, we divide the proof into three steps. In the first two steps, we will prove \eqref{holsmall:des2} in the region $\{x_n\geq (v_{1,n}/4)^3\}$.  In the final step, we use these estimates to derive oscillation estimates near $\gamma_-$. This is more delicate because the boundary contribution involves $\mathcal{R}_bf$, and its treatment relies on the regularity of the solution on $\gamma_+$.
    \begin{itemize}
        \item Let $x_{1,n}\geq |v_{1,n}/4|^3$. There is a constant $\rho_1\eqsim( x_{1,n})^{\frac13}$ such that 
        \begin{align}\label{holsmall:proj0}
            \mathcal{H}_{\rho_1/8}(z_1)\cap \gamma=\emptyset,\quad\mathcal{H}_{\rho_1}(z_1)\cap \gamma_0\neq\emptyset,\quad \mathcal{H}_{\rho_1}(z_1)\subset \mathcal{H}_{8\rho_1}(\overline{z_1}),
        \end{align} where $\overline{z_1}\coloneqq (t_1,x_1',0,v_1',0)$ is the natural projection of $z_1$ onto $\gamma_0$.  Using \autoref{lem.calpha} with $f$ replaced by $f-f(\overline{z}_1)$ and \eqref{holsmall:proj0}, we have for any $\rho\leq \frac{\rho_1}{32}$,
        \begin{equation*}
        \begin{aligned}
            \|f-f(z_1)\|_{L^\infty(\mathcal{H}_\rho(z_1))}&\leq c\left(\frac{\rho}{\rho_1}\right)^{\beta}\left(\|f-f(\overline{z_1})\|_{L^\infty(\mathcal{H}_{\rho_1/(16)}({z_1}))}+\rho_1^2\|F\|_{L^\infty(\mathcal{H}_{\rho_1/(16)}(z_1))}\right)\\
            & \leq c\left(\left(\frac{\rho}{\rho_1}\right)^{\beta}\|f-f(\overline{z_1})\|_{L^\infty(\mathcal{H}_{8\rho_1}({\overline{z_1}}))}+\rho^{\beta}\|F\|_{L^\infty(\mathcal{H}_{1})}\right)
        \end{aligned}
        \end{equation*}
        where $c=c(n,\Lambda)$ and we take $\beta$ to be less than the H\"older exponent $\beta_0=\beta_0(n,\Lambda)$ determined in \autoref{lem.calpha}. Combining this and \eqref{holsmall:ineq0.osc} together with the fact that $\overline{z_1}\in \gamma_0$, we have the desired estimate \eqref{holsmall:des2}. On the other hand, when $\rho>\frac{\rho_1}{32}$, then we see $\mathcal{H}_{\rho}(z_1)\subset\mathcal{H}_{2^8\rho}(\overline{z_1})$. Thus, by directly applying \eqref{holsmall:ineq0.osc} with $0$ replaced by $\overline{z_1}\in\gamma_0$, we get
        \begin{align*}
            \|f-f(z_1)\|_{L^\infty(\mathcal{H}_{\rho}(z_1))}\leq \|f-f(z_1)\|_{L^\infty(\mathcal{H}_{2^8\rho}(\overline{z_1}))}&\leq \osc_{\mathcal{H}_{2^8\rho}(\overline{z_1})}f\\
            &\leq c\rho^\beta\left(\|f\|_{L^\infty(\mathcal{H}_1)}+\|F\|_{L^\infty(\mathcal{H}_1)}+[g]_{C^\beta(\mathcal{Q}_1\cap\gamma_-)}\right).
        \end{align*}
        Thus, we also have \eqref{holsmall:des2} in this case. 
        \item Let $x_{1,n}\leq -(v_{1,n}/4)^3$. Let $\widetilde{z_1}\coloneqq z_1\circ (-\frac{x_{1,n}}{v_{1,n}},0,0)$ and $\overline{z_1}\coloneqq (\widetilde{t_1},\widetilde{x_1}',0,\widetilde{v_1}',0)$, which are the natural projections of $z_1$ onto $\gamma_+$ and of $\widetilde{z_1}$ onto $\gamma_0$, respectively. In particular, we have for $r_1\eqsim \left(\frac{x_{1,n}}{|{v_{1,n}}|}\right)^{\frac12}$ and $\rho_1\eqsim |v_{1,n}|$,
        \begin{equation}\label{holsmall:proj+}
        \begin{aligned}
            &\mathcal{H}_{r_1/8}(z_1)\cap \gamma=\emptyset,\quad \mathcal{H}_{r_1}(z_1)\cap \gamma_+\neq\emptyset,\quad \mathcal{H}_{r_1}(z_1)\subset \mathcal{H}_{8r_1}(\widetilde{z_1}),\\
            &\mathcal{H}_{\rho_1/8}(\widetilde{z_1})\cap\gamma_0=\emptyset,\quad \mathcal{H}_{\rho_1}(\widetilde{z_1})\cap\gamma_0\neq\emptyset,\quad \mathcal{H}_{\rho_1}(\widetilde{z_1})\subset\mathcal{H}_{8\rho_1}(\overline{z_1}).
            \end{aligned}
            \end{equation}
        If $\rho\leq \frac{r_1}{32}$, using \autoref{lem.calpha} with $f$ replaced by $f-f(\widetilde{z_1})$ leads to
        \begin{align}\label{holsmall:ineq4}
            \|f-f(z_1)\|_{L^\infty(\mathcal{H}_\rho(z_1))}
            & \leq c\left(\left(\frac{\rho}{r_1}\right)^{\beta}\|f-f(\widetilde{z_1})\|_{L^\infty(\mathcal{H}_{r_1/(16)}({\widetilde{z_1}}))}+\rho^{\beta}\|F\|_{L^\infty(\mathcal{H}_{1})}\right),
        \end{align}
        where $c=c(n,\Lambda)$. Similarly, we have for any $\rho\leq \frac{\rho_1}{32}$,
    \begin{align*}
            \|f-f(\widetilde{z_1})\|_{L^\infty(\mathcal{H}_\rho(\widetilde{z_1}))}
            & \leq c\left(\left(\frac{\rho}{\rho_1}\right)^{\beta}\|f-f(\overline{z_1})\|_{L^\infty(\mathcal{H}_{\rho_1/(16)}({\overline{z_1}}))}+\rho^{\beta}\|F\|_{L^\infty(\mathcal{H}_{1})}\right)
        \end{align*}
        for some constant $c=c(n,\Lambda)$. Thus, by appropriately combining the previous two inequalities with \eqref{holsmall:ineq0.osc} and \eqref{holsmall:proj+}, according to the size of the radius $\rho$, we can deduce \eqref{holsmall:des2} in this case as well. So far we have proved \eqref{holsmall:des2} for any $z\in \mathcal{H}_{\frac12}\cap \{x_{1,n}\geq (v_{1,n}/4)^3\}$. This implies that for any $\mathcal{H}_r(z_0)$ with $\mathcal{H}_{2r}(z_0)\subset \{v_n<0\}$, we have 
        \begin{align}\label{holsmall:gamma+.stable}
            [f]_{C^{\beta}(\mathcal{H}_r(z_0))}\leq c\left(\|f\|_{L^\infty(\mathcal{H}_1)}+\|F\|_{L^\infty(\mathcal{H}_1)}\right)
        \end{align}
        for some constant $c=c(n,\Lambda)$. We will use this estimate to get the regularity of $\mathcal{R}_bf$.
        \item $x_{1,n}\geq (v_{1,n}/4)^3$. Here, we use the same cylinders as in \eqref{holsmall:proj+}. If $\rho\leq \frac{r_1}{32}$, as in \eqref{holsmall:ineq4}, we get
        \begin{align}\label{holsmall:ineq5}
            \|f-f(z_1)\|_{L^\infty(\mathcal{H}_\rho(z_1))}
            & \leq c\left(\left(\frac{\rho}{r_1}\right)^{\beta}\|f-f(\widetilde{z_1})\|_{L^\infty(\mathcal{H}_{r_1/(16)}({\widetilde{z_1}}))}+\rho^{\beta}\|F\|_{L^\infty(\mathcal{H}_{1})}\right)
        \end{align}
        for some constant $c=c(n,\Lambda)$. In particular we have used interior regularity estimates. Next, when $\rho\leq\frac{\rho_1}{32}$, we use \autoref{lem.calpha} with $f$ replaced by $f-f(\overline{z_1})$ to see that
       \begin{align*}
            \|f-f(\widetilde{z_1})\|_{L^\infty(\mathcal{H}_\rho(\widetilde{z_1}))}
            & \leq c\left(\frac{\rho}{\rho_1}\right)^{\beta}\|f-f(\overline{z_1})\|_{L^\infty(\mathcal{H}_{\rho_1/(16)}({\overline{z_1}}))}\\
            &\quad+c\rho^{\beta}\left(\|F\|_{L^\infty(\mathcal{H}_{1})}+[\alpha\mathcal{R}_bf+(1-\alpha)g]_{C^{\beta}(\gamma_-\cap\mathcal{Q}_{\rho_1/(16)}(\widetilde{z_1}))}\right),
        \end{align*}
        where $c=c(n,\Lambda)$. We are now going to estimate the boundary term $\mathcal{R}_bf$. First, by \autoref{lem.geo}, we have
        \begin{align*}
            [\mathcal{R}_bf]_{C^{\beta_0}(\gamma_-\cap \mathcal{Q}_{\rho_1/(16)}(\widetilde{z}_1))}\leq  c[f]_{C^{\beta_0}(\gamma_-\cap \mathcal{Q}_{\rho_1/(16),b}(\mathcal{R}_b\widetilde{z}_1))}
        \end{align*}
        for some constant $c=c(n)$, where the set $\mathcal{Q}_{\frac{\rho_1}{16},b}(z)$ is given in \autoref{lem.geo}. In addition, for any $z\in \gamma_-\cap \mathcal{Q}_{\frac{\rho_1}{16},b}(\mathcal{R}_b\widetilde{z}_1)$, it holds $v_n<-\frac{b\rho_1}{16}$ by the fourth condition given in \eqref{holsmall:proj+}. Hence, there is a small constant $m=m(b)$ such that for any $z_0\in \gamma_-\cap \mathcal{Q}_{\frac{\rho_1}{16},b}(\mathcal{R}_b\widetilde{z}_1)$, 
        \begin{align*}
            \mathcal{H}_{2m\rho_1}(z_0)\subset \{v_n<0\}.
        \end{align*}
        Now using \eqref{holsmall:gamma+.stable} and a standard covering argument, we have 
        \begin{align*}
             [\mathcal{R}_bf]_{C^{\beta}(\gamma_-\cap \mathcal{Q}_{\rho_1/(16)}(\widetilde{z}_1))}\leq c\left(\|f\|_{L^\infty(\mathcal{H}_1)}+\|F\|_{L^\infty(\mathcal{H}_1)}\right)
        \end{align*}
        for some constant $c=c(n,\Lambda,b)$, where $c$ tends to $\infty$ as $b\to0$. Thus we have
        \begin{align*}
            \|f-f(\widetilde{z_1})\|_{L^\infty(\mathcal{H}_\rho(\widetilde{z_1}))}
            & \leq c\left(\frac{\rho}{\rho_1}\right)^{\beta}\|f-f(\overline{z_1})\|_{L^\infty(\mathcal{H}_{\rho_1/(16)}({\overline{z_1}}))}\\
            &\quad+c\rho^{\beta}\left(\|f\|_{L^\infty(\mathcal{H}_1)}+\|F\|_{L^\infty(\mathcal{H}_{1})}+[g]_{C^{\beta}(\gamma_-\cap\mathcal{Q}_{1}(\widetilde{z_1}))}\right).
        \end{align*}
        Combination this estimate with \eqref{holsmall:ineq5} and \eqref{holsmall:ineq0.osc}, we deduce \eqref{holsmall:des2} when $x_{1,n}\geq (v_{1/n}/4)^3$. 
     \end{itemize}
Finally, as in \autoref{lem.oscillation}, the constant $c$ in \eqref{holsmall:des2} tends to $\infty$ either $b\to0$ or $\alpha\to1$. The desired H\"older estimates immediately follows from \eqref{holsmall:des2} after using \autoref{lem.bdd}, which completes the proof.
\end{proof}

\subsection{Local boundedness for bounce-back}
\label{subsec:bounce-back}

In this subsection, we prove the local boundedness of the solution  near the grazing set when the boundary satisfies the bounce-back condition. 

First, for any $z_0\in\bbR^{1+2n}$ and $R>0$, we introduce two cylinders 
\begin{align*}
    \mathcal{Q}^-_R(z_0)\coloneqq \{(t,x,-v)\,:\,z\in \mathcal{Q}_R(z_0)\}\quad\text{and}\quad \mathcal{H}^-_R(z_0)\coloneqq \{(t,x,-v)\,:\,z\in \mathcal{H}_R(z_0)\},
\end{align*}
which will be used to describe the boundary behavior of a solution at $\gamma_-$. In particular, we write
\begin{align*}
    S(\mathcal{Q}_R(z_0))\coloneqq \mathcal{Q}_R(z_0)\cup \mathcal{Q}^-_R(z_0)\quad\text{and}\quad S(\mathcal{H}_R(z_0))\coloneqq \mathcal{H}_R(z_0)\cup \mathcal{H}^-_R(z_0).
\end{align*}

Now, we state the main result of this subsection.
\begin{theorem}\label{thm.bounce}
   Let $z_0\in\gamma_0$ and let $f$ be a weak solution to 
    \begin{equation}\label{bounce:eq}
\left\{
\begin{alignedat}{3}
\partial_tf+v\cdot\nabla_xf-\ddiv(A\nabla_vf)&=B\cdot\nabla_v+F&&\qquad \mbox{in  $S(\mathcal{H}_1(z_0)) $}, \\
f(t,x',0,v)&=f(t,x',0,-v)&&\qquad  \mbox{in $\gamma_-\cap S(\mathcal{Q}_1(z_0))$}.
\end{alignedat} \right.
\end{equation}
Then we have
\begin{align*}
    \|f\|_{L^{\infty}(S(\mathcal{H}_{1/2}(z_0)))}\leq c\left(\|f\|_{L^1(S(\mathcal{H}_1(z_0)))}+\|F\|_{L^{\infty}(S(\mathcal{H}_1(z_0)))}\right)
\end{align*}
for some constant $c=c(n,\Lambda,|v_0|)$.
\end{theorem}
First, we give some comments about the main result.
\begin{itemize}
    \item At the boundary $\mathcal{Q}_R(z_0)\cap \gamma_-$, the value of the solution $f$ is prescribed through its value on $\mathcal{Q}_R^-(z_0)\cap\gamma_+$. In general, the latter set is not contained in  $\mathcal{Q}_R(z_0)$, as the velocity $v_0'$ and its reflection $-v_0'$ may be widely separated. Thus, we introduce two sets $S(\mathcal{Q}_R(z_0))$ and $S(\mathcal{H}_R(z_0))$, to properly define the boundary value.
    \item All arguments to establish the local boundedness are the same as in the proof of \autoref{lem.bdd}, except for the proof of \eqref{extension1:ineq.bdry.weightv} in \autoref{lem.extension1}. More precisely, in the proof of \autoref{lem.extension1} (see the estimate of $J_1$), the boundary condition is used to determine the sign of a boundary integral. However, if one works only with the cylinder $\mathcal{Q}_R(z_0)$, this argument breaks down, as the set $\mathcal{Q}_R^-(z_0)\cap\gamma_+$ is not necessarily contained in  $\mathcal{Q}_R(z_0)$. To overcome this, we instead work with the union of the kinetic cylinder and its reflection $S(\mathcal{H}_R(z_0))$. As a result, this creates some dependence $|v_0|$ in the energy estimates, as $\mathcal{H}^-_R(z_0)$ is not the usual kinetic cylinder. 
\end{itemize}
First, we give the lemma corresponding to \autoref{lem.extension1}.
\begin{lemma}\label{lem.extension1.bounce}
Let $\frac{1}2\leq r<\rho\leq 1$ and let $f$ be a solution to \eqref{bounce:eq}. Then we have for any $k\geq0$,
    \begin{equation}\label{extension1.bounce:ineq.bdry.weightv}
    \begin{aligned}
        &\int_{S(\mathcal{Q}_{r}(z_0))\cap \gamma}(f-k)_+^2{|v\cdot n_x|^2}\,d\Gamma+\int_{S(\mathcal{H}_r(z_0))}|\nabla_v(f-k)_+|^2\,dz\\
        &\leq c\left(\int_{S(\mathcal{H}_{\rho}(z_0))}\frac{|(f-k)_+|^2}{(\rho-r)^2}+|F\mbox{\large$\chi$}_{f>k}|^2\,dz\right)
    \end{aligned}
    \end{equation}
    for some constant $c=c(n,\Lambda,|v_0|)$.

\end{lemma}
\begin{proof}
We may assume $t_0=x_0=0$. Let $\psi\in C_c^\infty(Q_{\frac{r+3\rho}4})$ with $\psi\equiv 1$ on $Q_{\frac{r+\rho}2}$ and $\psi(t,x,v)=\psi(t,x,-v)$. Next we define
\begin{align*}
    \widetilde{\psi}(t,x,v)\coloneqq \psi(t,x-tv_0,v-v_0)+\psi(t,x-tv_0,v+v_0)
\end{align*}
to see that $\psi\in C_c^\infty(S(\mathcal{Q}_{\frac{r+3\rho}4}(z_0)))$ with $\psi\equiv 1$ in $S(\mathcal{Q}_{\frac{r+\rho}2}(z_0))$ and $\widetilde{\psi}(t,x,v)=\widetilde{\psi}(t,x,-v)$. In addition, we observe 
\begin{align}\label{extension1.bounce:tildepsi}
    |(\partial_t+v\cdot\nabla_x)\widetilde{\psi}|\leq \|\partial_t\psi\|_{L^\infty}+c(1+|v_0|)\|\partial_x\psi\|_{L^\infty}\leq \frac{c|v_0|}{(\rho-r)}.
\end{align}
As in \autoref{lem.changeofvar}, we deduce \eqref{extension1:ene.ineq.grad} with $\psi$, $\mathcal{Q}_{\frac{r+3\rho}4}$, and $\mathcal{H}_{\frac{r+3\rho}4}$ replaced by $\widetilde{\psi}$, $S(\mathcal{Q}_{\frac{r+3\rho}4})(z_0)$, and $S(\mathcal{H}_{\frac{r+3\rho}4})(z_0)$, respectively. First, we use the boundary condition and the fact that $\widetilde{\psi}(t,x,v)=\widetilde{\psi}(t,x,-v)$ to see that 
\begin{align*}
    \int_{\gamma\cap S(\mathcal{Q}_{\frac{r+3\rho}4}(z_0))}\varphi_M(f)(\widetilde{\psi})^2(n_x\cdot v)\,d\Gamma&=\int_{\gamma\cap \mathcal{Q}_{\frac{r+3\rho}4}(z_0)}\varphi_M(f)(\widetilde{\psi})^2(v_n)\,d\Gamma\\
    &\quad+\int_{\gamma\cap \mathcal{Q}^-_{\frac{r+3\rho}4}(z_0)}\varphi_M(f)(\widetilde{\psi})^2(v_n)\,d\Gamma=0,
\end{align*}
For the remainder terms $J_2-J_5$ in \eqref{extension1:ene.ineq.grad}, we follow the same lines as in \eqref{extension1:ineq3} together with \eqref{extension1.bounce:tildepsi}, to derive
\begin{align*}
    \int_{S(\mathcal{H}_r(z_0))}|\nabla_v(f-k)_+|^2\,dz\leq c\left(\int_{S(\mathcal{H}_\rho(z_0))}\frac{|(f-k)_+|^2}{(\rho-r)^2}\,dz+|F\mbox{\large$\chi$}_{f>k}|^2\right),
\end{align*}
where $c=c(n,\Lambda,|v_0|)$. To get the weighted $L^2$-estimate at the boundary $\gamma$, we choose 
\begin{align}\label{defn.widephi}
    \widetilde{\phi}(t,x,v)\coloneqq  \phi(t,x-tv_0,v-v_0)+\phi(t,x-tv_0,v+v_0),
\end{align}
where the function $\phi$ is determined in \eqref{extension1:ineq.energy1}. Then, repeating the argument of \eqref{extension1:ineq01} together with $|(\partial_t+v\cdot\nabla_x)\widetilde{\phi}|\leq \frac{c|v_0|}{\rho-r}$, we deduce
\begin{align*}
    \int_{\gamma_+\cap S(\mathcal{Q}_{\frac{r+\rho}{2}}(z_0))}(f-k)^2_+(\widetilde{\psi})^2(n_x\cdot v)_+\leq c\left(\int_{S(\mathcal{H}_{\frac{r+3\rho}4}(z_0))}\frac{(f-k)_+^2}{(\rho-r)^2}+|F\mbox{\large$\chi$}_{f>k}|^2\right),
\end{align*}
where $c=c(n,\Lambda,|v_0|)$. Now using the boundary condition, we have
\begin{align*}
    \int_{\gamma\cap S(\mathcal{Q}_{{r}}(z_0))}(f-k)^2_+|n_x\cdot v|\leq c\left(\int_{S(\mathcal{H}_{\frac{r+3\rho}4}(z_0))}\frac{(f-k)_+^2}{(\rho-r)^2}+|F\mbox{\large$\chi$}_{f>k}|^2\right)
\end{align*}
for some constant $c=c(n,\Lambda,|v_0|)$. Altogether, we deduce the desired estimate.
\end{proof}
Now we are ready to prove our main result in this subsection.
\begin{proof}[Proof of \autoref{thm.bounce}] 
    As in the first part of \autoref{lem.extension1} and \autoref{lem.q.int}, we have that $(f-k)_+$ is a subsolution to 
    \begin{align*}
        \partial_t(f-k)_++v\cdot\nabla_x(f-k)_+-\Delta_v(f-k)_+=F_1+\nabla_v\cdot F_2\quad\text{in }S(\mathcal{H}_{\frac{r+\rho}2}(z_0)),
    \end{align*}
    where 
    \begin{align*}
        \|F_1\|^2_{L^2(S(\mathcal{H}_{\frac{r+\rho}2}(z_0))}+\|F_2\|^2_{L^2(S(\mathcal{H}_{\frac{r+\rho}2}(z_0))}&\leq c\int_{S(\mathcal{H}_{\frac{r+\rho}2}(z_0)}\left(|\nabla_v(f-k)_+|^2+|F\mbox{\large$\chi$}_{f>k}|^2\right)\\
        &\leq c\left(\int_{S(\mathcal{H}_{\frac{r+3\rho}4}(z_0))}\frac{(f-k)_+^2}{(\rho-r)^2}+|F\mbox{\large$\chi$}_{f>k}|^2\right),
    \end{align*}
    for some constant $c=c(n,\Lambda,|v_0|)$, where we have used \eqref{extension1.bounce:ineq.bdry.weightv} for the last inequality. Now we take the cut-off function  $\widetilde{\phi}$ defined in \eqref{defn.widephi} instead of  the function $\phi$ given in the proof of \autoref{lem.extension}. By repeating all the arguments in \autoref{lem.extension} with $\phi$ replaced by $\widetilde{\phi}$ and using \eqref{extension1.bounce:ineq.bdry.weightv}, we get for any $\frac12\leq r<\rho\leq1$,
    \begin{align*}
        \left(\int_{S(\mathcal{H}_{r}(z_0))}({f}-k)_+^q\,dz\right)^{\frac{2}q}\leq c\left(\int_{S(\mathcal{H}_\rho(z_0))}\frac{({f}-k)_+^2}{(\rho-r)^2}+|F\mbox{\large$\chi$}_{f>k}|^2\,dz\right)
    \end{align*}
    for some constants $q>2$ and $c=c(n,\Lambda,|v_0|)$. Thus, by the standard De-Giorgi's
iteration, we get the desired $L^\infty$-estimate.
\end{proof}

\section{Liouville theorems in the half-space}
\label{sec:Liouville}

In this section, we prove a 1D Liouville theorem for solutions subject to $\alpha$-reflection-type boundary conditions. 

To this end, for any $\alpha\in(0,1)$ and $b\in(0,1]$, we first construct an explicit  homogeneous solution $\phi_{\alpha,b}$ to the equation
\begin{equation}\label{eq.1d.alphab}
\left\{
\begin{alignedat}{3}
v\partial_x\phi_{\alpha,b}-\partial_{vv}\phi_{\alpha,b}&=0&&\qquad \mbox{in  $\{x>0\}$}, \\
\phi_{\alpha,b}(0,v)&=\alpha \phi_{\alpha,b}(0,-bv)&&\qquad  \mbox{in $\{x=0\}\times \{v>0\}$}.
\end{alignedat} \right.
\end{equation}
Next, we prove that any solution $h$ to a localized equation of \eqref{eq.1d.alphab} in ${H}_1$ is comparable to $\phi_{\alpha,b}$, in the sense that $h/\phi_{\alpha,b}$ enjoys suitable oscillation estimates in ${H}_1$. Using this, we establish that any global solution $h$ of \eqref{eq.1d.alphab} with a suitable growth at infinity agrees with $\phi_{\alpha,b}$ up to a multiplicative constant. This proves a Liouville theorem in the half-space (see \autoref{lem.liou.nd}).

\subsection{1D solutions with alpha-reflection condition}

First, we provide an explicit representation of the function $\phi_{\alpha,b}$ and its properties.

\begin{lemma}\label{lem.phialpha}
    For any $\alpha\in(0,1)$ and $b\in(0,1]$, define
     \begin{align}\label{phialpha:defn.phialpha}
    \phi_{\alpha,b}(x,v)\coloneqq x^{\frac{\lambda_{\alpha,b}}3}\left(M\left(-\frac{\lambda_{\alpha,b}}{3},\frac23,-\frac{v^3}{9x}\right)+\frac{\Gamma(\frac23)\Gamma(\frac{1-\lambda_{\alpha,b}}3)}{\Gamma(-\frac{\lambda_{\alpha,b}}3)\Gamma(\frac43)}\frac{v}{(9x)^{\frac13}}M\left(-\frac{\lambda_{\alpha,b}-1}{3},\frac43,-\frac{v^3}{9x}\right)\right),
\end{align}
where $\lambda_{\alpha,b}$ is the constant defined by
\begin{align}\label{phialpha:defn.lambdaalpha}
    \lambda_{\alpha,b}\coloneqq \inf\left\{ \lambda>0\,:\,2\cos\left(\pi\left(\frac{\lambda}{3}+\frac13\right)\right)=\alpha b^{\lambda}\right\}.
\end{align}
Then we have the following:
\begin{itemize}
\item $\lambda_{\alpha,b}\in (0,\frac12)$. In particular, $\lambda_{\alpha,1}\coloneqq \frac{3}{\pi} \arccos(\frac{\alpha}2)-1$.
    \item $\phi_{\alpha,b}>0$ in $\R^2 \setminus \{(0,0)\}$.
    \item $\phi_{\alpha,b}(0,v)=c|v|^{\lambda_\alpha}$ for some constant $c$, when $v\leq0$.
     \item $\phi_{\alpha,b}$ is a weak solution to \eqref{eq.1d.alphab}. In particular $\partial_v\phi_{\alpha,b}\in L^{4+\delta}({H}_R)$ for some small $\delta=\delta(\alpha,b)$ and any $R>1$.
    \item $\phi_{\alpha,b}(x,v)\eqsim \max\{|x|^{\frac13},|v|\}^{\lambda_{\alpha,b}}$. 
    \item For any $R>0$, it holds $\phi_{\alpha,b}\in C^{\lambda_{\alpha,b}}({H}_R)$, but $\phi_{\alpha,b}\notin C^{\lambda_{\alpha,b}+\eps}({H}_R)$ for any $\eps>0$.
\end{itemize} 
\end{lemma}
\begin{proof}
Fix $\alpha\in(0,1)$ and $b\in(0,1]$. First, we prove that the number $\lambda_{\alpha,b}$ is well-defined. We observe that the function $g_1:[0,1]\to \bbR$ given by $g_1(\xi)\coloneqq\alpha b^{\xi}$ is continuous with $g_1(0)=\alpha$ and $g_1(1)=\alpha b$. On the other hand, the function $g_2(\xi)\coloneqq 2\cos(\pi(\frac\xi3+\frac13))$ is continuous with $g_2(0)=1$ and $g_2(\frac12)=0$. Thus, we have $(g_1-g_2)(0)<0<(g_1-g_2)(1)$, which implies that there is at least a positive number $\lambda\in(0,\frac12)$ such that $g_1(\lambda)=g_2(\lambda)$. Thus, $\lambda_{\alpha,b}$ is well-defined and it is in $(0,\frac12)$. Moreover, when $b=1$, we have the explicit formula $\lambda_{\alpha,1}=\frac{3}{\pi}\arccos(\frac{\alpha}2)-1$.

Next, we verify the sign of $\phi_{\alpha,b}$. Indeed, by \cite[(3.34), (3.35)]{HJV14}, we have for any $\lambda\in(0,\frac16)$,
\begin{align*}
   \Lambda(\zeta)\coloneqq \frac{M\left(-\lambda,\frac23,-\zeta^3\right)}{\Gamma(\frac13-\lambda)\Gamma(\frac23)}+\zeta\frac{M(\frac13-\lambda,\frac43,-\zeta^3)}{\Gamma(-\lambda)\Gamma(\frac43)}>0\quad\text{for any $\xi\in\bbR$}
\end{align*}
and 
\begin{align}\label{phialpha:limit.Lambda}
    \Lambda(\zeta)\sim |\zeta|^{3\lambda}\quad\text{as }\zeta\to-\infty\quad\text{and}\quad \Lambda(\zeta)\sim 2\cos\left(\pi\left(\lambda+\frac13\right)\right)\zeta^{3\lambda}\quad\text{as }\zeta\to\infty,
\end{align}
where we write $f(x)\sim g(x)$ as $x\to\infty$ if $\lim\limits_{x\to\infty}f(x)/g(x)=1$. Therefore, using this together with the fact that $\Gamma(\frac13-\lambda)\Gamma(\frac23)>0$, we have $\phi_{\alpha,b}>0$ except at the origin. In addition, by fixing $v$ and taking $x\to0$, we deduce from \eqref{phialpha:limit.Lambda} that
\begin{align}\label{phialpha:bdry.value}
    \phi_{\alpha,b}(0,v)= 2c\cos\left(\pi\left(\frac{\lambda_{\alpha,b}}3+\frac13\right)\right)v^{\lambda_\alpha}\quad\text{in }v\geq0\quad\text{and}\quad \phi_\alpha(0,v)=c|v|^{\lambda_\alpha}\quad\text{in }v\leq0
\end{align}
for some constant $c>0$.  

Now we will prove that $\phi_{\alpha,b}$ is a weak solution to \eqref{eq.1d.alphab}. First, note that $\phi_{\alpha,b}$ is smooth in $x>0$ and is continuous up to the boundary $\{x=0\}\cap \{v>0\}$ by the fact that $M(\cdot,\cdot,-\zeta)$ is analytic and \eqref{phialpha:bdry.value}. Next, from \cite[Claim]{HJV14}, $\phi_{\alpha,b}$ solves $v\partial_x\phi_{\alpha,b}-\partial_{vv}\phi_{\alpha,b}=0$.
Moreover, using \eqref{phialpha:bdry.value} and \eqref{phialpha:defn.lambdaalpha}, we get 
\begin{equation}\label{phialpha:bdry.cond}
    \phi_{\alpha,b}(0,v)=\alpha\phi_{\alpha,b}(0,-bv)\quad\text{in }v\geq0.
\end{equation}
Thus, by \cite[Remark 2.2]{KiWe26}, it suffices to show that $v\partial_x\phi_{\alpha,b},\partial_{v}\phi_{\alpha,b}\in L^2({H}_R)$ for any $R>0$. Due to the equation satisfied by $\phi_{\alpha,b}$, we have $v\partial_x\phi_{\alpha,b}\in L^2(\mathcal{H}_R)$ is equivalent to $\partial_{vv}\phi_{\alpha,b}\in L^2(\mathcal{H}_R)$.

To show this, first we consider the case that $v\geq0$. By differentiating the function \eqref{phialpha:defn.phialpha} together with \cite[13.3.15]{DLMF}, we deduce
\begin{align*}
    |\partial_{vv}\phi_{\alpha,b}|&\lesssim  x^{\frac{\lambda_{\alpha,b}}3-1}|v|M\left(-\frac{\lambda_{\alpha,b}}{3}+1,\frac53,-\frac{v^3}{9x}\right)+x^{\frac{\lambda_{\alpha,b}}3-2}|v|^4M\left(-\frac{\lambda_{\alpha,b}}{3}+2,\frac83,-\frac{v^3}{9x}\right)\\
    &\quad+\frac{v}{x^{\frac13}}\left(x^{\frac{\lambda_{\alpha,b}}3-1}|v|M\left(-\frac{\lambda_{\alpha,b}}{3}+\frac43,\frac73,-\frac{v^3}{9x}\right)+x^{\frac{\lambda_{\alpha,b}}3-2}|v|^4M\left(-\frac{\lambda_{\alpha,b}}{3}+\frac{7}{3},\frac{10}3,-\frac{v^3}{9x}\right)\right)\\
    &\quad+x^{\frac{\lambda_{\alpha,b}}3-\frac13}\frac{|v|^2}{x}M\left(-\frac{\lambda_{\alpha,b}}{3}+\frac43,\frac73,-\frac{v^3}{9x}\right),
\end{align*}
where we write $f(x)\lesssim g(x)$ if $f(x)\leq cg(x)$ for some constant $c$ depending only on $\alpha$ and $b$. When $v\leq0$, we use instead \cite[13.2.42]{DLMF} and \cite[13.2.40]{DLMF}, which leads to the formula 
    \begin{align}\label{phialpha:rep.U}
        \phi_{\alpha,b}(x,v)=mx^{\frac{\lambda_{\alpha,b}}3}U\left(-\frac{\lambda_{\alpha,b}}3,\frac23,-\frac{v^3}{9x}\right)
    \end{align}
    for some constant $m=m(\alpha,b)>0$. Similarly, we differentiate the function  $\phi_{\alpha,b}$ given in \eqref{phialpha:rep.U} together with \cite[13.322]{DLMF} so that 
    \begin{align*}
    |\partial_{vv}\phi_{\alpha,b}|\lesssim x^{\frac{\lambda_{\alpha,b}}3-1}|v|U\left(-\frac{\lambda_{\alpha,b}}{3}+1,\frac53,-\frac{v^3}{9x}\right)+x^{\frac{\lambda_{\alpha,b}}3-2}|v|^4U\left(-\frac{\lambda_{\alpha,b}}{3}+2,\frac83,-\frac{v^3}{9x}\right).
\end{align*}
Using the previous two inequalities and the fact that $U(a,b,z)\leq z^{-a}$ when $z>0$ and $M(a,b,z)\lesssim (-z)^{-a}$ when $z<0$, we have
\begin{equation}\label{phialpha:d2vh}
\begin{aligned}
    &\int_{\{-v^3\geq x\}\cap {H}_R}|\partial_{vv}h|^2+\int_{\{|v|^3\leq x\}\cap {H}_R}|\partial_{vv}h|^2+\int_{\{v^3\geq x\}\cap {H}_R}|\partial_{vv}h|^2\\
    &\lesssim \int_{\{-v^3\geq x\}\cap {H}_R}|v|^{-4+2\lambda_{\alpha,b}}+\int_{\{|v|^3\leq x\}\cap {H}_R}|x|^{\frac{2\lambda_{\alpha,b}-4}3}+\int_{\{v^3\geq x\}\cap {H}_R}|v|^{-4+2\lambda_{\alpha,b}}\leq c(\alpha,b,R),
\end{aligned}
\end{equation}
as $\lambda_{\alpha,b}>0$. We have proved $\partial_{vv}\phi_{\alpha,b}\in L^2({H}_R)$ and this gives $v\partial_x\phi_{\alpha,b}\in L^2({H}_R)$ using the equation. In addition, by differentiating the function $\phi_{\alpha,b}$ using the two representations \eqref{phialpha:defn.phialpha} and \eqref{phialpha:rep.U} from above, we similarly prove $\partial_v\phi_{\alpha,b}\in L^{4+\delta}({H}_R)$ for some small $\delta>0$ as in the proof of $\partial_{vv}\phi_{\alpha,b}\in L^2({H}_R)$. In particular, $\|\partial_v\phi_{\alpha,b}\|_{L^{4+\delta}({H}_1)}\leq c(\alpha,b)$. Thus, $\phi_{\alpha,b}$ is a weak solution, as desired. 

Lastly, we will prove $\phi_{\alpha,b}(x,v)\eqsim \max\{|x|^{\frac13},|v|\}^{\lambda_{\alpha,b}}$. By \eqref{phialpha:bdry.value}, there is a constant $c_{\max}=c_{\max}(\alpha,b)$ such that if $|v|^3>c_{\max}x$, $\phi_{\alpha,b}\eqsim x^{\frac{\lambda_{\alpha,b}}3}(|v^3|/x)^{\frac{\lambda_{\alpha,b}}3}\eqsim |v|^{\lambda_{\alpha,b}}$. On the other hand, using \cite[13.2.13]{DLMF}, when $|v|^3<c_{\min}x$ for some constant $c_{\min}=c_{\min}(\alpha,b)<1$, then we have $\phi_{\alpha,b}(x,v)\eqsim x^{\frac{\lambda_{\alpha,b}}3}$. In addition, for $c_{\min}x\leq |v|^3\leq c_{\max}x$, we have 
\begin{align*}
    M\left(-\frac{\lambda_{\alpha,b}}{3},\frac23,-\frac{v^3}{9x}\right)+\frac{\Gamma(\frac23)\Gamma(\frac{1-\lambda_{\alpha,b}}3)}{\Gamma(-\frac{\lambda_{\alpha,b}}3)\Gamma(\frac43)}\frac{v}{(9x)^{\frac13}}M\left(-\frac{\lambda_{\alpha,b}-1}{3},\frac43,-\frac{v^3}{9x}\right)\eqsim1,
\end{align*}
 which implies $\phi_{\alpha,b}(x,v) \eqsim x^{\frac{\lambda_{\alpha,b}}3}$. Thus, we get $\phi_{\alpha,b}(x,v)\eqsim \max\{|x|^{\frac13},|v|\}^{\lambda_{\alpha,b}}$, which gives the oscillation estimate at the origin. Since $\phi_{\alpha,b}$ is a weak solution, \autoref{lem.schhol.away-} implies that  $\phi_{\alpha,b}$ satisfies oscillation estimates of order  ${\lambda_{\alpha,b}}$ away from the origin. Combining the oscillation estimates at and away from the origin, we conclude that for any $R>0$, $\phi_{\alpha,b}\in C^{\lambda_{\alpha,b}}({H}_R)$. However, as $\phi_{\alpha,b}(x,0)=cx^{\frac{\lambda_{\alpha,b}}3}$, $\phi_{\alpha,b}\notin C^{\lambda_{\alpha,b}+\eps}({H}_R)$ with $\eps>0$. This completes the proof.
\end{proof}

\subsection{A boundary-Harnack type principle in terms of the 1D solution}

Now we prove the comparability between any weak solution $h$ to a localized version of the equation \eqref{eq.1d.alphab} and $\phi_{\alpha,b}$. The precise statement reads as follows:

\begin{proposition}\label{prop.osc}
    
     Let $\alpha\in(0,1)$, $b\in(0,1]$ with $\alpha/b^2\in(0,1]$. Let $h$ be a weak solution to 
    \begin{equation*}
\left\{
\begin{alignedat}{3}
v\partial_xh-\partial_{vv}h&=0&&\qquad \mbox{in  ${H}_{2R}$}, \\
h(0,v)&=\alpha h(0,-bv)&&\qquad  \mbox{in $\{x=0\}\times (0,2R)$}
\end{alignedat} \right.
\end{equation*}
with $R\geq1$.
Then there is a small $\beta=\beta(\alpha,b)$ such that for any $r\leq1$,
\begin{align*}
    \sup_{H_{rR}}\frac{h}{\phi_{\alpha,b}}-\inf_{H_{rR}}\frac{h}{\phi_{\alpha,b}}\leq cr^{\beta}R^{-\lambda_{\alpha,b}}\|h\|_{L^\infty(\mathcal{H}_R)}
\end{align*}
for some constant $c=c(\alpha,b)$, where the constant $\lambda_{\alpha,b}$ is determined in \eqref{phialpha:defn.lambdaalpha}.
\end{proposition}

To get this, first we prove an upper bound of $h/\phi_{\alpha,b}$ by using the following maximum principle for solutions satisfying the $\alpha$-reflection-type boundary condition.

\begin{lemma}\label{lem.maximum}
   Let $\alpha\in(0,1)$, $b\in(0,1]$ with $\alpha/b^2\in(0,1]$. Let $h\in C(\overline{{H}_1})$ be a weak solution to 
    \begin{equation}\label{maximum:eq}
\left\{
\begin{alignedat}{3}
v\partial_xh-\partial_{vv}h&=0&&\qquad \mbox{in  ${H}_1$}, \\
h(0,v)&=\alpha h(0,-bv)&&\qquad  \mbox{in $\{x=0\}\times(0,1)$},\\
h&\leq0&&\qquad\mbox{in $\partial_{\mathrm{kin}}{H}_1\setminus (\{x=0\}\times(0,1))$}.
\end{alignedat} \right.
\end{equation} Assume $\partial_{vv}h\in L^2(H_1)$. Then we have $h\leq0$ in $H_1$.
\end{lemma}

First, recall that $\partial_{\mathrm{kin}}{H}_1 \setminus (\{x=0\}\times(0,1)) = ([0,1] \times \{v=1\}) \cup ([0,1] \times \{v=-1\}) \cup (\{ x = 1\} \times (-1,0))$. 

Next, we point out that the assumption $\partial_{vv}h\in L^2(H_1)$ will always be satisfied later, as we verify in \autoref{lem.upper} below. We expect \autoref{lem.maximum} to remain true also without this additional assumption.

\begin{proof}
    Note that $h$ is smooth in the interior of ${H}_1$ by standard interior regularity results. Since we have assumed $\partial_{vv}h\in L^2(H_1)$, we get $v\partial_xh\in L^2(H_1)$, which ensures that for any $\psi\in L^2(H_1)$, $\int_{H_1}\partial_{vv}h\psi$ and $\int_{H_1}v\partial_{x}h\psi$ are well defined.
    Thus, using the equation, we have 
    \begin{align*}
        \int_{H_1}(v\partial_xh-\partial_{vv}h)h_+=0.
    \end{align*}
    Now by the fundamental theorem of calculus together with the fact that $h_+ = 0$ on $([0,1] \times \{v=1\}) \cup ([0,1] \times \{v=-1\})$, we get
    \begin{align*}
        0=\int_{{H}_1}v\partial_xh h_++\partial_vh\partial_vh_+=\int_{{H}_1}\frac12v\partial_x(h_+)^2+(\partial_vh_+)^2.
    \end{align*}
    Furthermore, using the fact that $h_+=0$ on $\{x=1\}\times (-1,0)$ and $vh_+^2(1,v)\geq0$ when $v\geq0$, we have
    \begin{align*}
        \int_{{H}_1}\frac12v\partial_x(h_+)^2=\int_{-1}^{1}\frac12 v(h_+^2(1,v)-h_+^2(0,v))\,dv\geq -\int_{-1}^{1}\frac{v}{2} h_+^2(0,v)\,dv.
    \end{align*}
    Now we use the boundary condition on $\{x=0\}\times (0,1)$, the change of variables $v \mapsto -b v$, and the fact that $b\in(0,1]$ and $\alpha/b^2\in(0,1]$ to derive
    \begin{align*}
        -\int_{-1}^{1}\frac{v}{2} h_+^2(0,v)\,dv&=-\int_{-1}^{0}\frac{v}{2}h_+^2(0,v) -\int_{0}^{1}\frac{v}{2}(\alpha h_+)^2(0,-bv)\\
        &=-\int_{-1}^{0}\frac{v}{2}h_+^2(0,v) + \int_{-b}^{0}\frac{v}{2}\left(\frac{\alpha}b\right)^2 h_+^2(0,v) \\
        &\geq \left(1 - \left(\frac{\alpha}{b}\right)^2 \right) \int_{-1}^{0}\frac{|v|}{2}h_+^2(0,v)\geq0.
    \end{align*}
    Altogether, we conclude that $\partial_vh_+\equiv 0$ in ${H}_1$, which implies $h_+=0$ as $h_+\coloneqq 0$ on $\partial_{\mathrm{kin}}{H}_1\setminus (\{x=0\}\times(0,1))$. Thus we have $h\leq0$ in ${H}_1$.
\end{proof}
With this lemma at hand, we can prove the following upper bound of $h/\phi_{\alpha,b}$. 
\begin{lemma}\label{lem.upper}
    Let $\alpha\in(0,1)$, $b\in(0,1]$ with $\alpha/b^2\in[0,1]$. Let $h$ be a weak solution to 
    \begin{equation*}
\left\{
\begin{alignedat}{3}
v\partial_xh-\partial_{vv}h&=0&&\qquad \mbox{in  ${H}_1$}, \\
h(0,v)&=\alpha  h(0,-bv)&&\qquad  \mbox{in $\{x=0\}\times (0,1)$}.
\end{alignedat} \right.
\end{equation*}
Then there is a constant $c=c(\alpha,b)$ such that
\begin{align*}
    \sup_{{H}_{1/2}}|h/\phi_{\alpha,b}|\leq c\|h\|_{L^\infty(\mathrm{H}_{1/2})}
\end{align*}
\end{lemma}

\begin{proof}
    Note from \autoref{lem.phialpha} that there is a constant $c=c(\alpha,b)$ such that 
    \begin{align*}
        c\phi_{\alpha,b}\geq 1 \quad \text{on } \partial_{\mathrm{kin}}{H}_{1/2}\setminus\{x=0\}\times (0,(1/2)^3).
    \end{align*}
 Before applying \autoref{lem.maximum}, we will show $\partial_{vv}h\in L^2(H_{1/2})$.
    First, note from \autoref{lem.holsmall} that $h\in C^{\gamma}(\overline{H_{\frac34}})$ for some small $\gamma>0$, which gives 
    \begin{align}\label{upper:ineq1}
        |h(x,v)|\leq c\max\{x^{1/3},|v|\}^{\gamma}.
    \end{align}
    In addition, by \cite[Lemma 3.13, Lemma 3.14]{KiWe26} and \autoref{lem.schhol.away-} when $|v|^3\geq x/8$ and $|v|^3\leq x/8$, respectively, we deduce 
    \begin{align}\label{upper:ineq2}
        |\partial_{vv}h(x,v)|\leq c\rho^{-2}\|h\|_{L^\infty(H_\rho(x,v))},
    \end{align}
    where $\rho\coloneqq \max\{x^{1/3},|v|\}/8$. Thus, as in \eqref{phialpha:d2vh} together with \eqref{upper:ineq1} and \eqref{upper:ineq2}, we have $\partial_{vv}h\in L^2(H_{1/2})$.
    Thus, applying \autoref{lem.maximum} to the functions $h-c\phi_{\alpha,b}\|h\|_{L^\infty({H}_{1/2})}$ and $-h-c\phi_{\alpha,b}\|h\|_{L^\infty({H}_{1/2})}$ which are solutions to \eqref{maximum:eq} and are continuous in $\overline{{H}}_{1/2}$, we get the desired result.
\end{proof}

Now it remains to obtain a pointwise lower bound of $h/\phi_{\alpha,b}$ for any nonnegative solution $h$. In case $\alpha=0$, we have proved such lower bound in \cite[Lemma 3.7]{KiWe26} by constructing a suitable barrier based on the fact that $\phi_{0}\geq0$ and $\partial_v\phi_0\leq 0$, where $\phi_0$ is the solution to \eqref{eq.1d.alphab} with $\alpha=0$. However, note that $\phi_{\alpha,b}$ does not satisfy the latter property if $\alpha>0$, since $\partial_v\phi_{\alpha,b}$ changes sign. We overcome this issue by a mirror extension argument across $\{x=0\}\times\{v> v_1\}$ for some $v_1 > 0$ which we provide in the following lemma.

\begin{lemma}\label{lem.mirroext}
   Let $\alpha\in(0,1)$, $b\in(0,1]$ with $\alpha/b^2\in[0,1]$. Let $h\in C(\overline{{H}_1})$ be a weak solution to 
    \begin{equation}\label{mirroext:eq}
\left\{
\begin{alignedat}{3}
v\partial_xh-\partial_{vv}h&=0&&\qquad \mbox{in  ${H}_1$}, \\
h(0,v)&=\alpha  h(0,-bv)&&\qquad  \mbox{in $\{x=0\}\times (0,1)$}.
\end{alignedat} \right.
\end{equation}
Given $x_1\in(0,1)$ and $0<v_1<v_2<1$, we define the function $\widetilde{h}(x,v)$ by
\begin{align}\label{mirroext:defn.tildeh}
    \widetilde{h}(x,v)\coloneqq\begin{cases}
        h(x,v)&\quad\text{in }[0,x_1]\times [v_1,v_2],\\
        \alpha h(-b^3x,-bv)&\quad\text{in }[-x_1,0]\times[v_1,v_2].
    \end{cases}
\end{align}
Then, $\widetilde{h}$ is a weak solution to 
\begin{align*}
    v\partial_x\widetilde{h}-\partial_{vv}\widetilde{h}=0\quad\text{in }(-x_1,x_1)\times (v_1,v_2).
\end{align*}
\end{lemma}

Note that in case $\alpha = b = 1$, the mirror extended function $\widetilde{h}$ becomes a solution in $\R^2$, i.e. we can set $v_1 = 0$ and extend $h$ also in a neighborhood of the grazing set. This fact is well known (see \cite{GHJO20}) and has been used in the proof of the Liouville theorem in \cite[Lemma 3.6]{RoWe25}.

\begin{proof}
   First, by the standard interior regularity theory, we observe that $h$ is smooth inside ${H}_1$, and in $[0,x_1) \times (v_1,v_2)$ by \autoref{lem.schhol.away-} and the boundary condition. Moreover, as in the proof of \autoref{lem.upper}, we have $\partial_{vv}h\in L^2([0,x_1]\times [v_1,v_2])$, which gives $v\partial_xh\in L^2([0,x_1]\times [v_1,v_2])$ by using the equation.
    Note that $\partial_v\widetilde{h}\in L^2(D)$, where $D\coloneqq (-x_1,x_1)\times (v_1,v_2)$, as $\partial_vh\in L^2({H}_1)$. Now, we extend the solution across the line $\{x=0\}\times \{v\geq v_1\}$. It suffices to show that for any $\psi\in C_c^\infty(D)$,
    \begin{align*}
        \int_{D}-v\widetilde{h}\partial_x\psi+\int_{D}\partial_v\widetilde{h}\partial_v\psi=0.
    \end{align*}
    First, we will estimate the terms $J_1$ and $J_2$ given by
    \begin{align*}
        \int_{D}\partial_v\widetilde{h}\partial_v\psi=\int_{D\cap \{x>0\}}\partial_v\widetilde{h}\partial_v\psi+\int_{D\cap\{x<0\}}\partial_v\widetilde{h}\partial_v\psi\eqqcolon J_1+J_2.
    \end{align*}
    By the definition of $\widetilde{h}$ and using the equation \eqref{mirroext:eq} together with the fact that $\psi(\cdot,v_1)=\psi(\cdot,v_2)=\psi(x_1,\cdot)=0$ and the smoothness of $h$ up to $\gamma_-$, we have 
    \begin{align*}
        J_1=\int_{v_1}^{v_2}\int_{0}^{x_1}\partial_vh\partial_v\psi\,dx\,dv&=-\int_{v_1}^{v_2}\int_{0}^{x_1}\partial_{vv}h\psi\,dx\,dv\\
        &=-\int_{v_1}^{v_2}\int_{0}^{x_1}v\partial_xh\, \psi\,dx\,dv\\
        &=\int_{v_1}^{v_2}\int_{0}^{x_1}v\partial_x\psi\,h\,dx\,dv-\int_{v_1}^{v_2} v(\psi h)(x,v)\,dv\bigg\rvert_{x=0}^{x=x_1}\\
        &=\int_{v_1}^{v_2}\int_{0}^{x_1}v\widetilde{h}\partial_{x}\psi\,dx\,dv+\int_{v_1}^{v_2} v(\psi h)(0,v)\,dv.
    \end{align*}
    Similarly, using this time the regularity of $h$ up to $\gamma_+$, we get
    \begin{align*}
        J_2=\int_{v_1}^{v_2}\int_{-x_1}^{0}(-b)\alpha\partial_vh(-b^3x,-bv)\partial_v\psi\,dx\,dv
        &=\int_{-bv_2}^{-bv_1}\int_{0}^{b^3x_1}-\alpha b^{-3}\partial_vh(x,v)\partial_v\psi(-x/b^3,-v/b)\,dx\,dv\\
        &=(-b^{-2})\alpha\int_{-bv_2}^{-bv_1}\int_{0}^{b^3x_1}v\partial_xh\, \psi(-x/b^3,-v/b)\,dx\,dv\\
        &=(-b^{-5})\alpha \int_{-bv_2}^{-bv_1}\int_{0}^{b^3x_1}v\partial_x\psi(-x/b^3,-v/b)\,h\,dx\,dv\\
        &\quad+(-b^{-2})\alpha\int_{-bv_2}^{-bv_1} vh(x,v)\psi(-x/b^3,-v/b)\,dv\bigg\rvert_{x=0}^{x=b^3x_1},
    \end{align*}
    where we have used the change of variables $v \mapsto -bv$ several times and applied the equation \eqref{mirroext:eq}. Again, by the change of variables and the fact that $\psi(-x_1,\cdot)\equiv0$, we further rewrite $J_2$ as
    \begin{align*}
        J_2&=\alpha\int_{v_1}^{v_2}\int_{-x_1}^{0}vh(-b^3x,-bv)\partial_{x}\psi\,dx\,dv-\alpha \int_{v_1}^{v_2} vh(0,-bv)\psi(0,v)\,dv.
    \end{align*}
    Now combining $J_1$ and $J_2$ together with the boundary condition, we obtain
    \begin{align*}
        \int_{D}-v\widetilde{h}\partial_x\psi+\int_{D}\partial_v\widetilde{h}\partial_v\psi= \int_{D}-v\widetilde{h}\partial_x\psi+J_1+J_2=0.
    \end{align*}
    This completes the proof.
\end{proof} 
Now we are ready to prove the lower bound of $h/\phi_{\alpha,b}$ for any given nonnegative solution $h$.
\begin{lemma}\label{lem.lower}
    Let $\alpha\in(0,1)$, $b\in(0,1]$ and $\alpha/b^2\in(0,1]$. Let $h$ be a weak solution to 
    \begin{equation*}
\left\{
\begin{alignedat}{3}
v\partial_xh-\partial_{vv}h&=0&&\qquad \mbox{in  ${H}_1$}, \\
h(0,v)&=\alpha h(0,-bv)&&\qquad  \mbox{in $\{x=0\}\times (0,1)$},\\
h&\geq0&&\qquad\mbox{in ${H}_1$}.
\end{alignedat} \right.
\end{equation*}
Then, there is a constant $c=c(\alpha,b)$ such that
\begin{align*}
    \inf_{{H}^+_{1,\frac12}}\frac{h}{\phi_{{\alpha,b}}}\leq c\inf_{{H}_{\frac12}}\frac{h}{\phi_{\alpha,b}},
\end{align*}
where ${H}^+_{1,\frac12}\coloneqq\{(x,v)\in{H}_1\,:\, (1/2)^3\leq x\leq 1\}$.
\end{lemma}
\begin{proof}
    We may assume $m_0\coloneqq \inf\limits_{{H}^+_{1,\frac12}}\frac{h}{\phi_{\alpha,b}}\neq0$. By \autoref{lem.phialpha}, $\phi_{\alpha,b}\geq\frac1c $ for some constant $c=c(\alpha,b)$ in ${H}^+_{1,\frac12}$ so that we have 
    \begin{align}\label{lbdd:ineq0}
        h\geq \frac{m_0}{c}\quad\text{in }{H}^+_{1,\frac12}.
    \end{align}
     Since $h\geq0$ in ${H}_1$, by \cite[(3.6) in Lemma 3.4]{KiWe26}, we have 
    \begin{equation}\label{lbdd:ineq.h}
       h(1,-v)\leq ch(x,-v) \quad\text{for any }(x,v)\in[0,1]\times [-3/4,-b/2],
    \end{equation}
    which implies $h(x,v)\geq \frac{m_0}{c}$ in $(x,v)\in [0,1]\times[-1,-b/2]$, where $c=c(\alpha,b)$. 
    
    Now we will prove 
    \begin{align}\label{lbdd:ineq1.h}
        h(x,v)\geq \frac{m_0}{c}\quad\text{in }  [0,(1/2)^3]\times\{v=1/2\}.
    \end{align} By the aid of \autoref{lem.mirroext}, we find that $\widetilde{h}$ defined in \eqref{mirroext:defn.tildeh} with $x_1=(3/4)^3$, $v_1=1/4$, and $v_2=3/4$ is a nonnegative weak solution to
    \begin{align*}
        v\partial_x\widetilde{h}-\partial_{vv}\widetilde{h}=0\quad\text{in }(-(3/4)^3,(3/4)^3)\times (1/4,3/4)\eqqcolon D.
    \end{align*}
  Since $\widetilde{h}$ solves the equation in  $D$ and since $(x,1/2)$ is an interior point of $D$ for any $x\in[0,1/2^3]$,  we can apply \cite[(3.5) in Lemma 3.4]{KiWe26} to obtain
   \begin{align*}
       \widetilde{h}(0,1/2)\leq c\widetilde{h}(x,1/2)\quad\text{for any }x\in[0,(1/2)^3].
   \end{align*}
    This gives \eqref{lbdd:ineq1.h} by the definition of $\widetilde{h}$ and \eqref{lbdd:ineq.h}. Combining \eqref{lbdd:ineq0}, \eqref{lbdd:ineq.h}, \eqref{lbdd:ineq1.h}, and the fact that $\phi_{\alpha,b}\leq c$ in ${H}_{1/2}$ for some constant $c=c(\alpha,b)$, there is a constant $C=C(\alpha,b)\geq1$ such that $-h+\frac{m_0\phi_{\alpha,b}}{C}$ is a weak solution to \eqref{maximum:eq} with ${H}_1$ replaced by ${H}_{1/2}$. Thus, the desired estimate follows from \autoref{lem.maximum}.
\end{proof}

Now we prove the first main result in this section.
\begin{proof}[Proof of \autoref{prop.osc}.]
    By considering the function ${h}_R(x,v)\coloneqq h(R^3x,Rv)/R^{\lambda_{\alpha,b}}$, we may assume $R=1$. Since we have the upper and lower bound estimate of $h/\phi_{\alpha,b}$ in \autoref{lem.upper} and \autoref{lem.lower}, respectively, the proof is exactly the same as in the proof of \cite[Proposition 3.6]{KiWe26} with $R_0=1$. 
\end{proof}

\subsection{Liouville theorems in the half-space}

Next, we establish the second main result of this section, namely a 1D Liouville theorem.
\begin{proposition}\label{prop.liou1}
   Let $\alpha\in(0,1)$, $b\in(0,1]$ with $\alpha/b^2\in(0,1]$. Let $\lambda_{\alpha,b}$ be the constant determined in \eqref{phialpha:defn.lambdaalpha} and $\beta_{\alpha,b}$ be the constant determined in \autoref{prop.osc}. Let $h$ be a weak solution to 
    \begin{equation}\label{liou1:eq}
\left\{
\begin{alignedat}{3}
v\partial_xh-\partial_{vv}h&=0&&\qquad \mbox{in  $\{x>0\}$}, \\
h(0,v)&=\alpha \phi_{\alpha,b}(0,-bv)&&\qquad  \mbox{in $\{x=0\}\times \{v>0\}$},\\
\|h\|_{L^\infty({H}_R)}&\leq cR^{\lambda_{\alpha,b}+\eps} && \qquad\text{for any }R\geq1,
\end{alignedat} \right.
\end{equation}
where $\eps\in(0,\beta_{\alpha,b})$ and $c\in\bbR$. Then $h=m\phi_{\alpha,b}$ for some  $m\in\bbR$.
\end{proposition}

\begin{proof}
   To keep the notation simple, we denote $\beta := \beta_{\alpha,b}$ and $\lambda := \lambda_{\alpha,b}$. Fix $\delta=\frac{\eps+\beta}{2}\in(\eps,\beta)$. We now apply \autoref{prop.osc} with $r=R^{-\frac{\delta}{\beta}}$ to see that
   \begin{align*}
       \sup_{H_{R^{1-\frac{\delta}{\beta}}}}\frac{h}{\phi_{\alpha,b}}-\inf_{H_{R^{1-\frac{\delta}{\beta}}}}\frac{h}{\phi_{\alpha,b}}\leq cR^{-\lambda -\delta}\|h\|_{L^\infty(\mathcal{H}_R)}.
   \end{align*}
   Now using the growth condition on $h$, we further estimate 
   \begin{align*}
        \sup_{H_{R^{1-\frac{\delta}{\beta}}}}\frac{h}{\phi_{\alpha,b}}-\inf_{H_{R^{1-\frac{\delta}{\beta}}}}\frac{h}{\phi_{\alpha,b}}\leq cR^{-\delta+\eps}.
   \end{align*}
   Since $\eps<\delta$ and $\frac{\delta}{\beta}<1$, by taking $R\to\infty$, we deduce that $\osc_{\{x>0\}\times \R} \frac{h}{\phi_{\alpha,b}} = 0$ and therefore $\frac{h}{\phi_{\alpha,b}}=m$ for some constant $m\in\bbR$. This completes the proof. 
\end{proof}

Lastly, in light of \autoref{prop.liou1} and a dimension reduction process, we have the following $n$-dimensional Liouville's theorem for a general second order differential operator with constant coefficients.
\begin{theorem}\label{lem.liou.nd}
    Fix $\alpha\in(0,1)$, $b\in(0,1]$ with $\alpha/b^2\in(0,1]$. Let $\lambda_{\alpha,b}$ be the constant determined in \eqref{phialpha:defn.lambdaalpha} and $\beta_{\alpha,b}$ be the constant determined in \autoref{prop.osc}. Let $f$ be a weak solution to 
    \begin{equation*}
\left\{
\begin{alignedat}{3}
\partial_tf+v\cdot\nabla_xf-a^{i,j}\partial_{v_i,v_j}f&=0&&\qquad \mbox{in  $\{x_n>0\}$}, \\
f(t,x',0,v',v_n)&=\alpha f(t,x',0,v',-bv_n)&&\qquad  \mbox{in $\{x_n=0\}\times \{v_n>0\}$},\\
\|f\|_{L^\infty({H}_R)}&\leq cR^{\lambda_{\alpha,b}+\eps}&&\qquad\text{for any }R\geq1,
\end{alignedat} \right.
\end{equation*}
where $\eps\in(0,\beta_{\alpha,b})$, $c\in\bbR$ and $a^{n,n}=1$. Then $f(z)=m\phi_{\alpha,b}(x_n,v_n)$ for some constant $m\in\bbR$.
\end{theorem}
\begin{proof}
    Since we have uniform $C^\beta$ estimates up to the boundary by \autoref{lem.holsmall}, we use the difference quotient technique from \cite[Theorem 3.5]{RoWe25} and \cite[Theorem 4.1]{KiWe26}. Note that as we take difference quotient only in  $(t,x',v')$, the boundary condition is preserved. Following the same lines as in the proofs of \cite[Theorem 3.5]{RoWe25} and \cite[Theorem 4.1]{KiWe26} together with the growth condition and the fact that $\lambda_{\alpha,b}+\eps<1$, we deduce that $f=f(x_n,v_n)$ depends only on $(x_n,v_n)$ and that it is a 1D solution to \eqref{liou1:eq}. Therefore, by \autoref{prop.liou1}, we have $f(z)=m\phi_{\alpha,b}(x_n,v_n)$ for some constant $m\in\bbR$. This completes the proof.
\end{proof}

\section{Optimal regularity}
\label{sec:opt}
In this section, we will prove the optimal regularity of kinetic equations with $\alpha$-reflection-type boundary conditions with $\alpha<1$.
The main result of this section is the following.

\begin{theorem}\label{thm.opthol}
    Fix $\alpha\in(0,1)$, $b\in(0,1]$ with $\alpha/b^2\in(0,1]$. Let $z_0\in\gamma_0$, $R\leq1$ and let $f$ be a weak solution to 
    \begin{equation*}
\left\{
\begin{alignedat}{3}
\partial_tf+v\cdot\nabla_xf-\ddiv(A\nabla_vf)&=B\cdot\nabla_vf+F&&\qquad \mbox{in  $\mathcal{H}_R(z_0)$}, \\
f&=\alpha \mathcal{R}_bf+(1-\alpha)g&&\qquad  \mbox{in $\gamma_-\cap \mathcal{Q}_R(z_0)$},
\end{alignedat} \right.
\end{equation*}
where $A\in C^\eps(\mathcal{H}_R(z_0))$, $B\in L^\infty(\mathcal{H}_R(z_0))$, $F\in L^\infty(\mathcal{H}_R(z_0))$, and $g\in C^{\lambda_{\alpha,b}+\eps}(\gamma_-\cap \mathcal{Q}_R(z_0))$ for some $\eps>0$. 

Then we have 
\begin{equation}\label{opthol:main1}
\begin{aligned}
    R^{\lambda_{\alpha,b}}[f]_{C^{\lambda_{\alpha,b}}(H_{R/2}(z_0))}&\leq c\mathbf{B}(z_0,R)^{\lambda_{\alpha,b}}\Bigg(\|f\|_{L^\infty(\mathcal{H}_R(z_0))}+R^2\|F\|_{L^\infty(\mathcal{H}_R(z_0))}\\
    &\qquad\qquad\qquad\qquad+R^{\lambda_{\alpha,b}+\eps}[g]_{C^{\lambda_{\alpha,b}+\eps}(\gamma_-\cap \mathcal{Q}_R(z_0))}\Bigg),
\end{aligned}
\end{equation}
where $c=c(n,\Lambda,\eps,\|A\|_{C^\eps(\mathcal{H}_R(z_0))},\alpha,b)$ and $\mathbf{B}(z_0,R)\coloneqq 1+\|B\|_{L^\infty(\mathcal{H}_R(z_0))}$. Moreover, for any  $\delta\in(0,\lambda_{\alpha,b})$, we have 
\begin{align}\label{opthol:main2}
    R^{\delta}[f]_{C^{\delta}(H_{R/2}(z_0))}&\leq c\mathbf{B}(z_0,R)^{\delta}\Bigg(\|f\|_{L^\infty(\mathcal{H}_R(z_0))}+R^2\|F\|_{L^\infty(\mathcal{H}_R(z_0))}+R^{\delta}[g]_{C^{\delta}(\gamma_-\cap \mathcal{Q}_R(z_0))}\Bigg),
\end{align}
where $c=c(n,\Lambda,\eps,\|A\|_{C^\eps(\mathcal{H}_R(z_0))},\delta,\alpha,b)$.
\end{theorem}
To prove this result, we first establish expansion estimates of order $\lambda_{\alpha,b}+\eps$ for some small $\eps>0$ near the grazing set. Note that we already have H\"older estimates away from $\gamma_-\cup\gamma_0$ by \autoref{lem.schhol.away-}, which implies that the boundary data belongs to a suitable H\"older space. This, in turn, is sufficient to obtain H\"older estimates up to $\gamma_-$. Thus, combining these estimates via a covering argument yields the optimal regularity.

The following proposition yields expansion estimates at the grazing set of order $\lambda_{\alpha,b} +\eps$ for some $\eps > 0$.

\begin{proposition}\label{thm.expansion}
    Fix $\alpha\in(0,1)$, $b\in(0,1]$ with $\alpha/b^2\in(0,1]$. Let $z_0\in\gamma_0$, $R\leq1$, and let $f$ be a weak solution to 
    \begin{equation}\label{expansion:eq}
\left\{
\begin{alignedat}{3}
\partial_tf+v\cdot\nabla_xf-\ddiv(A\nabla_vf)&=B\cdot\nabla_vf+F&&\qquad \mbox{in  $\mathcal{H}_R(z_0)$}, \\
f&=\alpha \mathcal{R}_bf+(1-\alpha)g&&\qquad  \mbox{in $\gamma_-\cap \mathcal{Q}_R(z_0)$},
\end{alignedat} \right.
\end{equation}
where $A\in C^\eps(\mathcal{H}_R(z_0))$, $B\in L^\infty(\mathcal{H}_R(z_0))$, and $F\in L^\infty(\mathcal{H}_R(z_0))$ for some $\eps\in(0,\beta_{\alpha,b})$, where $(A(z_0))_{n,n}=1$ and the small constant $\beta_{\alpha,b}$ is determined in \autoref{prop.osc}. 

Then there is a constant $C\in\bbR$ such that for any $r\in(0,R/2]$,
\begin{equation}\label{expansion:maingoal}
\begin{aligned}
    \|f-f(z_0)-C\phi_{\alpha,b}\|_{L^\infty(\mathcal{H}_R(z_0))}&\leq c\left(\frac{r}R\right)^{\lambda_{\alpha,b}+\eps}\Bigg(\|f\|_{L^\infty(\mathcal{H}_R(z_0))}+R^2\|F\|_{L^\infty(\mathcal{H}_R(z_0))}\\
    &\qquad\qquad\qquad\qquad+R^{\lambda_{\alpha,b}+\eps}[g]_{C^{\lambda_{\alpha,b}+\eps}(\gamma_-\cap \mathcal{Q}_R(z_0))}\Bigg),
\end{aligned}
\end{equation}
where $c=c(n,\Lambda,\eps,\alpha,b,\|A\|_{C^\eps(\mathcal{H}_R(z_0))},\|B\|_{L^\infty(\mathcal{H}_R(z_0))})$.
\end{proposition}
The proof relies on a blow-up argument, which requires uniform H\"older estimates up to the boundary even when the right-hand side is unbounded. We therefore first establish the following auxiliary lemma before proceeding to the proof of \autoref{thm.expansion}. 
\begin{lemma}\label{lem.replace}
   Fix $\alpha\in(0,1)$, $b\in(0,1]$ with $\alpha/b^2\in(0,1]$.  Let $f$ be a bounded weak solution to 
    \begin{equation*}
\left\{
\begin{alignedat}{3}
\partial_tf+v\cdot\nabla_xf-\ddiv(A\nabla_vf)&=B\cdot\nabla_vf+F\partial_{v_n}\phi_{\alpha,b}-\ddiv(G\partial_{v_n}\phi_{\alpha,b})&&\qquad \mbox{in  $\mathcal{H}_1$}, \\
f&=\alpha \mathcal{R}_bf+(1-\alpha)g&&\qquad  \mbox{in $\gamma_-\cap \mathcal{Q}_1$},
\end{alignedat} \right.
\end{equation*}
where $A\in C^\eps(\mathcal{H}_1)$, $B\in L^\infty(\mathcal{H}_1)$, and $F\in L^\infty(\mathcal{H}_1)$. Then there is a constant $\gamma=\gamma(n,\Lambda,\alpha,b)$ such that
\begin{align*}
    [f]_{C^\gamma(\mathcal{H}_{1/2})}\leq c(\|f\|_{L^\infty(\mathcal{H}_1)}+\|F\|_{L^\infty(\mathcal{H}_1)}+\|G\|_{L^\infty(\mathcal{H}_1)}+[g]_{C^{\gamma}(\gamma_-\cap \mathcal{Q}_1)})
\end{align*}
for some constant $c=c(n,\Lambda,\eps,\alpha,b,\|A\|_{C^\eps(\mathcal{H}_1)},\|B\|_{L^\infty(\mathcal{H}_1)})$.
\end{lemma}
In our applications, we can assume $f$ to be uniformly bounded up to the boundary so that we can restrict ourselves to bounded solution here.

\begin{proof}
    By \cite[Lemma 2.7]{Zhu24}, there is a weak solution $h$ to
     \begin{equation*}
\left\{
\begin{alignedat}{3}
\partial_th+v\cdot\nabla_xh-\ddiv(A\nabla_vh)&=B\cdot\nabla_vh+F\partial_{v_n}\phi_{\alpha,b}-\ddiv(G\partial_{v_n}\phi_{\alpha,b})&&\qquad \mbox{in  $\mathcal{H}_{1}$}, \\
h&=0&&\qquad  \mbox{in $\partial_{\mathrm{kin}}\mathcal{H}_{1}$}.
\end{alignedat} \right.
\end{equation*}
By \autoref{lem.phialpha}, we observe that $\partial_{v_n}\phi_{\alpha,b}=\partial_{v_n}\phi_{\alpha,b}(x_n,v_n)\in L^{4+\delta}({H}_1)$ for some small $\delta>0$  and $\sup_{(t,x',v')}\|\partial_{v_n}\phi_{\alpha,b}\|_{L^{4+\delta}_{x_n,v_n}({H}_1)}\leq c(\alpha,b)$. Next, as in the proof of \cite[Lemma 5.2]{KiWe26}, we get that there is a small constant $\gamma=\gamma(n,\Lambda,\alpha,b)$ such that
\begin{align*}
    \|h\|_{C^\gamma(H_{7/8})}\leq c(\|h\|_{L^2(\mathcal{H}_1)}+\|F\|_{L^\infty(\mathcal{H}_1)}+\|G\|_{L^\infty(\mathcal{H}_1)}),
\end{align*}
where $c=c(n,\Lambda,\alpha,b,\|B\|_{L^\infty(\mathcal{H}_1)})$. Although \cite[Lemma 5.2]{KiWe26} is stated under the assumption $\delta=\frac13$, its proof works for any $\delta>0$. The corresponding constants $\beta$ and $c$ in \cite[Lemma 5.2]{KiWe26} then depend on $\delta$ as well.

In addition, from the standard energy estimate as in \cite[Lemma 2.8]{Zhu24}, we have 
\begin{align*}
    \|h\|_{L^2(\mathcal{H}_1)}\leq c(\|F\|_{L^\infty(\mathcal{H}_1)}+\|G\|_{L^\infty(\mathcal{H}_1)}).
\end{align*}
Next, we consider the equation
\begin{equation*}
\left\{
\begin{alignedat}{3}
\partial_t(f-h)+v\cdot\nabla_x(f-h)-\ddiv(A\nabla_v(f-h))&=B\cdot\nabla_v(f-h)&&\qquad \mbox{in  $\mathcal{H}_{3/4}$}, \\
f-h&=\alpha\mathcal{R}_b(f-h)+(1-\alpha)\widetilde{g}&&\qquad  \mbox{in $\gamma_-\cap \mathcal{Q}_{3/4}$},
\end{alignedat} \right.
\end{equation*}
where we write $\widetilde{g}\coloneqq g+\frac{\alpha}{1-\alpha}\mathcal{R}_bh$. By \autoref{lem.holsmall} and \autoref{lem.geo}, we have 
\begin{align*}
    [f-h]_{C^{\gamma}(\mathcal{H}_{1/2})}&\leq c(\|f-h\|_{L^\infty(\mathcal{H}_{3/4})}+[\widetilde{g}]_{C^\gamma(\gamma_-\cap \mathcal{Q}_{3/4})})\\
    &\leq c(\|f-h\|_{L^\infty(\mathcal{H}_{3/4})}+[h]_{C^{\gamma}(\mathcal{H}_{3/4})}+[g]_{C^{\gamma}(\gamma_-\cap \mathcal{Q}_{3/4})})
\end{align*}
for some constant $c=c(n,\Lambda,\alpha,b,\|B\|_{L^\infty(\mathcal{H}_1)})$.
Now we combine all the estimates to see that 
\begin{align*}
    [f]_{C^\gamma(\mathcal{H}_{1/2})}\leq [h]_{C^{\gamma}(\mathcal{H}_{1/2})}+[f-h]_{C^{\gamma}(\mathcal{H}_{1/2})}\leq c(\|f\|_{L^\infty(\mathcal{H}_1)}+\|F\|_{L^\infty(\mathcal{H}_1)}+\|G\|_{L^\infty(\mathcal{H}_1)}+[g]_{C^{\gamma}(\gamma_-\cap \mathcal{Q}_1)}),
\end{align*}
which completes the proof.
\end{proof}

Now we are ready to prove the expansion estimate from \eqref{expansion:maingoal}.

\begin{proof}[Proof of \autoref{thm.expansion}]
By \autoref{lem.holsmall}, $f$ is H\"older continuous up to the boundary and $f(z_0)=\alpha\mathcal{R}_bf(z_0)+(1-\alpha)g(z_0)$ is well-defined. Since $z_0\in \gamma_0$, we have $\mathcal{R}_bf(z_0)=f(z_0)$ so that $f(z_0)=g(z_0)$. Thus, by considering $\widetilde{f}(z)\coloneqq f(z)-f(z_0)$, we have that $\widetilde{f}$ is a solution to \eqref{expansion:eq} with $g$ replaced by $\widetilde{g}=g(z)-g(z_0)$, which gives $\widetilde{f}(z_0)=\widetilde{g}(z_0)=0$. Therefore, we always assume $f(z_0)=g(z_0)=0$.

Moreover, by a scaling argument (see \cite[Proposition 5.1]{KiWe26}) and the fact that $z_0\in\gamma_0$, we may assume $R=1$ and $z_0=0$. Now it suffices to prove \eqref{expansion:maingoal} under the assumptions $R=1$, $z_0=0$, and $f(0)=g(0)=0$. 

We prove it by contradiction. Suppose there are sequences $(A_l)_l$, $(B_l)_l$, $(F_l)_l$, $(g_l)_{l}$, and $(f_l)_l$ such that 
\begin{equation*}
\left\{
\begin{alignedat}{3}
\partial_tf_l+v\cdot\nabla_xf_l-\ddiv(A_l\nabla_vf_l)&=B_l\cdot\nabla_vf_l+F_l&&\qquad \mbox{in  $\mathcal{H}_1$}, \\
f_l&=\alpha \mathcal{R}_bf_l+(1-\alpha)g_l&&\qquad  \mbox{in $\gamma_-\cap \mathcal{Q}_1$},
\end{alignedat} \right.
\end{equation*}
where $(A_l(0))_{n,n}=1$, $f_l(0)=g_l(0)=0$,
\begin{equation}\label{expansion:ass.coeff}
\begin{aligned}
    &\|A_l\|_{C^\eps(\mathcal{H}_1)}+\|B_l\|_{L^\infty(\mathcal{H}_1)}\leq \Lambda,\\
    &\|f_l\|_{L^\infty(\mathcal{H}_1)}+\|F_l\|_{L^\infty(\mathcal{H}_1)}+[g]_{C^{\lambda_{\alpha,b}+\eps}(\gamma_-\cap \mathcal{Q}_1)}\leq 1,
\end{aligned}
\end{equation}
but 
\begin{align}\label{expansion:contr} 
    \sup_{l\in\N}\inf_{C\in\bbR}\inf_{\rho\in(0,\frac12]}\rho^{-(\lambda_{\alpha,b}+\eps)}\|f_l-C\phi_{\alpha,b}\|_{L^\infty(\mathcal{H}_\rho)}=\infty.
\end{align}
We now choose a suitable blow-up sequence to derive a contradiction. To this end, we start by choosing a constant $C_{l,\rho}$ for each $l\in\N$ and $\rho>0$ such that for any $C\in\bbR$,
\begin{align}\label{expansion:l2proj}
    \int_{\mathcal{H}_\rho}|f_l-C_{l,\rho}\phi_{\alpha,b}|^2\,dz\leq \int_{\mathcal{H}_\rho}|f_l-C\phi_{\alpha,b}|^2\,dz\quad\text{and}\quad
    \int_{\mathcal{H}_\rho}(f_l-C_{l,\rho}\phi_{\alpha,b})\phi_{\alpha,b}\,dz=0.
\end{align}
Note that the constant $C_{l,\rho}$ is obtained by taking an $L^2$-projection (see \cite{RoWe25,KiWe26}). Next, we define a function $\theta(r):(0,1/2]\to\bbR$ by
\begin{align*}
    \theta(r)\coloneqq \sup_{l\in\N}\sup_{\rho\in[r,\frac12]}\rho^{-(\lambda_{\alpha,b}+\eps)}\|f_l-C_{l,\rho}\phi_{\alpha,b}\|_{L^\infty(\mathcal{H}_\rho)}.
\end{align*}
Since $\phi_{\alpha,b}$ is homogeneous with degree $\lambda_{\alpha,b}$, using \cite[Lemma 2.14]{RoWe25}, we have $\theta(r)\to\infty$ as $r\to0$. More precisely, if $\theta(r)\leq N$, then by \cite[Lemma 2.14]{RoWe25}, for each $l$, we can find a constant $C_{l,\infty}$ such that $\|f_l-C_{l,\infty}\phi_{\alpha,b}\|_{L^\infty(\mathcal{H}_\rho)}\leq c\rho^{\lambda_{\alpha,b}+\eps}$, which contradicts \eqref{expansion:contr}.

As $\lim\limits_{r\to0}\theta(r)=\infty $ and $\theta(r)$ is decreasing, there are sequences $(l_m)_m$ and $(r_{l_m})_m$ such that $r_{l_m}\to0$ as $m\to\infty$ and 
\begin{align}\label{expansion:compar}
    \theta(r_{l_m})\geq r_{l_m}^{-(\lambda_{\alpha,b}+\eps)}\|f_{l_m}-C_{l_m,r_{l_m}}\phi_{\alpha,b}\|_{L^\infty(H_{r_{l_m}})}\geq \frac{\theta(r_{l,m})}2.
\end{align}
To simplify the notation, we shall henceforth assume that $l_m=m$.

For each $m$ and for any $R\leq \frac{1}{r_m}$, using \cite[Lemma 2.15]{RoWe25}, we have
\begin{align}\label{expansion:linfty}
    \|f_{m}-C_{m,r_{m}}\phi_{\alpha,b}\|_{L^\infty(H_{Rr_{m}})}\leq c\theta(r_{m})(Rr_{m})^{\lambda_{\alpha,b}+\eps}
\end{align}
with 
\begin{align}\label{expansion:limit.cmrm}
    \lim_{m\to\infty}\frac{|C_{m,r_m}|}{\theta(r_m)}=0.
\end{align}
Now we are ready to define a blow up sequence $\widetilde{f}_m$ by 
\begin{align*}
    \widetilde{f}_m(z)\coloneqq \frac{f_m(S_{r_m}z)-C_{m,r_{m}}\phi_{\alpha,b}(S_{r_m}z)}{r_m^{\lambda_{\alpha,b}+\eps}}
\end{align*}
and observe that $\widetilde{f}_m$ solves
\begin{equation}\label{expansion:eq.tildem}
\left\{
\begin{alignedat}{3}
\partial_t\widetilde{f}_m+v\cdot\nabla_x\widetilde{f}_m-\ddiv(\widetilde{A}_m\nabla_v\widetilde{f}_m)&=\widetilde{B}_m\cdot\nabla_v\widetilde{f}_m+\widetilde{F}_m+\widetilde{G}_m&&\qquad \mbox{in  $H_{1/r_m}$}, \\
\widetilde{f}_m&=\alpha \mathcal{R}_b\widetilde{f}_m+(1-\alpha)\widetilde{g}_m&&\qquad  \mbox{in $\gamma_-\cap Q_{1/r_m}$},
\end{alignedat} \right.
\end{equation}
where 
\begin{align*}
    &\widetilde{A}_m\coloneqq A(S_{r_m}z),\quad \widetilde{B}_m(z)\coloneqq r_mB(S_{r_m}z),\quad \widetilde{g}_m(z)\coloneqq \frac{g(S_{r_m}z)}{r_m^{\lambda_{\alpha,b}+\eps}\theta(r_m)},\quad \widetilde{F}_m(z)\coloneqq \frac{r_m^{2-(\lambda_{\alpha,b}+\eps)}F_m(S_{r_m}z)}{\theta(r_m)}\\
    &\widetilde{G}_m(z)\coloneqq-\frac{C_{m,r_m}\ddiv((\widetilde{A}_m-I)\nabla_v(\phi_{\alpha,b}(S_{r_m}z))+C_{m,r_m}\widetilde{B}_m\cdot\nabla_v(\phi_{\alpha,b}(S_{r_m}z))}{r_m^{\lambda_{\alpha,b}+\eps}}.
\end{align*}
In particular, we have used the fact that $v\cdot \nabla_x\phi_{\alpha,b}-\ddiv(\nabla_v\phi_{\alpha,b})=0$ by \autoref{lem.phialpha}.

Next we prove uniform estimates of the sequences $\widetilde{g}_m$, $\widetilde{F}_m$, $\widetilde{G}_m$ and $\widetilde{f}_m$, in order to find the limiting equation of $\widetilde{f}_m$.

To do so, using the homogeneity of $\phi_{\alpha,b}$ and the fact that $\phi_{\alpha,b}$ depends only on $(x_n,v_n)$ and $(\widetilde{A}_m(0))_{n,n}=1$, we have 
\begin{align*}
    -\ddiv((\widetilde{A}_m-I)\nabla_v(\phi_{\alpha,b}))&=-\sum_{i=1}^{n-1}\partial_{v_i}(\widetilde{A}_m)\partial_{v_n}\phi_{\alpha,b}-\partial_{v_n}((\widetilde{A}_m-I)_{n,n}\partial_{v_n}\phi_{\alpha,b})\\
    &=-\sum_{i=1}^n\partial_{v_i}((\widetilde{A}_m-\widetilde{A}_m(0))_{i,n}\partial_{v_n}\phi_{\alpha,b}),
\end{align*}
which gives
\begin{align*}
    \widetilde{G}_m\coloneqq -\frac{C_{m,r_m}\sum\limits_{i=1}^n\partial_{v_i}((\widetilde{A}_m-\widetilde{A}_m(0))_{i,n}\partial_{v_n}\phi_{\alpha,b})+C_{m,r_m}\widetilde{B}_m\cdot\nabla_v\phi_{\alpha,b}}{r_m^{\eps}\theta(r_m)},
\end{align*}
Moreover, we observe
 $\|(\widetilde{A}_m-\widetilde{A}_m(0))\|_{L^\infty(\mathcal{H}_R)}\leq c(R r_m)^{\eps}$, $\|\widetilde{B}_m\|_{L^\infty(\mathcal{H}_R)}\leq cr_m$ by \eqref{expansion:ass.coeff}. Now using these observations, $g_m(0)=0$, \eqref{expansion:ass.coeff}, \eqref{expansion:limit.cmrm}, we have for each $R\geq1$
\begin{equation}\label{expansion:limit.dataseq}
\begin{aligned}
    &\lim_{m\to\infty}\left\|\frac{C_{m,r_m}(\widetilde{A}_m-\widetilde{A}_m(0))}{r_m^\eps\theta(r_m)}\right\|_{C^{\eps}(\mathcal{H}_R)}+\lim_{m\to\infty}\left\|\frac{C_{m,r_m}\widetilde{B}_m}{r_m^\eps\theta(r_m)}\right\|_{L^\infty(\mathcal{H}_R)}= 0,
    \\&\lim_{m\to\infty}\left\|\widetilde{F}_m(S_{r_m}z)\right\|_{L^\infty(\mathcal{H}_R)}=\lim_{m\to\infty}\left\|\frac{r_m^{2-(\lambda_{\alpha,b}+\eps)}F_m(S_{r_m}z)}{\theta(r_m)}\right\|_{L^\infty(\mathcal{H}_R)}=0,\\
    &
     \lim_{m\to\infty}\left\|\widetilde{g}_m\right\|_{C^{\lambda_{\alpha,b}+\eps}(\mathcal{H}_R)}= \lim_{m\to\infty}\left\|\frac{g_m(S_{r_m}z)-g_m(0)}{r_m^{\lambda_{\alpha,b}+\eps}\theta(r_m)}\right\|_{C^{\lambda_{\alpha,b}+\eps}(\mathcal{H}_R)}=0,
\end{aligned}
\end{equation}
where we have also used the fact that $\theta(r)\to\infty$ as $r\to0$ for the last limit.
Now using \eqref{expansion:limit.dataseq} and \eqref{expansion:linfty} to get $\|\widetilde{f}_m\|_{L^\infty(H_{2R})}\leq c(R)$, we can apply a rescaled version of \autoref{lem.replace} to $\widetilde{f}_m$ so that for any $R\leq \frac{1}{2r_m}$, we get
\begin{align}\label{expansion:cgamma}
    \|\widetilde{f}_m\|_{C^\gamma(\mathcal{H}_R)}\leq c(n,\Lambda,\eps,\alpha,b,R)
\end{align}
for some $\gamma=\gamma(n,\Lambda,\alpha,b)$.  Since $\phi_{\alpha,b}(x_n,v_n)$ is smooth away from $\{x_n=0\}$ by \autoref{lem.phialpha}, by \cite[Lemma 2.5]{KiWe26} together with \eqref{expansion:linfty} and \eqref{expansion:limit.dataseq}, we get that for any $\mathcal{Q}_1(z_0)\Subset \{x_n>0\}$ with $R\leq1$,
\begin{align}\label{expansion:c1eps}
    \|\widetilde{f}_m\|_{C^{1,\eps}(\mathcal{Q}_1(z_0))}\leq c(n,\Lambda,\eps,\alpha,b,z_0)
\end{align}
 for some constant $c$ independent of $m$, which implies $\widetilde{f}_m\to \widetilde{f}_\infty$ in $C^{1,\eps}(\mathcal{Q}_1(z_0))$.
 
 Next, by testing \eqref{expansion:eq} with $\widetilde{f}_m\psi^2$, where $\psi\in C_c^\infty(\mathcal{Q}_{R})$ with $\psi\equiv1$ on $\mathcal{Q}_{R/2}$, we deduce
 \begin{align*}
    \int_{\mathcal{H}_{\frac{R}2}}|\nabla_v \widetilde{f}_m|^2&\leq c\Bigg(\|\widetilde{f}_m\|_{L^\infty(\mathcal{H}_R)}++C_{m,r_m}\left(\frac{\|\widetilde{A}_m-\widetilde{A}_m(0)\|_{L^\infty(\mathcal{H}_R)}+\|\widetilde{B}_m\|_{L^\infty(\mathcal{H}_R)}}{r_m^\eps\theta(r_m)}\right)\int_{\mathcal{H}_R}|\nabla_{v_n}\phi_{\alpha,b}|^2\\
    &\quad\qquad+\int_{\mathcal{H}_R}|\widetilde{F}_m|^2+\|\widetilde{g}_m\|_{L^\infty(\gamma_-\cap\mathcal{Q}_R(z_0))}\Bigg),
\end{align*}
 where $c=c(n,\Lambda,R)$. Now using \eqref{expansion:linfty}, \eqref{expansion:limit.dataseq}, and the fact that $\|\partial_{v_n}\phi_{\alpha,b}\|_{L^{2}(\mathcal{H}_R)}\leq c(R)$, we derive
\begin{align}\label{expansion:l2grad}
    \|\nabla_v\widetilde{f}_m\|_{L^2(\mathcal{H}_R)}\leq c(n,\Lambda,\alpha,b,R).
\end{align}
Therefore, in light of \cite[Lemma D.1]{KiWe26} together with \eqref{expansion:limit.dataseq}-\eqref{expansion:l2grad}, by taking $m\to\infty$, we get
 \begin{equation}\label{expansion:limit.eq}
\left\{
\begin{alignedat}{3}
\partial_t\widetilde{f}_\infty+v\cdot\nabla_x\widetilde{f}_\infty-\ddiv(\widetilde{A}_\infty\nabla_v\widetilde{f}_\infty)&=0&&\qquad \mbox{in  $\{x_n>0\}$}, \\
\widetilde{f}_\infty(t,x',0,v',v_n)&=\alpha \widetilde{f}_\infty(t,x',0,v',-bv_n)&&\qquad  \mbox{in $\{x_n=0\}\times \{v_n>0\}$},\\
\|\widetilde{f}_\infty\|_{L^\infty({H}_R)}&\leq cR^{\lambda_{\alpha,b}+\eps}&&\qquad\text{for any }R\geq1,
\end{alignedat} \right.
\end{equation}
where the last condition follows from \eqref{expansion:linfty}. Moreover, as $\widetilde{g}_m$ uniformly converges to 0 and $\widetilde{f}_m$ also uniformly converges to $\widetilde{f}_\infty$ by \eqref{expansion:limit.dataseq} and \eqref{expansion:cgamma}, respectively, we obtain the boundary condition given in the second line of \eqref{expansion:limit.eq}.

By \autoref{lem.liou.nd}, we have $\widetilde{f}_\infty=C\phi_{\alpha,b}$ for some constant $C$. Then by \eqref{expansion:l2proj}, using the definition of $\widetilde{f}_m$ and the fact that $\widetilde{f}_m\rightrightarrows\widetilde{f}_\infty$, we have 
\begin{align*}
    \lim_{m\to\infty}\int_{H_{1}}\widetilde{f}_mC\phi_{\alpha,b}\,dz=\int_{H_{1}}\widetilde{f}_\infty^2\,dz=0,
\end{align*} which gives $\widetilde{f}_\infty\equiv0$. This contradicts \eqref{expansion:compar}, as $\|\widetilde{f}_{\infty}\|_{L^\infty(\mathcal{H}_1)}=\lim\limits_{m\to\infty}\|\widetilde{f}_{m}\|_{L^\infty(\mathcal{H}_1)}\geq\frac12$. This completes the proof.
\end{proof}

Combining \autoref{thm.expansion} and \autoref{lem.schhol.away-} together with a covering argument, we can prove the main result of this section.

\begin{proof}[Proof of \autoref{thm.opthol}]
    Since $z_0\in\gamma_0$, we may assume $z_0=0$ and $R=1$. Next, as in the scaling argument given in \cite[(5.4)]{KiWe26}, we may assume $\|B\|_{L^\infty(\mathcal{H}_1)}\leq 1$. Now we want to show that for any $z_1\in \mathcal{H}_{1/2}$ and $\rho\leq \frac{1}{128}$,
    \begin{align}\label{opthol:ineq0}
        \|f-f(z_1)\|_{L^\infty(\mathcal{H}_\rho(z_1))}\leq c\rho^{\lambda_{\alpha,b}}(\|f\|_{L^\infty(\mathcal{H}_1)}+\|F\|_{L^\infty(\mathcal{H}_1)}+[g]_{C^{\lambda_{\alpha,b}+\eps}(\gamma_-\cap \mathcal{Q}_1)}),
    \end{align}
    where $c=c(n,\Lambda,\alpha,b,\|A\|_{C^\eps(\mathcal{H}_1)})$. 
    
    From now on, for simplicity, all constants are understood to depend only on $n,\Lambda,\alpha,b,\|A\|_{C^\eps(\mathcal{H}_1)}$, and we do not indicate this dependence explicitly.

    First, we note from \autoref{thm.expansion} that for any $z_0=(t_0,x_0',x_{0,n},v_0',v_{0,n})\in \mathcal{H}_{\frac12}$ with $x_{0,n}=v_{0,n}=0$, such that $z_0\in\gamma_0$, 
    \begin{align*}
        \|f-f(z_0)-C_0\phi_{\alpha,b}\|_{L^\infty(\mathcal{H}_\rho(z_0))}\leq \rho^{\lambda_{\alpha,b}+\eps}(\|f\|_{L^\infty(H_{1/4}(z_0))}+\|F\|_{L^\infty(H_{1/4}(z_0))}+[g]_{C^{\lambda_{\alpha,b}+\eps}(\gamma_-\cap Q_{1/4}(z_0))}),
    \end{align*}
     whenever $\rho\leq \frac14$. Moreover taking $\rho=\frac14$, we have 
     \begin{align*}
         |C_0|\leq c(\|f\|_{L^\infty(H_{1/4}(z_0))}+\|F\|_{L^\infty(H_{1/4}(z_0))}+[g]_{C^{\lambda_{\alpha,b}+\eps}(\gamma_-\cap Q_{1/4}(z_0))}).
     \end{align*}
     Combining this with \autoref{lem.phialpha} and using that $x_{0,n}=v_{0,n}=0$ yields that for any $z_0\in\gamma_0$
    \begin{align}\label{opthol:ineq1}
        \osc_{\mathcal{H}_\rho(z_0)}f\leq 2\|f-f(z_0)\|_{L^\infty(\mathcal{H}_\rho(z_0))}\leq c\rho^{\lambda_{\alpha,b}}(\|f\|_{L^\infty(H_{1})}+\|F\|_{L^\infty(H_{1})}+[g]_{C^{\lambda_{\alpha,b}+\eps}(\gamma_-\cap \mathcal{Q}_{1})}).
    \end{align}
    Using this, \autoref{lem.schhol.away-} and a covering argument, we are going to deduce \eqref{opthol:ineq0}. In fact, the proof is identical to that of \autoref{lem.holsmall}. More precisely, the key estimate \eqref{holsmall:des2} is obtained by combining the oscillation estimate \eqref{holsmall:ineq0.osc} near the grazing set, the H\"older estimates away from $\gamma_0$ provided by \autoref{lem.calpha}, and a covering argument.  Although \autoref{lem.calpha} is stated up to $\gamma_0$, in the proof of \autoref{lem.holsmall}, we only use this result away from $\gamma_0$. Since \eqref{opthol:ineq1} and \autoref{lem.schhol.away-} play the roles of \eqref{holsmall:ineq0.osc} and \autoref{lem.calpha} (away from $\gamma_0$), respectively,  the estimate \eqref{opthol:ineq0} follows immediately by repeating the argument of \autoref{lem.holsmall} with $\beta=\lambda_{\alpha,b}$. Since we have established \eqref{opthol:ineq0}, this yields \eqref{opthol:main1}. For \eqref{opthol:main2}, first note that applying the same argument as in \autoref{thm.expansion} with $\lambda_{\alpha,b}$ replaced by $\delta$, we have for any $r\leq \frac{R}2$ and $z_0\in\gamma_0$
     \begin{align*}
         \|f-f(z_0)\|_{L^\infty(\mathcal{H}_R(z_0))}\leq c\left(\frac{r}{R}\right)^{\delta}\left(\|f\|_{L^\infty(\mathcal{H}_R(z_0))}+R^2\|F\|_{L^\infty(\mathcal{H}_R(z_0))}+R^\delta[g]_{C^\delta(\gamma_-\cap \mathcal{Q}_R(z_0))}\right).
     \end{align*}
     Therefore, we proceed exactly as in the proof of \eqref{opthol:main1} with $\lambda_{\alpha,b}$ replaced by $\delta$ to obtain \eqref{opthol:main2}. This completes the proof.

\end{proof}

\section{Maxwell and super-elastic boundary condition}
\label{sec:main}

In this section, we prove optimal regularity results for Maxwell- and super-elastic boundary conditions and prove our main results \autoref{thm:main-Maxwell}, \autoref{thm:main-Holder}, and \autoref{thm:main-general}. 

\subsection{Maxwell boundary condition}
First, we prove the boundary regularity with Maxwell boundary condition in the half space $\{x_n>0\}$. Next, using this together with a flattening argument and a covering argument, we establish the optimal regularity in a general domain $\Omega$ (see \autoref{thm:Maxwell-body}).

\begin{theorem}\label{thm.maxhalf}
Let $f$ be a weak solution to 
     \begin{equation}\label{maxhalf:eq}
\left\{
\begin{alignedat}{3}
\partial_tf+v\cdot\nabla_xf-\ddiv(A\nabla_vf)&=B\cdot\nabla_vf+F&&\qquad \mbox{in  $\mathcal{H}_{2,\infty}$}, \\
f&=\alpha \mathcal{R}f+(1-\alpha)\mathcal{N}f&&\qquad  \mbox{in $\gamma_-$},
\end{alignedat} \right.
\end{equation}
where $\mathcal{H}_{2,\infty}\coloneqq I_2\times (B_{8}'\times (0,8))\times \bbR^n$ and 
\begin{align}\label{maxhalf:defn.nf}
    \mathcal{N}f(t,x,v)\coloneqq \mathcal{M}(t,x,v)\int_{\bbR^n}f(t,x',0,w)(w_n)_-\,dw.
\end{align}
Fix $p\geq 2$ and denote $r_0\coloneqq \frac{\langle v_0\rangle^{-p}}{64}$ for any given $v_0\in \bbR^n$. Let $\eps\in(0,1)$, $\lambda\in\N$, and $M>0$. Then we have the following.
\begin{itemize}
    \item Assume 
    \begin{align}\label{maxhalf:ass1}
        \sup_{z_0\in \mathcal{H}_{2,\infty}}\left(\langle v_0\rangle^{M}\|\mathcal{M}\|_{C^{\lambda_\alpha+\eps}(\mathcal{H}_{r_0}(z_0))}+\|A\|_{C^\eps(\mathcal{H}_{r_0}(z_0))}+\langle v_0\rangle^{-2}\|B\|_{L^\infty(\mathcal{H}_{r_0}(z_0))}\right)\leq c_B
    \end{align}
    for some $c_B>0$. Then there is a constant $M_0=M_0(n,p)$ such that if $M>M_0$, then for any $z_0\in\mathcal{H}_{1,\infty}$, we have 
    \begin{align}\label{maxhalf:maineq}
        \|f\|_{C^{\lambda_\alpha}(\mathcal{H}_{r_0}(z_0))}&\leq c\langle v_0\rangle^{q}(\|f\|_{L^1(\mathcal{H}_{2r_0}(z_0))}+\|F\|_{L^\infty(\mathcal{H}_{2r_0}(z_0))})\nonumber\\
        &\quad+c\langle v_0\rangle^{q}(\|f\|_{L^1(\mathcal{H}_{2r_0}(R(z_0)))}+\|F\|_{L^\infty(\mathcal{H}_{2r_0}(R(z_0)))})\\
        &\quad+c\langle v_0\rangle^{q} \|\mathcal{M}\|_{C^{\lambda_{\alpha}+\eps}(\mathcal{H}_{2r_0}(z_0))}\|\langle \cdot\rangle^qf(t,x,\cdot)\|_{L^\infty(I_2\times B_8'\times(0,8);L^1(\bbR^n))}\nonumber\\
        &\quad+c\langle v_0\rangle^{q} \|\mathcal{M}\|_{C^{\lambda_{\alpha}+\eps}(\mathcal{H}_{2r_0}(z_0))}\|\langle \cdot\rangle^qF(t,x,\cdot)\|_{L^\infty(\mathcal{H}_{2,\infty})},\nonumber
    \end{align}
    where $c=c(n,\Lambda,\alpha,p,c_B)$, $q=q(n,p)$, $R(z_0) \coloneqq (t_0,x_0,v_0,-v_{0,n})$ and the constant $\lambda_{\alpha}\equiv \lambda_{\alpha,1}$ is determined in \eqref{phialpha:defn.lambdaalpha}.
    \item Assume 
    \begin{align}\label{maxhalf:ass2}
        \sup_{z_0\in \mathcal{H}_{2,\infty}}\left(\langle v_0\rangle^{M}\|\mathcal{M}\|_{C^{6\lambda+\frac76}(\mathcal{H}_{r_0}(z_0))}+\|A\|_{C^{6\lambda+1}(\mathcal{H}_{r_0}(z_0))}+\langle v_0\rangle^{-2}\|B\|_{C^{6\lambda}(\mathcal{H}_{r_0}(z_0))}\right)\leq c_B
    \end{align}
    for some $c_B>0$. Then there is a constant $M_1=M_1(n,\lambda,\eps,p,\alpha)$ such that if $M>M_1$, then for any $z_0\in\mathcal{H}_{1,\infty}$ and $\mathcal{H}_{2r_0}(z_0)\cap \gamma_0=\emptyset$,
    \begin{align}\label{maxhalf:maineq2}
        \|f\|_{C^{\lambda,\eps}(\mathcal{H}_{r_0}(z_0))}&\leq c\left(\|\langle \cdot\rangle^qf(t,x,\cdot)\|_{L^\infty(\mathcal{H}_{2,\infty})}+\|\langle \cdot\rangle^qF(t,x,\cdot)\|_{C^{6\lambda}(\mathcal{H}_{2,\infty})}\right),
    \end{align}
    where $c=c(n,\Lambda,\eps,\lambda,\alpha,c_B,p)$, $q=q(n,\lambda,\eps,p,\alpha)$ and we write
    \begin{align}\label{maxhalf:defn.wenorm}
        \|\langle \cdot\rangle^qF(t,x,\cdot)\|_{C^{6\lambda}(\mathcal{H}_{2,\infty})}\coloneqq \sup_{z_0\in\mathcal{H}_{2,\infty}}\langle v_0\rangle^q\|F\|_{C^{6\lambda}(\mathcal{H}_{1}(z_0)\cap \mathcal{H}_{2,\infty})}.
    \end{align}
\end{itemize}
\end{theorem}

A key step in the proof is to establish sufficient regularity for $\mathcal{N}f$, which we establish in \autoref{lem.Nfreg}. This will allow us to apply \autoref{thm.opthol} and \autoref{lem.schhol.away-}, to prove \eqref{maxhalf:maineq} and \eqref{maxhalf:maineq2}, respectively. 
Having at hand \autoref{thm.opthol}, this argument is relatively straightforward for the proof of the first claim \eqref{maxhalf:maineq}. The proof of the second claim \eqref{maxhalf:maineq2} uses \autoref{lem.schhol.away-} instead of \autoref{thm.opthol}, however it is a lot more technical since we require higher regularity of $\cN f$. Our proof follows a similar procedure as the one given in \cite[Section 7]{KiWe26}. The scheme using difference quotient techniques and the inductive argument as following:
\begin{alignat}{7}\label{max:procedure}
& f\in L^\infty
& \sqarrow{1}&\; \mathcal{N}f\in C^{\frac76}
& \sqarrow{2}&\; f\in C^{\lambda_0}
& \sqarrow{3}&\; \frac{\delta_{h}f}{|h|^{\frac{\lambda_0}3}}\in L^\infty
\nonumber\\
& \frac{\delta_{h}f}{|h|^{\frac{\lambda_0}3}}\in L^\infty
& \sqarrow{1}&\; \mathcal{N}\frac{\delta_hf}{|h|^{\frac{\lambda_0}3}}\in C^{\frac76}
& \sqarrow{2'}&\; \frac{\delta_hf}{|h|^{\frac{\lambda_0}3}}\in C^{\lambda_0-\delta}
& \sqarrow{3}&\; \frac{\delta_{h}f}{|h|^{\frac{2\lambda_0}3-\eps}}\in L^\infty
\nonumber\\
& \frac{\delta_{h}f}{|h|^{\frac{2\lambda_0}3-\eps}}\in L^\infty
& \sqarrow{1}&\; \mathcal{N}\frac{\delta_hf}{|h|^{\frac{2\lambda_0}3-\eps}}\in C^{\frac76}
& \sqarrow{2'}&\; \frac{\delta_hf}{|h|^{\frac{2\lambda_0}3-\eps}}\in C^{\lambda_0-\delta}
& \sqarrow{3}&\; \frac{\delta_{h}f}{|h|^{\frac{3\lambda_0}3-2\eps}}\in L^\infty\\
& \vdots
& \sqarrow{1}&\; \vdots
& \sqarrow{2'}&\; \vdots
& \sqarrow{3}&\; \vdots
\nonumber\\
& \frac{\delta_{h,k}f}{|h|^{k+\frac{\lambda_0}3-\eps}}\in L^\infty
& \sqarrow{1}&\; \mathcal{N}\frac{\delta_{h,k}f}{|h|^{k+\frac{\lambda_0}3-\eps}}\in C^{\frac76}
& \sqarrow{2'}&\; \frac{\delta_{h,k}f}{|h|^{k+\frac{\lambda_0}3-\eps}}\in C^{\lambda_0-\delta}
& \sqarrow{3}&\; \frac{\delta_{h,k}f}{|h|^{k+\frac{2\lambda_0}3-2\eps}}\in L^\infty,\nonumber
\end{alignat}
where $\lambda_0\coloneqq\frac12$, $\delta_hf(z)$ is $f(t+h,x,v)-f(t,x,v)$ or $\delta_hf(z)=f(t,x+he_i,v)-f(t,x,v)$ with the $i$-th unit vector $e_i$ for some $i\in\{1,2,\ldots ,n-1\}$ and $\delta_{h,k}=\delta_h\circ\delta_{h,{k-1}}f$. 

Now we explain the arguments used in each step labeled $(1),(2),(2'),(3)$.
\begin{itemize}
    \item[(1)] We use the regularity at $\gamma_+$ of $\frac{\delta_{h,k}f}{|h|^{k+\alpha}}$ together with the boundedness of $\frac{\delta_{h,k}f}{|h|^{k+\alpha}}$ for some $\alpha\in[0,1)$ and $k\in \N\cup\{0\}$ (see \textbf{Step 1} in the proof of \cite[Proposition 7.2]{KiWe26}). 
    \item[(2)] We apply the $C^{\lambda_0}$ estimate for the in-flow condition when the right-hand side is bounded and the boundary data belongs to $C^{\frac76}$. 
    \item[(3)] This step follows by elementary properties of kinetic H\"older spaces (see  \cite[(A.2)]{KiWe26}).
    \item[(2')] This step involves several ingredients. First, we observe that the function $\frac{\delta_{h,k}f}{|h|^{k+\alpha}}$ solves the equation whose right-hand side contains terms of the form $-\ddiv(\frac{\delta_hA}{|h|}\nabla_v\frac{\delta_{h,k-1}f}{|h|^{k-1+\alpha}})$, which are generally unbounded. By the inductive hypothesis, we have uniform $C^{\lambda_0-\delta}$ estimates for  $\frac{\delta_{h,k-1}f}{|h|^{k-1+\alpha}} $ as well as uniform estimates for $\mathcal{N}\frac{\delta_{h,k-1}f}{|h|^{k-1+\alpha}}\in C^{\frac76}$ at the boundary. The latter implies uniform $C^{\frac76}$ estimates of $\nabla_v\frac{\delta_{h,k-1}f}{|h|^{k-1+\alpha}} $ away from $\gamma_0$. Combining the H\"older estimates of $\frac{\delta_{h,k-1}f}{|h|^{k-1+\alpha}}$ with the gradient estimate away from $\gamma_0$,  $\nabla_v\frac{\delta_{h,k-1}f}{|h|^{k-1+\alpha}}$ belongs to a suitable $L^q$ space up to $\gamma_0$. Then, by the perturbation argument given in \cite[Lemma 7.5]{KiWe26} and \textbf{Step 3} in \cite[Proposition 7.2]{KiWe26}, it follows that $\frac{\delta_{h,k}f}{|h|^{k+\alpha}}\in C^{\lambda_0-\delta}$. 
\end{itemize}
This way we prove inductively $\mathcal{N}\frac{\delta_{h,k}f}{|h|^{k}}\in C^{\frac76}$ for each $k\in\N\cup\{0\}$.

Now we adapt this argument to the present setting. Observe that steps (1) and (3) remain unchanged, since they do not rely on any boundary condition.
In step (2), however, an addition argument is needed. Near $\gamma_0$ and away from $\gamma_-$, $C^{\lambda_\alpha}$ regularity follows from \autoref{thm.opthol} and \eqref{lem.schhol.away-.depB}, respectively. The main issue arises in estimating the $C^{\lambda_\alpha}$ norm on a cylinder $\mathcal{H}_{r_0}(z_0)$ located near $\gamma_-$. In this case, the reflection operator $\mathcal{R}f$ requires control of $\|f\|_{C^{\lambda_\alpha}(\mathcal{H}_{r_0}(R(z_0)))}$. This regularity follows from the corresponding regularity near $\gamma_+$ and is given in \eqref{lem.schhol.away-.depB}. As a result, we get additional terms of the form $\|f\|_{L^\infty(\mathcal{H}_{2r_0}(R(z_0)))}+\|F\|_{L^\infty(\mathcal{H}_{2r_0}(R(z_0)))}$ in our estimates (see the second line in the right-hand side of \eqref{maxhalf:maineq}, for instance). 

Step (2') proceeds similarly, except that the optimal H\"older exponent is now $\lambda_\alpha$. We apply a difference quotient operator to the equation, which creates an unbounded right-hand side of the form $-\ddiv(\frac{\delta_hA}{|h|^{\frac{\lambda_{\alpha}}3}}\nabla_vf)$. Since the difference quotient are merely taken in tangential directions, the boundary condition is preserved (see \eqref{Nfreg:eq.diffquo}). Next, note that we have $\mathcal{R}f, \mathcal{N}f\in C^{\frac76}$ at $\gamma_-$, which gives $\nabla_vf\in C^{\frac76}$ away from $\gamma_0$. Combining this gradient regularity with the fact that $f\in C^{\lambda_\alpha}$, we deduce $\nabla_vf\in L^q$ for some suitable $q$. Thus, by the perturbation argument of \autoref{lem.pertur} below, we can prove $\frac{\delta_hf}{|h|^{\lambda_\alpha/3}}\in C^{\lambda_\alpha-\delta}$ for any $\delta>0$ near $\gamma_0$. On the other hand, as $\frac{\delta_{h}f}{|h|^{\lambda_\alpha/3}}\in L^\infty$,  \eqref{lem.schhol.away-.depB} further implies $\frac{\delta_hf}{|h|^{{\lambda_\alpha}/3}}\in C^{\lambda_\alpha-\delta}$ for any $\delta>0$ away from $\gamma_-\cup\gamma_0$. Since we have proved H\"older estimates near $\gamma_0$ and away from $\gamma_-$, the H\"older estimate of $\frac{\delta_hf}{|h|^{\lambda_\alpha/3}}$ near $\gamma_-$ follows exactly as in the case of $f$ (see step (2) above). Consequently, we inductively repeat the procedure  described in \eqref{max:procedure} with $\lambda_0$ replaced by $\lambda_\alpha$.   

Based on the discussion above, we will prove the regularity of $\mathcal{N}f$. Since many parts of the proof closely follow from the same arguments as in \cite{KiWe26}, we omit the corresponding details and refer the reader to \cite{KiWe26}.

\begin{lemma}\label{lem.Nfreg}
    Let $f$ be a weak solution to \eqref{maxhalf:eq}. Fix $p\geq \frac12$ and write $r_0\coloneqq \frac{\langle v_0\rangle^{-p}}{64}$ for any $v_0\in\mathcal{H}_{2,\infty}$. 
    \begin{itemize}
        \item Assume \eqref{maxhalf:ass1}. Then there is a constant $M_0=M_0(n,p)$ such that if $M>M_0$, then for any $z_0\in \mathcal{H}_{1,\infty}$ and $\eps\in(\lambda_\alpha,1)$,
        \begin{align}\label{Nfreg:main1}
            [\mathcal{N}f]_{C^{\eps}(\gamma_-\cap \mathcal{Q}_{r_0}(z_0))}\leq c\langle v_0\rangle\|\mathcal{M}\|_{C^{\eps}(\gamma_-\cap\mathcal{Q}_{r_0}(z_0))}\left(\|\langle \cdot\rangle^qf(t,x,\cdot)\|_{L^\infty(\mathcal{H}_{2,\infty})}+\|\langle \cdot\rangle^qF(t,x,\cdot)\|_{L^\infty(\mathcal{H}_{2,\infty})}\right)
        \end{align}
        for some $c=c(n,\Lambda,\eps,\alpha,c_B,p)$ and $q=q(n,p,\eps)$.
        \item Assume \eqref{maxhalf:ass2}. Then there is a constant $M_1=M_1(n,\lambda,\eps,p,\alpha)$ such that if $M>M_1$, then for any $z_0\in \mathcal{H}_{1,\infty}$ and $\eps\in(0,1)$,
        \begin{align}\label{Nfreg:main2}
            [\mathcal{N}f]_{C^{\lambda,\eps}(\gamma_-\cap \mathcal{Q}_{r_0}(z_0))}\leq c\left(\|\langle \cdot\rangle^qf(t,x,\cdot)\|_{L^\infty(\mathcal{H}_{2,\infty})}+\|\langle \cdot\rangle^qF(t,x,\cdot)\|_{C^{6\lambda}(\mathcal{H}_{2,\infty})}\right)
        \end{align}
        for some $c=c(n,\Lambda,\eps,\alpha,c_B,p,\lambda,\mathcal{M})$ and $q=q(n,\lambda,p,\eps,\alpha)$.
    \end{itemize}
\end{lemma}
\begin{proof}
   Fix $k\in\N\cup\{0\}$ and define a function $g(z)$ by 
    \begin{align*}
        g(z)\coloneqq \int_{\bbR^n}f(t,x',0,w)(w_n)_-\,dw
    \end{align*}
    and
    \begin{align*}
        \widehat{g}_k(z)\coloneqq \delta_{h,k}g(z),
    \end{align*}
    where $\delta_{h,k}g=\delta_{h,k-1}\circ\delta_hg$ and $\delta_{h}g=g(t+h,x,v)-g(t,x,v)$. We assume 
\begin{align}\label{Nfreg:ass.M}
    \sup_{z_0\in\mathcal{H}_{1,\infty}}\langle v_0\rangle^{M}\sum_{i=0}^k\left\|\partial_t^i\mathcal{M}\right\|_{C^{\frac76}(\gamma_-\cap\mathcal{Q}_{32r_0}(z_0))}\leq 1
\end{align}
for some large $M > 1$. Based on an inductive argument, we will show that for any $z_0\in \mathcal{H}_{1,\infty}$, $r\in(0,r_0]$, and $|h|\leq 2^{-100k}$, if $M=M(n,k,p,\alpha)$ is sufficiently large, then 
\begin{align}\label{Nfreg:main.quo}
    \|\widehat{g}_k-a_r\|_{L^\infty(\gamma_-\cap \mathcal{Q}_r(z_0))}\leq cr^{\frac76}\langle v_0\rangle ^{c_k}D_k,
\end{align}
where $c_k=c_k(n,k,p,\alpha)$ and $c=c(n,\Lambda,k, C_k,\alpha)$, 
\begin{align*}
    &a_r\coloneqq \int_{\bbR^{n-1}}\int_{w_n<-16r-r_{w'}}\frac{\delta_{h,k}f}{|h|^k}(t_0,x_0',0,w',w_n)(w_n)_-\,dw_n\,dw',\\
 &   C_l\coloneqq\sup_{z_0\in \mathcal{H}_{\frac32,\infty}}\left\|{A}\right\|_{C^{3l+1}(\mathcal{H}_{32r_0}(z_0))}+\frac{1}{\langle v_0\rangle^2}\left\|B\right\|_{C^{3l}(\mathcal{H}_{32r_0}(z_0))},\\
  &D_k\coloneqq \|\langle \cdot\rangle^{c_k}f(t,x,\cdot)\|_{L^\infty(\mathcal{H}_{2,\infty})}+\left\|\langle\cdot\rangle^{c_k}{F(t,x,\cdot)}\right\|_{C^{3k}(\mathcal{H}_{2,\infty})}
\end{align*}
    with $r_{w'}\coloneqq 16(r^2|w'-v_0'|)^{\frac13}$. Note that the notation $\left\|\langle\cdot\rangle^{c_k}{F(t,x,\cdot)}\right\|_{C^{3k}(\mathcal{H}_{2,\infty})}$ is introduced in \eqref{maxhalf:defn.wenorm}.
    To this end, first by \cite[(7.25)]{KiWe26}, which uses only the kinetic geometry, we have 
    \begin{equation}\label{Nfreg:mainineq}
\begin{aligned}
    &|h|^{k}\|\widehat{g}_k-a_r\|_{L^\infty(\gamma_-\cap\mathcal{Q}_{r}(z_0))}\\
    &\leq cr^{\frac43}\langle v_0\rangle \|\langle\cdot\rangle^4\delta_{h,k}f(t,x',0,\cdot)\|_{L^\infty(\bbR^n)}\\
    &\quad+cr^{\frac{2(2-\eps)}{3}}\langle v_0\rangle\sup_{(t,x')\in \mathbf{B}_{\frac32}}{\int_{ \frac{1}{{\langle w\rangle}^{p}}\leq |w_n|}\int_{{w_n<-64r
   }}\langle w\rangle^5[\delta_{h,k}f]_{C^{2-\eps}(\mathcal{H}_{\frac{1}{64{\langle w\rangle}^p}}(t,x',0,w))}(w_n)_-\,dw}\\
    &\quad+ cr^{\frac{2(2-\eps)}{3}}\langle v_0\rangle\sup_{(t,x')\in \mathbf{B}_{\frac32}} \int_{\frac{1}{{\langle w\rangle}^{p}}\geq |w_n|}\int_{{w_n<-64r
    }}\langle w\rangle^5[\delta_{h,k}f]_{C^{2-\eps}(\mathcal{H}_{\frac{|w_n|}{64}}(t,x',0,w))}(w_n)_-\,dw,
\end{aligned}
\end{equation}
where $\mathbf{B}_{\frac32}\coloneqq I_{(\frac32)^2}\times B'_{(\frac32)^3}$ and $c=c(n,\eps,p)$. 

Next, we apply the difference quotient operator $\delta_{h,l}$ to \eqref{maxhalf:eq} so that for any $l\leq k$,
    \begin{equation}\label{Nfreg:eq.diffquo}
\left\{
\begin{alignedat}{3}
(\partial_t+v\cdot\nabla_x)\delta_{h,l}f-\ddiv(A_{l,h}\nabla_v\delta_{h,l}f)&=B_{l,h}\cdot\nabla_v\delta_{h,l}f+\delta_{h,l}F+G_l&&\quad \mbox{in  $\{x_n>0\}\cap \mathcal{H}_{\frac74,\infty}$}, \\
\delta_{h,l}f&=\alpha\delta_{h,l}\mathcal{R}f+(1-\alpha)\delta_{h,l}\mathcal{N}f&&\quad  \mbox{in $\gamma_-$},
\end{alignedat} \right.
\end{equation}
where 
\begin{align*}
    G_l\coloneqq \sum_{i=0}^{l}\ddiv(\delta_{h,l-i}A_{i,h}\nabla_v\delta_{h,l}f)+\delta_{h,l-i}B_{i,h}\cdot\nabla_v\delta_{h,i}f.
\end{align*}
From \cite[(7.30)]{KiWe26}, which uses only the regularity at $\gamma_+$, we have 
\begin{align}\label{Nfreg:c2eps}
    \left[\frac{\delta_{h,l}f}{|h|^l}\right]_{C^{2-\eps}(\mathcal{H}_\rho(z_1))}&\leq c\langle v_1\rangle^{4(l+1)}\sum_{i=0}^{l}\left(\rho^{-2+\eps} \left\|\frac{\delta_{h,i}f}{|h|^i}\right\|_{L^\infty(\mathcal{H}_{2\rho}(z_1))}+\left\|\frac{\delta_{h,i}F}{|h|^i}\right\|_{L^\infty(\mathcal{H}_{2\rho}(z_1))}\right),
\end{align}
where $c=c(n,l,\eps,c_B,C_l)$, $\rho\leq r_1=\frac{\langle v_1\rangle ^{-p}}{64}$ with $x_{1,n}=0$, and $v_{1,n}<0$. Indeed, we can choose
\begin{align*}
    C_0 \coloneqq \sup_{z_0\in \mathcal{H}_{\frac32,\infty}}\left\|{A}\right\|_{C^{\eps}(\mathcal{H}_{32r_0}(z_0))}+\frac{1}{\langle v_0\rangle^2}\left\|B\right\|_{L^{\infty}(\mathcal{H}_{32r_0}(z_0))}.
\end{align*}
Thus, plugging \eqref{Nfreg:c2eps} into \eqref{Nfreg:mainineq}, we can deduce \eqref{Nfreg:main1}. Now we will focus on \eqref{Nfreg:main.quo} with $k\geq1$.

As in \textbf{Step 3} given in the proof of \cite[Proposition 7.2]{KiWe26}, we will show that 
for any $j\in\{0,1,\ldots , k\}$ and $\delta\in(0,\lambda_\alpha)$, there is $N_j=N_j(n,j,p,\alpha)$ such that for any $z_2\in \mathcal{H}_{\frac32,\infty}$ and $\rho\leq {r_2}=\frac{1}{64{\langle v_2\rangle}^p}$,
\begin{equation}\label{Nfreg:ind}
\begin{aligned}
&\left[\frac{\delta_{h,j}f}{|h|^j}\right]_{C^{\lambda_\alpha-\delta}(\mathcal{H}_{2\rho}(z_2))}+\left\|\frac{\delta_{h,j}f}{|h|^j}\right\|_{L^\infty(\mathcal{H}_{2\rho}(z_2))}\\
&\leq c\langle v_2\rangle^{N_j}\left(\|f\|_{L^\infty(\mathcal{H}_{4r_2}(z_2))}+\left\|{F}\right\|_{C^{3j}(\mathcal{H}_{4r_2}(z_2))}+D_j\left\|{\mathcal{M}}\right\|_{C^{3_j+\frac76}(\gamma_-\cap\mathcal{Q}_{4r_2}(z_2))}\right)\\
&\quad +c\langle v_2\rangle^{N_j}\left(\|f\|_{L^\infty(\mathcal{H}_{4r_2}(R(z_2)))}+\left\|{F}\right\|_{C^{3j}(\mathcal{H}_{4r_2}(R(z_2)))}\right)
\end{aligned}
\end{equation}
where $c=c(n,\delta,j,C_j,p,\alpha)$ provided that $M=M(n,k,p,\alpha)$ given in \eqref{Nfreg:ass.M} is sufficiently large. First, note that we prove this with $\delta=0$ if $j=0$. To this end, when $v_{2,n}\leq \frac72r_2$ or $\mathcal{Q}_{3r_2}(z_2)\cap \gamma_-=\emptyset$, we use \autoref{thm.opthol} or \eqref{lem.schhol.away-.depB}, respectively. Then we derive \eqref{Nfreg:ind} with $j=0$ and $\delta=0$. Suppose $\mathcal{Q}_{3r_2}(z_2)\cap \gamma_-\neq\emptyset$ with $v_{2,n}>\frac72r_2$. Then by \eqref{lem.schhol.away-.depB} with $R=r_2(=\frac{1}{64\langle v_2\rangle^{p}})$, we have 
\begin{align*}
    &\left[f\right]_{C^{\lambda_\alpha}(\mathcal{H}_{2r_2}(z_2))}+\left\|f\right\|_{L^\infty(\mathcal{H}_{2r_2}(z_2))}\\
&\leq c\langle v_2\rangle^{\frac{1}2+\frac{p}{4}}\Bigg(\|f\|_{L^\infty(\mathcal{H}_{3r_2}(z_2))}+\left\|{F}\right\|_{L^{\infty}(\mathcal{H}_{3r_2}(z_2))}+D_0\left\|{\mathcal{M}}\right\|_{C^{\lambda_\alpha}(\gamma_-\cap\mathcal{Q}_{3r_2}(z_2))}+[\mathcal{R}f]_{C^{\lambda_\alpha}(\gamma_-\cap\mathcal{Q}_{3r_2}(z_2))}\Bigg).
\end{align*}
Since $\mathcal{Q}_{\frac{7r_2}2}(\mathcal{R}z_2)\cap\gamma_-=\emptyset$, by \eqref{Nfreg:c2eps} with $l$, $\rho$, and $z_1$, replaced by $0$, ${r_2}{}$, $\mathcal{R}z_2$, respectively and some covering argument, we deduce
\begin{align*}
    [\mathcal{R}f]_{C^{\lambda_\alpha}(\gamma_-\cap\mathcal{Q}_{3r_2}({R}(z_2)))}\leq c\langle v_2\rangle^{\frac12+\frac{p}{4}}\left(\|f\|_{L^\infty(\mathcal{H}_{4r_2}(R(z_2)))}+\left\|{F}\right\|_{L^{\infty}(\mathcal{H}_{4r_2}(R(z_2)))}\right).
\end{align*}
Combining the previous two inequalities leads to \eqref{Nfreg:ind} with $j=0$ whenever $\mathcal{Q}_{3r_2}(z_2)\cap \gamma_-\neq\emptyset$ with $v_{2,n}>\frac72r_2$. This verifies \eqref{Nfreg:ind} with $j=0$, in particular $\delta=0$.

Now we assume that \eqref{Nfreg:ind} holds for any $j\in\{0,1,2,\ldots, l\}$ and we will prove \eqref{Nfreg:ind} with $j=l+1$.

Following \cite[\textbf{Step 3-1} in Proposition 7.2]{KiWe26} with $\frac12(=\lambda_0)$ replaced by $\lambda_\alpha$, which requires only elementary computations with \eqref{Nfreg:mainineq}, \eqref{Nfreg:c2eps}, and \eqref{Nfreg:ind}, we have for $\alpha<\frac{\lambda_\alpha}{3}$,
\begin{align}\label{Nfreg:ind.g}
    \left\|\frac{\delta_{h,l+1}g}{|h|^{l+\alpha}}\right\|_{C^{\frac76}(\gamma_-\cap\mathcal{Q}_{r_0}(z_0))}\leq cD_{l+1}\quad\text{and}\quad \left\|\frac{\delta_{h,j}g}{|h|^{j}}\right\|_{C^{\frac76}(\gamma_-\cap\mathcal{Q}_{r_0}(z_0))}\leq cD_j
\end{align}
for any $j\leq l$, where $c=c(n,\alpha,l,\eps,c_B,C_l,p)$. 
As in \cite[\textbf{Step 3-2} in Proposition 7.2]{KiWe26}, we have for any $\mathcal{Q}_{4\rho}(z_1)\cap \gamma_0=\emptyset$,
\begin{align*}
   & \rho^{-\frac16}\left\|\frac{\nabla_v\delta_{h,l}f}{|h|^{l}}\right\|_{L^{\infty}(\mathcal{H}_\rho(z_1))}+ \left[\frac{\nabla_v\delta_{h,l}f}{|h|^{l}}\right]_{C^{\frac16}(\mathcal{H}_\rho(z_1))}\\
    &\leq c\langle v_1\rangle^{4(l+1)}\sum_{i=0}^l\rho^{-\frac76} \left\|\frac{\delta_{h,i}f-\delta_{h,i}f_{}(z_1)}{|h|^i}\right\|_{L^\infty(\mathcal{H}_{2\rho}(z_1))}\\
    &\quad+c\langle v_1\rangle^{4(l+1)}\sum_{i=0}^l\left(\left\|\frac{\delta_{h,i}F}{|h|^i}\right\|_{L^\infty(\mathcal{H}_{2\rho}(z_1))}+\left[\frac{\delta_{h,i}\mathcal{\mathcal{R}}f}{|h|^i}\right]_{C^{\frac76}(\gamma_-\cap\mathcal{Q}_{2\rho}(z_1))}+\left[\frac{\delta_{h,i}\mathcal{N}f}{|h|^i}\right]_{C^{\frac76}(\gamma_-\cap\mathcal{Q}_{2\rho}(z_1))}\right),
\end{align*} 
where $c=c(n,l,c_B,C_l)$. The only difference compared to the first inequality given in \cite[\textbf{Step 3-2} in Proposition 7.2]{KiWe26} is the term $\left[\frac{\delta_{h,i}\mathcal{\mathcal{R}}f}{|h|^i}\right]_{C^{\frac76}(\gamma_-\cap\mathcal{Q}_{2\rho}(z_1))}$. From \autoref{lem.geo} with $b=1$, \eqref{Nfreg:c2eps} with $\eps=\frac56$ and \eqref{Nfreg:ind}, we deduce
\begin{align*}
    \left[\frac{\delta_{h,i}\mathcal{\mathcal{R}}f}{|h|^i}\right]_{C^{\frac76}(\gamma_-\cap\mathcal{Q}_{2\rho}(z_1))}&\leq c\left[\frac{\delta_{h,i}{}f}{|h|^i}\right]_{C^{\frac76}(\mathcal{Q}_{2\rho}(R(z_1)))}\\
    &\leq c\rho^{-\frac76}\langle v_1\rangle^{c_i}\left(\|f\|_{L^\infty(\mathcal{H}_{4r_1}(R(z_1)))}+\left\|{F}\right\|_{C^{3(i+1)}(\mathcal{H}_{4r_1}(R(z_1)))}\right)\\
    &\quad+c\rho^{-\frac76}D_l\langle v_1\rangle^{c_i}\left\|{\mathcal{M}}\right\|_{C^{3i+\frac76}(\gamma_-\cap\mathcal{Q}_{4r_1}(R(z_1)))}.
\end{align*}
Combining the previous two inequalities and following the same lines as in \cite[(7.38) in Proposition 7.2]{KiWe26} with (7.31) and (7.37) in \cite{KiWe26} replaced by \eqref{Nfreg:ind} and \eqref{Nfreg:ind.g}, respectively, we have
\begin{equation*}
\begin{aligned}
    \left\|\frac{\delta_{h,l}\nabla_vf}{|h|^l}\right\|_{L^{\infty}(\mathcal{H}_\rho(z_1))}&\leq c\langle v_1\rangle^{c_l}\rho^{-1-\delta+\lambda_\alpha}\left(\|f\|_{L^\infty(\mathcal{H}_{4r_1}(z_1))}+\left\|{F}\right\|_{C^{3l}(\mathcal{H}_{4r_1}(z_1))}\right)\\
    &\quad+c\langle v_1\rangle^{c_l}\rho^{-1-\delta+\lambda_\alpha}\left(\|f\|_{L^\infty(\mathcal{H}_{4r_1}(R(z_1)))}+\left\|{F}\right\|_{C^{3l}(\mathcal{H}_{4r_1}(R(z_1)))}\right)\\
    &\quad+cD_l\langle v_1\rangle^{c_l}\rho^{-1-\delta+\lambda_\alpha}\left\|\mathcal{M}\right\|_{C^{3l+\frac76}(\gamma_-\cap\mathcal{Q}_{4r_1}(z_1))},
\end{aligned}
\end{equation*}
where $c_l=c_l(n,l,p,\alpha)$ and $c=c(n,\alpha,\delta,l,c_B,C_l,p)$.
The remainder of the argument follows exactly as in \cite[\textbf{Step 3} in Proposition 7.2]{KiWe26} except that we apply \autoref{lem.pertur} in place of \cite[Lemma 7.5]{KiWe26}. Moreover, the number of iterations in \cite[\textbf{Step 3} in Proposition 7.2]{KiWe26} becomes $[3/\lambda_{\alpha}]$ instead of $6$. In particular, the exponent $8-\delta_0(=\frac{8}{1+2\delta})$ used in \textbf{Step 3} is replaced by $\frac{4}{1+\delta-\lambda_\alpha}$.  Thus, we have \eqref{Nfreg:ind} when $M=M(n,k,p,\alpha)$ is sufficiently large. Lastly, by repeating the argument in \cite[\textbf{Step 4} in Proposition 7.2]{KiWe26} together with the fact that taking the difference quotient operator in the $x'$-variables does not change the boundary condition, we deduce \eqref{Nfreg:main2}. This completes the proof.
\end{proof}

Next, we provide an auxiliary lemma that was used in the proof of \autoref{lem.Nfreg}. This lemma shows sharp H\"older estimates at the grazing set $\gamma_0$ when the right-hand side is given by unbounded functions. 
\begin{lemma}\label{lem.pertur}
  Let $z_0\in\gamma_0$ and  let $f$ be a weak solution to 
      \begin{equation*}
\left\{
\begin{alignedat}{3}
\partial_tf+v\cdot\nabla_xf-\ddiv(A\nabla_vf)&=B\cdot\nabla_vf+F-\ddiv(G)&&\qquad \mbox{in  $\mathcal{H}_{R}(z_0)$}, \\
f&=\alpha \mathcal{R}f+(1-\alpha)g&&\qquad  \mbox{in $\gamma_-$},
\end{alignedat} \right.
\end{equation*}
where $A\in C^\eps(\mathcal{H}_R(z_0))$. Fix $p\geq\frac12$, $C_1\geq1$ and write $r_1\coloneqq \frac{\langle v_1\rangle^{-p}}{64}$ and $\widehat{r_1}\coloneqq \frac{\langle v_1\rangle^{-p}}{C_1}$ for any given $z_1$. Fix $\delta<\lambda_\alpha$, $R=2r_0 =\frac{\langle v_0\rangle^{-p}}{64}$ and assume that for some $c_B$ and $C_0$
\begin{equation*}
\begin{aligned}
    &\langle v_0\rangle^{-2}\|B\|_{L^\infty(\mathcal{H}_R(z_0))}\leq c_B,\\
    &\sup_{z_1\in\gamma_0\cap\mathcal{Q}_R(z_0)}\sup_{(t,x',v')\in \mathcal{H}'_{\widehat{r}_1}(t_1,x_1',v_1')}\left(\|G(t,x',\cdot,v',\cdot)\|_{L^{\frac4{1+\delta-\lambda_\alpha}}({H}_{\widehat{r}_1})}+\|F(t,x',\cdot,v',\cdot)\|_{L^{\frac4{1+\delta-\lambda_\alpha}}({H}_{\widehat{r}_1})}\right)\leq C_0,\\
    &\sup_{z_2\in\mathcal{Q}_R(z_0)}\sup_{\substack{\mathcal{Q}_\rho(z_2)\cap\gamma_0=\emptyset,\\
    \rho\leq \frac{\langle v_2\rangle^{-p}}{64}}}\left(\|G\|_{L^\infty(\mathcal{H}_\rho(z_2))}+\|F\|_{L^\infty(\mathcal{H}_\rho(z_2))}\right)\leq C_0.
\end{aligned}
\end{equation*}
Then we have 
\begin{align}\label{pertur:maineq}
    [f]_{C^{\lambda_\alpha-\delta}(\mathcal{H}_{r_0}(z_0))}\leq c\langle v_0\rangle^{c_p}\left(\|f\|_{L^\infty(\mathcal{H}_{2r_0}(z_0))}+[g]_{C^{\frac34}(\gamma_-\cap \mathcal{Q}_{2r_0}(z_0))}+C_0\right),
\end{align}
where $c_p=c(n,p)$ and $c=c(n,\Lambda,\eps,\alpha,\|A\|_{C^\eps(\mathcal{H}_{2r_0}(z_0))},c_B,\delta,C_0,p)$.
\end{lemma}
\begin{proof}
    Let $h$ be a weak solution to 
     \begin{equation*}
\left\{
\begin{alignedat}{3}
\partial_th+v\cdot\nabla_xh-\ddiv(A\nabla_vh)&=B\cdot\nabla_vh+F-\ddiv(G)&&\qquad \mbox{in  $\mathcal{H}_{R}(z_0)$}, \\
h&=0&&\qquad  \mbox{in $\partial_{\mathrm{kin}}\mathcal{H}_{R}(z_0)$}.
\end{alignedat} \right.
\end{equation*}
Then by \cite[Lemma 7.5]{KiWe26} with $\delta$ replaced by $\frac12-\lambda_\alpha+\delta$, we deduce 
\begin{align*}
    \|h\|_{C^{\lambda_\alpha-\delta}(\mathcal{H}_{3r_0/2}(z_0))}\leq c\langle v_0\rangle^{c_p}\left(\|h\|_{L^2(\mathcal{H}_{2r_0}(z_0))}+C_0\right),
\end{align*}
where $c=c(n,\Lambda,\eps,\alpha,\|A\|_{C^\eps(\mathcal{H}_{2r_0}(z_0))},c_B,\delta,C_0,p)$ and $c_p=c_p(n,p)$. Even though \cite[(7.14)]{KiWe26} is written for $h \in L^\infty$, the proof of \cite[Lemma 7.5]{KiWe26} actually goes through for $h \in L^2$. Moreover, by the standard energy estimate as in \cite[Lemma 2.8]{Zhu24} together with the assumptions of $G$ and $F$, we have 
\begin{align*}
    \|h\|_{L^2(\mathcal{H}_{2r_0}(z_0))}\leq c(\|F\|_{L^2(\mathcal{H}_{2r_0}(z_0)}+\|G\|_{L^2(\mathcal{H}_{2r_0}(z_0)})\leq cC_0,
\end{align*}
where $c=c(n,\Lambda)$. Next, note that $f-h$ solves  
\begin{equation*}
\left\{
\begin{alignedat}{3}
\partial_t(f-h)+v\cdot\nabla_x(f-h)-\ddiv(A\nabla_v(f-h))&=B\cdot\nabla_v(f-h)&&\qquad \mbox{in  $\mathcal{H}_{R}(z_0)$}, \\
f-h&=\alpha\mathcal{R}_b(f-h)+(1-\alpha)\widetilde{g}&&\qquad  \mbox{in $\gamma_-$},
\end{alignedat} \right.
\end{equation*}
where we write $\widetilde{g}\coloneqq g+\frac{\alpha}{1-\alpha}\mathcal{R}_bh$. By \eqref{opthol:main2} and \autoref{lem.bdd}, we have
\begin{align*}
    [f-h]_{C^{\lambda_\alpha-\delta}(\mathcal{H}_{r_0}(z_0))}\leq c\langle v_0\rangle^{c_p}\left(\|f-h\|_{L^2(\mathcal{H}_{3r_0/2}(z_0))}+\|\widetilde{g}\|_{C^{\lambda_\alpha-\delta}(\gamma_-\cap \mathcal{Q}_{3r_0/2}(z_0))}\right).
\end{align*}
Thus, combining all the estimates together with the fact that $\|g\|_{L^\infty(\gamma_-\cap \mathcal{Q}_{2r_0}(z_0))}\leq c\|f\|_{L^\infty(\mathcal{H}_{2r_0}(z_0))}$ gives the desired estimate.
\end{proof}

Now we are ready to prove the main result of this section.

\begin{proof}[Proof of \autoref{thm.maxhalf}] To prove \eqref{maxhalf:maineq}, first we obtain $L^\infty$ estimates of $f$. Note from the definition of $\mathcal{N}f$ in \eqref{maxhalf:defn.nf}, we have 
\begin{align*}
    \|\mathcal{N}f\|_{L^\infty(\gamma_-\cap \mathcal{Q}_{2r_0}(z_0))}\leq  \|\mathcal{M}\|_{L^\infty(\gamma_-\cap \mathcal{Q}_{2r_0}(z_0))}\|\langle \cdot\rangle f(t,x,\cdot)\|_{L^\infty(I_2\times B_8'\times(0,8);L^1(\bbR^n))}.
\end{align*}
Thus, by \eqref{lem.calpha} together with the fact that $r_0\|B\|_{L^\infty(\mathcal{H}_{2r_0}(z_0))}\leq c_B$, we get
\begin{align*}
    \|f\|_{L^\infty(\mathcal{H}_{r_0}(z_0))}\leq c\langle v_0\rangle^{c_q}(\|f\|_{L^1(\mathcal{H}_{2r_0}(z_0))}+\|F\|_{L^\infty(\mathcal{H}_{2r_0}(z_0))}+\|\mathcal{N}f\|_{L^\infty(\gamma_-\cap \mathcal{Q}_{2r_0}(z_0))})
\end{align*}
for some constant $c=c(n,\Lambda,p,c_B)$. Using this and taking $M_0=M_0(n,p)$ large enough, we have 
\begin{align*}
    \|\langle\cdot\rangle^qf(t,x,\cdot)\|_{L^\infty(\mathcal{H}_{2,\infty})}&\leq c(\|\langle \cdot\rangle ^{q_0}f(t,x,\cdot)\|_{L^\infty(I_2\times B_8'\times(0,8);L^1(\bbR^n))}+\|\langle \cdot\rangle^{q_0}F(t,x,\cdot)\|_{L^\infty(\mathcal{H}_{2,\infty})}.
\end{align*}
Now using this and \eqref{Nfreg:ind} with $j=0$, we get the desired estimate \eqref{maxhalf:maineq}. 

To prove \eqref{maxhalf:maineq2}, we take the constant $M_2=M_2(n,p,\lambda,\eps,\alpha)$ to get \eqref{Nfreg:main2}. First if $\mathcal{Q}_{3r_0/2}(z_0)\cap \gamma_-=\emptyset$, then we have \eqref{maxhalf:maineq2} by \autoref{lem.schhol.away-}. If $\mathcal{Q}_{3r_0/2}(z_0)\cap \gamma_-\neq\emptyset$, then we have 
\begin{align*}
        \|f\|_{C^{\lambda,\eps}(\mathcal{H}_{r_0}(z_0))}&\leq c\left(\|\langle \cdot\rangle^qf(t,x,\cdot)\|_{L^\infty(\mathcal{H}_{2,\infty})}+\|\langle \cdot\rangle^qF(t,x,\cdot)\|_{C^{6\lambda}(\mathcal{H}_{2,\infty})}+\|\mathcal{R}f\|_{C^{\lambda,\eps}(\gamma_-\cap \mathcal{Q}_{3r_0/2}(z_0))}\right).
    \end{align*}
Since $\mathcal{Q}_{2r_0}(\mathcal{R}z_0)\cap \gamma_-=\emptyset$, using \autoref{lem.geo} with $b=1$ and \eqref{maxhalf:maineq2} with $z_0$ replaced by $\mathcal{R}z_0$, we also have the desired estimate when $\mathcal{Q}_{3r_0/2}(z_0)\cap \gamma_-\neq\emptyset$. This completes the proof.
\end{proof}

Next, we provide an example to verify the optimality of the main result \autoref{thm.maxhalf}. 
\begin{example}\label{ex:opt}
    Recall from \autoref{lem.phialpha} the solution $\phi_{\alpha}\equiv\phi_{\alpha,1}$ to 
    \begin{equation*}
\left\{
\begin{alignedat}{3}
v\partial_x\phi_\alpha-\partial_{vv}\phi_\alpha&=0&&\qquad \mbox{in  $\{x>0\}$}. \\
\phi_\alpha(0,v)&=\alpha \phi_\alpha(0,-v)&&\qquad  \mbox{in $\{x=0\}\times \{v>0\}$}
\end{alignedat} \right.
\end{equation*}
with $\phi_\alpha(0,v)=|v|^{\lambda_{\alpha}}$ in $v<0$, where $\lambda_{\alpha}\equiv \lambda_{\alpha,1}$. We can choose a smooth function $\psi=\psi(v)$ such that $\psi(v)=1$ in $(-1,1)$, $\psi\in C_c^\infty((-3,3))$, $\psi(-v)=\psi(v)$ and 
\begin{align*}
    \int_{v<0}|v|^{\lambda_{\alpha}}\psi(v)(-v)\,dv=0.
\end{align*}
Then we have 
\begin{align*}
    \int_{v<0}(-v)(|v|^{\lambda_{\alpha}}\psi(v)+2e^{-|v|^2})\,dv=1.
\end{align*}
Now we consider $h\coloneqq \phi_\alpha\psi+2e^{-|v|^2}$ and $\mathcal{M}(v)\coloneqq 2e^{-|v|^2}$ to see that 
\begin{align*}
    h(0,v)=(\alpha\phi_\alpha\psi)(0,-v)+2e^{-|v|^2}&=\alpha \mathcal{R}h(0,v)+2(1-\alpha)e^{-|v|^2}\\
    &=\alpha \mathcal{R}h(0,v)+(1-\alpha)\mathcal{M}(v) \int_{v<0}(-v)(|v|^{\lambda_\alpha}\psi(v)+2e^{-|v|^2})\,dv\\
    &=\alpha\mathcal{R}h(0,v)+(1-\alpha)\mathcal{N}h(0,v).
\end{align*}
Moreover, we observe that $h$ is a solution to 
\begin{equation*}
\left\{
\begin{alignedat}{3}
v\partial_xh-\partial_{vv}h&=F&&\qquad \mbox{in  $\{x>0\}$}, \\
h(0,v)&=\alpha \mathcal{R}h+(1-\alpha)\mathcal{N}h&&\qquad  \mbox{in $\{x=0\}\times \{v>0\}$},
\end{alignedat} \right.
\end{equation*}
where $F=-2\partial_v\psi\partial_v\phi_a-\phi_\alpha\partial_{vv}\psi-2\partial_{vv}(e^{-|v|^2})$. Since $\partial_v\psi=\partial_{vv}\psi\equiv0$ in $|v|<1$ and due to the fact that $\phi_a$ is a smooth function away from the origin, we have $F\in C^{\infty}_{\mathrm{loc}}([0,\infty)\times \bbR)$, but $h\in C^{\lambda_\alpha}(H_1)\setminus C^{\lambda_\alpha+\delta}(H_1)$ for any $\delta>0$.
\end{example}

We now prove the optimal regularity for the Maxwell boundary condition in general domains.

\begin{theorem}
\label{thm:Maxwell-body}
    Let $\Omega$ be a bounded $C^{2,\eps}$ domain for some $\eps>0$. Let $f$ be a weak solution to 
     \begin{equation*}
\left\{
\begin{alignedat}{3}
\partial_tf+v\cdot\nabla_xf-\ddiv(A\nabla_vf)&=B\cdot\nabla_vf+F&&\qquad \mbox{in  $(-1,1)\times \Omega\times \bbR^n$}, \\
f&=\alpha \mathcal{R}f+(1-\alpha)\mathcal{N}f&&\qquad  \mbox{in $\gamma_-$},
\end{alignedat} \right.
\end{equation*}
    where 
    \begin{align*}
        \mathcal{N}f(z)\coloneqq \mathcal{M}(t,x,v)\int_{\bbR^n}f(t,x,w) (n_x\cdot w)_-\,dw.
    \end{align*}
    Suppose
    \begin{align*}
        \sup_{z_1\in (-1,1)\times \Omega\times \bbR^n}\left(\langle v_1\rangle^M\|\mathcal{M}\|_{C^{\lambda_\alpha+\eps}(\mathcal{H}_1(z_1))}+\|A\|_{C^\eps(\mathcal{H}_1(z_1))}+\|B\|_{L^\infty(\mathcal{H}_1(z_1))}\right)\leq 1
    \end{align*}
    for some $M>0$ and $\lambda_\alpha\equiv \lambda_{\alpha,1}$ is determined in \eqref{phialpha:defn.lambdaalpha}.
    
    Then, there is a constant $M_0=M_0(n,\eps)$ such that if $M>M_0$, then we have 
    \begin{align*}
        \|f\|_{C^{\lambda_\alpha}((-\frac18,\frac18)\times\Omega\times \bbR^n)}\leq c\left(\|\langle\cdot\rangle^qf(t,x,\cdot)\|_{L^\infty((-1,1)\times\Omega;L^1( \bbR^n))}+\|\langle\cdot\rangle^qF(t,x,\cdot)\|_{L^\infty((-1,1)\times\Omega\times \bbR^n)}\right)
    \end{align*}
    for some $q=q(n)$ and $c=c(n,\Lambda,\eps,\alpha)$. 
    
    In addition, if $\Omega$, $A$, $B$, $F$ and $\mathcal{M}$ are smooth, and $F$ and $\mathcal{M}$ have fast decay as $|v|\to\infty$, then $f$ is smooth away from the grazing set.
\end{theorem}

Note that here, we denoted $\mathcal{H}_1(z_1) := \cQ_1(z_1) \cap \big((-1,1) \times \Omega \times \R^n \big)$.

\begin{proof}[Proof of \autoref{thm:main-Maxwell} and \autoref{thm:Maxwell-body}]
    We proceed by the flattening argument given in \cite[Section 5]{RoWe25} and \cite[Section 8]{KiWe26}. Repeating the proof from  \cite[Theorem 8.2]{KiWe26} together with \autoref{thm.maxhalf} (which corresponds to \cite[Theorem 7.1]{KiWe26}), the desired result follows immediately. This completes the proof.
\end{proof}

Lastly, we end this subsection with the proof of \autoref{thm:main-Holder}.

\begin{proof}[Proof of \autoref{thm:main-Holder}]
First, we consider a solution $f$ to \eqref{maxhalf:maineq} with the coefficient matrix $A$ being a merely measurable. 
By \autoref{lem.calpha}, we have $f\in C^\gamma(\mathcal{H}_{R}(z_0))$ whenever $\mathcal{H}_{2R}(z_0)\cap(\gamma_0\cup\gamma_-)=\emptyset$ for some small $\gamma=\gamma(n,\Lambda)$. Next, taking $\eps$ very close to 2 in \eqref{Nfreg:mainineq}, we indeed derive \eqref{Nfreg:main1} with small $\eps>0$. Thus, as in the proof of \autoref{thm.maxhalf} with \eqref{lem.schhol.away-.depB} and \autoref{thm.opthol} replaced by \autoref{lem.calpha}, we get \eqref{maxhalf:maineq} with
$\lambda_\alpha$ replaced by $\beta$, where $\beta=\beta(n,\Lambda,\alpha,\eps)$. Lastly, using the flattening transformation as in the proof of \autoref{thm:Maxwell-body}, we get the desired estimate. This completes the proof.
\end{proof}

\subsection{Generalization of Maxwell condition}
\label{subsec:ext}
In this subsection, we will consider two generalized models of the Maxwell boundary condition (see also \cite{Mis10} for a more detailed discussion) and explain how to obtain the optimal regularity for such models. 

The first model generalizes the constant $\alpha$ to a function $\alpha\equiv \alpha(t,x')$ with $\alpha\in[0,\alpha_0]$ for some $\alpha_0<1$. In particular, we are interested in the case that $\alpha$ is given by 
\begin{align}\label{defn.alpha.genmax}
    \alpha(t,x')\coloneqq \Psi\left(\int_{\bbR^n}f(t,x',0,w)(w_n)_-\,dw\right)
\end{align}
for some $\Psi=\Psi(\xi)\in[0,\alpha_0]$ with $|\Psi'(\xi)|\leq c$. This is called the flux-dependent accommodation coefficient introduced in \cite{Cer00}.\\
Let $f$ be a solution to
\begin{equation*}
\left\{
\begin{alignedat}{3}
\partial_tf+v\cdot\nabla_xf-\ddiv(A\nabla_vf)&=B\cdot\nabla_vf+F&&\qquad \mbox{in  $\mathcal{H}_{2,\infty} = I_2\times (B_8'\times (0,8))\times \bbR^n$}, \\
f&=\alpha\mathcal{R}f+(1-\alpha)\mathcal{N}f &&\qquad  \mbox{in $\gamma_-$}.
\end{alignedat} \right.
\end{equation*}
We assume 
\begin{equation}\label{ass.alpha.genmax}
\begin{aligned}
        \sup_{z_0\in \mathcal{H}_{2,\infty}}\Bigg(\langle v_0\rangle^{M}\|\mathcal{M}\|_{C^{\lambda_{0}+\eps}(\mathcal{H}_{r_0}(z_0))}&+\langle v_0\rangle ^{-p}
        \|\alpha\|_{C^{\lambda_{0}+\eps}(\mathcal{H}_{r_0}(z_0))}\\
        &+\|A\|_{C^\eps(\mathcal{H}_{r_0}(z_0))}+\langle v_0\rangle^{-2}\|B\|_{L^\infty(\mathcal{H}_{r_0}(z_0))}\Bigg)\leq c_B
    \end{aligned}
    \end{equation}
for some $c_B>0$, $p>0$ and a large constant $M\gg1$, where $r_0\coloneqq \frac{\langle v_0\rangle^{-\frac12}}{64}$ and $\lambda_{0} := \lambda_{\alpha_0}:=\lambda_{\alpha_0,1}$. The constant $\lambda_{\alpha_0,1}$ is determined in \autoref{lem.phialpha}. 

Note that when $\alpha$ is given by \eqref{defn.alpha.genmax}, one can prove  
for any small $\eps>0$,
\begin{align}\label{ass2.alpha.genmax}
    \sup_{z_0\in \mathcal{H}_{2,\infty}}\|\alpha\|_{C^{\lambda_{0}+\eps}(\mathcal{H}_{r_0}(z_0))}\leq c_B\langle v_0\rangle^{p}
\end{align}
for some $p>0$ and $c_B>0$. We will briefly explain how to obtain this. First, we have \eqref{Nfreg:main.quo} with $k=0$, which requires only regularity at $\gamma_+$, independent of any assumption on $\gamma_-$. Together with the assumption $|\Psi'(\xi)|\leq c$, one can deduce \eqref{ass2.alpha.genmax}.
 
Now, under the assumption \eqref{ass.alpha.genmax}, one can prove the global $C^{\lambda_0}$ estimates of the solution $f$ as in \autoref{thm.maxhalf}. We will briefly sketch the proof.
\begin{itemize}
    \item First, the fact that $\alpha$ depends only on $t,x'$ allows us to apply the boundary condition exactly as in the constant case. Moreover, the assumption $\alpha\leq \alpha_0$ ensures that the iteration argument used in the oscillation estimates in \eqref{holsmall:ineq1.osck} remains valid. Thus, without any change of the proof given in Section \ref{sec:4}, we can obtain the global $C^\beta$ estimates for some small $\beta>0$, as stated in \autoref{lem.bdd} and \autoref{lem.holsmall}.
    \item We can obtain expansion estimates of the form \eqref{expansion:maingoal} with $\phi_{\alpha,b}$ and $\lambda_{\alpha,b}$ replaced by $\phi_{\alpha(t_0,x_0'),1}$ and $\lambda_{0}$, respectively. In particular, the blow-up argument of \autoref{thm.expansion} carries over verbatim after replacing $\phi_{\alpha,b}$ and $\lambda_{\alpha,b}$ by $\phi_{\alpha(t_0,x_0'),1}$ and $\lambda_{{\alpha_0}}':= \lambda_{\alpha(t_0,x_0'),1}$, respectively. The only modification is that in \eqref{expansion:limit.eq}, $\alpha$ is replaced by $\alpha(t_0,x_0')$. Since $\phi_{\alpha(t_0,x_0'),1}\in C^{\lambda'_{{\alpha_0}}}\subset C^{\lambda_0}$ by the fact that $\lambda_{\alpha,1}$ is non-decreasing with respect to $\alpha$ and ${\alpha(t_0,x_0')}\leq \alpha_0$, the expansion estimates together with the standard covering argument as in \autoref{thm.opthol} yields the $C^{\lambda_0}$ estimates at the boundary. 
    \item Combining the boundary $C^{\lambda_0}$ estimates with the assumption that $\alpha\in C^{\lambda_0+\eps}$, a straight-forward modification of the proof of \autoref{thm.maxhalf} yields \eqref{maxhalf:maineq}, with $\lambda_\alpha$ replaced by $\lambda_0$. 
    \item Finally, assuming that $\alpha$ is smooth and decays sufficiently fast at infinity, we have the smoothness of $(1-\alpha)\mathcal{N}f$ after a few simple modifications of the proof of \eqref{lem.Nfreg}. Therefore, \autoref{lem.schhol.away-} implies that $f$ is smooth away from $\gamma_0$. 
\end{itemize}

The second model replaces the operator $\mathcal{N}f$ by 
\begin{align*}
    \mathcal{N}f(t,x',v)\coloneqq \int_{\bbR^n}\mathcal{K}(v,w)f(t,x',0,w)(w_n)_-\,dw,
\end{align*}
where 
\begin{align*}
    \mathcal{K}(v,w)\geq0\quad\text{and}\quad \int_{\bbR^n}\mathcal{K}(v,w)(v_n)_-\,dv=1.
\end{align*}
The Maxwell boundary condition corresponds to the special case $\mathcal{K}(v,w)=e^{-|v|^2/2}$. Assume 
 \begin{align*}
        \sup_{z_0\in \mathcal{H}_{2,\infty}}&\Bigg(\langle v_0\rangle^{M}\left[\|\mathcal{K}(v,\cdot)\|_{L^{\infty}(\mathcal{H}_{r_0}(z_0))}+\sup_{w\in\bbR^n}\|\mathcal{K}(\cdot,w)\|_{C^{\lambda_{\alpha}+\eps}(\mathcal{H}_{r_0}(z_0))}\right]\\
        &\quad+\|A\|_{C^\eps(\mathcal{H}_{r_0}(z_0))}+\langle v_0\rangle^{-2}\|B\|_{L^\infty(\mathcal{H}_{r_0}(z_0))}\Bigg)\leq c_B
    \end{align*} 
    for some $\eps>0$ and $c_B>0$ and large constant $M\gg1$, where $r_0\coloneqq \frac{\langle v_0\rangle^{-2}}{64}$. Then, one can prove the same estimate as in \eqref{maxhalf:maineq}. To see this, fix $R>0$ and $(t_0,x_0',0,v_0)\in\bbR\times \bbR^{n-1}\times \bbR\times \bbR^n$. We first observe that
    \begin{align*}
        &\int_{\bbR^{n-1}}\int_{w_n<R}\mathcal{K}(v,w)f(t,x',0,w)(w_n)_--\int_{\bbR^{n-1}}\int_{w_n<R}\mathcal{K}(v_0,w)f(t_0,x_0',0,w)(w_n)_-\\
        &=\int_{\bbR^{n-1}}\int_{w_n<R}\mathcal{K}(v,w)f(t,x',0,w)(w_n)_--\int_{\bbR^{n-1}}\int_{w_n<R}\mathcal{K}(v,w)f(t_0,x_0',0,w)(w_n)_-\\
        &\quad+\int_{\bbR^n}\mathcal{K}(v,w)f(t_0,x_0',0,w)(w_n)_--\int_{\bbR^{n-1}}\int_{w_n<R}\mathcal{K}(v_0,w)f(t_0,x_0',0,w)(w_n)_-.
    \end{align*}
    Using this decomposition together with the assumption on the kernel $\mathcal{K}$, one can deduce from the argument of  \eqref{Nfreg:mainineq} that the same estimate \eqref{Nfreg:mainineq} holds, with $a_r=\int_{\bbR^{n-1}}\int_{w_n<R}\mathcal{K}(v_0,w)f(t_0,x_0',0,w)(w_n)_-$. As the remaining parts are the same, we can obtain \eqref{maxhalf:maineq}. 
    In addition, if we assume that $\mathcal{K}$ is smooth and decays sufficiently fast at infinity in both $v,w$-variables, then a similar argument as in \eqref{Nfreg:main2} shows that $\mathcal{N}f$ is smooth away from $\gamma_0$. Combining this with \autoref{lem.schhol.away-}, we can conclude that the solution $f$ is smooth away from $\gamma_0$.
    
\subsection{General reflection boundary condition}
In this subsection, we will prove the optimal regularity for kinetic equations with general reflection boundary conditions \eqref{eq:bdry-gen} in general domains and prove \autoref{thm:main-general}. We recall the super-elastic reflection operator $\mathcal{R}_b$ in general domain:
\begin{align*}
    \mathcal{R}_bf(t,x,v)=f(t,x,v-(1+b)(n_x\cdot v)n_x).
\end{align*}

Since the boundary condition \eqref{eq:bdry-gen} does not contain a nonlocal boundary term such as $\mathcal{N}f$ as in the Maxwell boundary condition \eqref{eq:bdry}, we can derive the global regularity by directly applying the flattening argument and \autoref{thm.opthol}. First, we state the global regularity result when the domain is given by $\mathcal{H}_{2,\infty}=I_2\times (B_8'\times (0,2^3))\times \bbR^n$.

\begin{theorem}\label{thm.inel}
Let $\eps\in(0,1)$, $\alpha\in(0,1)$, and $b\in(0,1]$ with $\alpha/b^2\in(0,1]$. Let $f$ be a weak solution to 
     \begin{equation}\label{inel:eq}
\left\{
\begin{alignedat}{3}
\partial_tf+v\cdot\nabla_xf-\ddiv(A\nabla_vf)&=B\cdot\nabla_vf+F&&\qquad \mbox{in  $\mathcal{H}_{2,\infty}$}, \\
f&=\alpha \mathcal{R}_bf+(1-\alpha)g&&\qquad  \mbox{in $\gamma_-$}.
\end{alignedat} \right.
\end{equation}
Fix $p\geq2$ and write $r_1=\frac{\langle v_1\rangle^{-p}}{64}$ for any given $v_1\in\bbR^n$. Then we have the following:
\begin{itemize}
    \item Assume $g\in C^{\eps}(\gamma_-)$ for some $\eps>0$ and for some $M\geq1$
    \begin{align*}
         \sup_{z_1\in \mathcal{H}_{2,\infty}}\langle v_1\rangle^{-2}\|B\|_{L^\infty(\mathcal{H}_{r_1}(z_1))}\leq M.
    \end{align*}  Then there is a small constant $\beta=\beta(n,\Lambda,\alpha,b,\eps,p)$ such that
    \begin{align}\label{ineql:main1}
    \|f\|_{C^{\beta}(\mathcal{H}_{1,\infty})}\leq c\left(\|\langle \cdot\rangle^q f(t,x,\cdot)\|_{L^1(\mathcal{H}_{2,\infty})}+\|\langle \cdot\rangle^q F(t,x,\cdot)\|_{L^\infty(\mathcal{H}_{2,\infty})}+[g]_{C_q^{\beta}(\gamma_-)}\right),
\end{align}
where $c=c(n,\Lambda,\alpha,b,\eps,p,M)$,  $q=q(n,p)$ and 
\begin{align*}
    [g]_{C^{\beta}_q(\gamma_-)} \coloneqq \sup_{z_0 \in \gamma_-} (1 + |v_0|)^q [g]_{C^{\beta}(\cQ_1(z_0) \cap \gamma_-)}.
\end{align*}
\item Let $g\in C^{\lambda_{\alpha,b}+\eps}(\gamma_-)$ for some $\eps>0$, where $\lambda_{\alpha,b}$ is as in \eqref{phialpha:defn.lambdaalpha}, and assume that for some $M\geq1$
\begin{align*}
        \sup_{z_1\in \mathcal{H}_{2,\infty}}\left(\|A\|_{C^\eps({H}_{r_1}(z_1))}+\langle v_1\rangle^{-2}\|B\|_{L^\infty(\mathcal{H}_{r_1}(z_1))}\right)\leq M.
    \end{align*}

    Then we have
\begin{align}\label{ineql:main2}
    \|f\|_{C^{\lambda_{\alpha,b}}(\mathcal{H}_{1,\infty})}\leq c\left(\|\langle \cdot\rangle^q f(t,x,\cdot)\|_{L^1(\mathcal{H}_{2,\infty})}+\|\langle \cdot\rangle^q F(t,x,\cdot)\|_{L^\infty(\mathcal{H}_{2,\infty})}+[g]_{C_q^{\lambda_{\alpha,b}+\eps}(\gamma_-)}\right),
\end{align}
where $c=c(n,\Lambda,\eps,\alpha, b,p,M)$, $q=q(n,p)$.
\end{itemize}
\end{theorem}
\begin{proof}
First, we prove \eqref{ineql:main1}. To this end, suppose $\mathcal{Q}_{2r_0}(z_0)\cap \gamma_-=\emptyset$ or $v_{0,n}\leq 3r_0$, where $r_0=\frac{\langle v_0\rangle^{-p}}{64}$. By \autoref{lem.calpha} and \autoref{lem.holsmall} together with \autoref{lem.rep} and the fact that $r_0\|B\|_{L^\infty(\mathcal{H}_{4r_0}(z_0))}\leq \langle v_0\rangle^{-2}\|B\|_{L^\infty(\mathcal{H}_{4r_0}(z_0))}\leq c$, we have
    \begin{align}\label{inel:ineq1}
        \|f\|_{C^{\beta}(\mathcal{H}_{r_0}(z_0))}\leq c\langle v_0\rangle^{q}\left(\|f\|_{L^1(\mathcal{H}_{4r_0}(z_0))}+\|F\|_{L^\infty(\mathcal{H}_{4r_0}(z_0))}+[g]_{C^{{\beta}}(\gamma_-\cap\mathcal{Q}_{4r_0}(z_0))}\right),
    \end{align}
    where $c=c(n,\Lambda,\eps,b,\alpha,p,M)$, $q=q(n,p)$ and $\beta=\beta(n,\Lambda,\eps,b,\alpha)$. Note that \autoref{lem.rep} is used to replace the $L^\infty$ norm of $g$ by the $C^{\beta}$ norm of $g$.  Next, when  $\mathcal{Q}_{2r_0}(z_0)\cap \gamma_-\neq\emptyset$ and $v_{0,n}> 3r_0$, by \autoref{lem.calpha} and \autoref{lem.rep}, we have 
     \begin{align*}
        \|f\|_{C^{\beta}(\mathcal{H}_{r_0}(z_0))}\leq c\langle v_0\rangle^{q}\left(\|f\|_{L^1(\mathcal{H}_{2r_0}(z_0))}+\|F\|_{L^\infty(\mathcal{H}_{2r_0}(z_0))}+[\alpha\mathcal{R}_bf+(1-\alpha)g]_{C^{\beta}(\gamma_-\cap\mathcal{Q}_{2r_0}(z_0))}\right).
    \end{align*}
    By using the super-elastic boundary condition and \autoref{lem.geo}, we obtain 
    \begin{align*}
        [\mathcal{R}_bf]_{C^{\beta}(\gamma_-\cap\mathcal{Q}_{2r_0}(z_0))}\leq c[f]_{C^{\beta}(\gamma_+\cap\mathcal{Q}_{2r_0,b}(R_b(z_0)))},
    \end{align*}
    where $c=c(n,\beta)$. Note that for any  $z\in \mathcal{Q}_{2r_0,b}(R_b(z_0))\cap \gamma_+$, we have $v_n<-br_0$. Therefore, using the regularity near $\gamma_+$ as in \eqref{inel:ineq1} together with a standard covering argument, we have
    \begin{align*}
       [\mathcal{R}_bf]_{C^{\beta}(\gamma_-\cap\mathcal{Q}_{2r_0}(z_0))}\leq c\langle v_0\rangle^{q}\left(\|f\|_{L^1(\mathcal{H}_{4r_0}(R_b(z_0)))}+\|F\|_{L^\infty(\mathcal{H}_{4r_0}(R_b(z_0)))}+[g]_{C^{{\beta}}(\gamma_-\cap\mathcal{Q}_{4r_0}(R_b(z_0)))}\right),
    \end{align*}
    where $c=c(n,\Lambda,\eps,b,\alpha,p,M)$. Combining all the estimates leads to
    \begin{align*}
        \|f\|_{C^{\lambda_{b}}(\mathcal{H}_{r_0}(z_0))}&\leq c\langle v_0\rangle^{q}\left(\|f\|_{L^1(\mathcal{H}_{4r_0}(z_0))}+\|F\|_{L^\infty(\mathcal{H}_{4r_0}(z_0))}+\|f\|_{L^1(\mathcal{H}_{4R_0}(R_b(z_0)))}+\|F\|_{L^\infty(\mathcal{H}_{4r_0}(R_b(z_0)))}\right)\\
        &\quad+ c\langle v_0\rangle^{q}\left([g]_{C^{{\beta}}(\gamma_-\cap\mathcal{Q}_{4r_0}(z_0))}+[g]_{C^{{\beta}}(\gamma_-\cap\mathcal{Q}_{4r_0}(R_b(z_0)))}\right)
    \end{align*}
    for some $q=q(n,p)$ and $c=c(n,\Lambda,\eps,b,\alpha,p,M)$. This implies \eqref{ineql:main1} for some large $q=q(n,p)$. 
    
    For \eqref{ineql:main2}, we follow the same lines as in the proof of \eqref{ineql:main1} with \autoref{lem.calpha} and \autoref{lem.holsmall} replaced by \autoref{lem.schhol.away-} and \autoref{thm.opthol}, respectively.
    This completes the proof.
\end{proof}

Now we provide the following lemma, which was used in the proof of \autoref{thm.inel} and states that the $L^\infty$ norm of the boundary data can be estimated by its H\"older semi-norm. 

\begin{lemma}\label{lem.rep}
Let $f$ be a weak solution to 
\begin{equation*}
\left\{
\begin{alignedat}{3}
\partial_tf+v\cdot\nabla_xf-\ddiv(A\nabla_vf)&=B\cdot\nabla_vf+F&&\qquad \mbox{in  $\mathcal{H}_R(z_0)$}, \\
f&=\alpha \mathcal{R}_bf+(1-\alpha)g&&\qquad  \mbox{in $\gamma_-\cap\mathcal{Q}_R(z_0)$},
\end{alignedat} \right.
\end{equation*}
where $z_0\in\gamma_-\cup\gamma_0$. If $g\in C^\eps(\gamma_-\cap\mathcal{Q}_R(z_0))$ for some $\eps>0$, then 
\begin{align*}
    \|g\|_{L^\infty(\gamma_-\cap\mathcal{Q}_R(z_0))}\leq c(R^{-(4n+2)}\|f\|_{L^1(\mathcal{H}_R(z_0))}+\|R^2F\|_{L^\infty(\mathcal{H}_R(z_0))}+R^\eps[\alpha\mathcal{R}_bf+(1-\alpha)g]_{C^{\eps}(\gamma_-\cap\mathcal{Q}_R(z_0))})
\end{align*}
for some constant $c=c(n,\Lambda,\eps,R\|B\|_{L^\infty(\mathcal{H}_R(z_0))},\alpha,b)$. In addition, if $z_0\in\gamma_0$, then we have 
\begin{align*}
    \|g\|_{L^\infty(\gamma_-\cap\mathcal{Q}_R(z_0))}\leq c(R^{-(4n+2)}\|f\|_{L^1(\mathcal{H}_R(z_0))}+\|R^2F\|_{L^\infty(\mathcal{H}_R(z_0))}+R^\eps[g]_{C^{\eps}(\gamma_-\cap\mathcal{Q}_R(z_0))})
\end{align*}
for some constant $c=c(n,\Lambda,\eps,R\|B\|_{L^\infty(\mathcal{H}_R(z_0))},\alpha,b)$.
\end{lemma}
\begin{proof}
    The first claim follows from \cite[Lemma 7.6]{KiWe26}. For the second claim, note that when $z_0\in \gamma_0$, then we have H\"older estimates of $f$ at $\mathcal{H}_{R}(z_0)$ without involving $\alpha\mathcal{R}_bf$ (see \autoref{lem.holsmall}). Therefore, by following the same argument in \cite[Lemma 7.6]{KiWe26} together with \autoref{lem.bdd} and \autoref{lem.holsmall}, we get the second claim. This completes the proof.
\end{proof}
Using the flattening argument provided in \cite[Section 8]{KiWe26} and using \autoref{thm.inel}, we are ready to prove \autoref{thm:main-general}.
\begin{proof}[Proof of \autoref{thm:main-general}]
By the flattening argument given in \cite[Section 5]{RoWe25} (see also \cite[Section 8]{KiWe26} and in particular the proof of \cite[Theorem 8.1]{KiWe26}), we can consider the solution $f$ to \eqref{inel:eq} for some large $p=p(n,\eps)$. Note that by taking the large constant $p$, we get
\begin{align*}
    \sup_{z_1\in\mathcal{H}_{2,\infty}}\|A\|_{C^\eps({H}_{r_1}(z_1))}\leq c
\end{align*}
for some constant $c=c(n,\Lambda,\eps,A,\Omega)$. In addition, by following the proof of \cite[Lemma 2.16 and (i) in Lemma 2.17]{RoWe25}, we observe that the boundary condition is preserved under the flattening transformation. Moreover, after applying the flattening transformation, the drift $B$ satisfies 
\begin{align*}
    \langle v_1\rangle^{-2}\|B\|_{L^\infty(\mathcal{H}_{r_0}(z_0))}\leq c
\end{align*}
for some constant $c=c(n,\Omega,B)$. Therefore, using the estimate \eqref{ineql:main2} given in \autoref{thm.inel} and the standard covering argument from \cite[Theorem 5.4]{RoWe25}, we get the optimal global H\"older estimates. 

Next, when $A$ is merely measurable, we use \eqref{ineql:main1} instead of \eqref{ineql:main2} to derive the global H\"older estimates.

Furthermore, if the given data is smooth, then smoothness away from the grazing set follows immediately by \autoref{lem.schhol.away-}, as the solution is smooth near $\gamma_+$ independently of the boundary condition. This completes the proof.
\end{proof}

\bibliographystyle{alpha}
\bibliography{literature}

\end{document}